\documentclass[11pt,oneside]{amsart}
\pdfoutput=1

\usepackage{amsmath,ifthen, amsfonts, amssymb,
srcltx, amsopn, color, enumerate}
\usepackage[cmtip,arrow]{xy}
\usepackage{pb-diagram, pb-xy}
\usepackage{overpic}

\usepackage[T1]{fontenc}
\usepackage{yfonts}

\usepackage{graphicx}

\newcommand{\showcomments}{yes}
\renewcommand{\showcomments}{no}

\newsavebox{\commentbox}
\newenvironment{com}%
{\ifthenelse{\equal{\showcomments}{yes}}%
{\footnotemark
        \begin{lrbox}{\commentbox}
        \begin{minipage}[t]{1.25in}\raggedright\sffamily\tiny
        \footnotemark[\arabic{footnote}]}
{\begin{lrbox}{\commentbox}}}%
{\ifthenelse{\equal{\showcomments}{yes}}%
{\end{minipage}\end{lrbox}\marginpar{\usebox{\commentbox}}}
{\end{lrbox}}}

\newcommand{\visual}{\partial}

\newtheorem{thm}{Theorem}[section]
\newtheorem{lem}[thm]{Lemma}

\newtheorem{cor}[thm]{Corollary}

\newtheorem{prop}[thm]{Proposition}

\newtheorem{thmi}{Theorem}

\theoremstyle{definition}
\newtheorem{defn}[thm]{Definition}
\newtheorem{rem}[thm]{Remark}

\newtheorem{conv}[thm]{Convention}

\newtheorem{claim*}{Claim}

\newtheorem{cons}[thm]{Construction}

\DeclareMathOperator{\image}{im}

\DeclareMathOperator{\stabilizer}{Stab}
\DeclareMathOperator{\diam}{diam}

\newcommand{\neb}{\mathcal N}

\newcommand{\homology}{\ensuremath{{\sf{H}}}}

\newcommand{\field}[1]{\mathbb{#1}}
\newcommand{\integers}{\ensuremath{\field{Z}}}

\newcommand{\naturals}{\ensuremath{\field{N}}}
\newcommand{\reals}{\ensuremath{\field{R}}}

\newcommand{\leftl}{\overleftarrow{\mathcal L}}
\newcommand{\rightl}{\overrightarrow{\mathcal L}}
\newcommand{\closure}[1]{Cl\left({#1}\right)}

\newcommand{\interior} [1] {{\ensuremath \text{\rm Int}(#1) }}

\newcommand{\liftapp}[1]{\widehat{#1_{_{\succ}}}}

\makeatletter

\newcommand{\Rmnum}[1]{\mathbf{{\expandafter\@slowromancap\romannumeral #1@}}}

\makeatother

\newcommand{\lefta}{\overleftarrow A}
\newcommand{\righta}{\overrightarrow A}

\let\oldmarginpar\marginpar
\renewcommand\marginpar[1]{\-\oldmarginpar[\raggedleft\footnotesize #1]%
{\raggedright\footnotesize #1}}

\newcommand{\co}{\colon}
\DeclareMathOperator{\comapp}{\ensuremath{\mathbf A}}

\newcounter{enumitemp}

\newcommand{\leftw}{\overleftarrow{W}}
\newcommand{\rightw}{\overrightarrow{W}}

\newcommand{\dist}{\textup{\textsf{d}}}
\DeclareMathOperator{\periodicline}{\Lambda}

\begin{document}
\title[Cubulating free-by-cyclic groups]{Cubulating hyperbolic free-by-cyclic groups: the irreducible case}

\author[M.F.~Hagen]{Mark F. Hagen}
\address{Dept. of Math., University of Michigan, Ann Arbor, Michigan, USA }
\email{markfhagen@gmail.com}

\author[D.T.~Wise]{Daniel T. Wise}
\address{Dept. of Math. and Stat., McGill University, Montreal, Quebec, Canada}
\email{wise@math.mcgill.ca}

\date{\today}

\begin{abstract}
Let $V$ be a finite graph and let $\phi:V\rightarrow V$ be an irreducible train track map whose mapping torus has word-hyperbolic fundamental group $G$.  Then $G$ acts freely and cocompactly on a CAT(0) cube complex.
\end{abstract}

\maketitle

\setcounter{tocdepth}{2}
\tableofcontents

\section*{Introduction}\label{sec:introduction}
The goal of this paper is to prove the following theorem:

\begin{thmi}\label{thmi:irreducible}
Let $F$ be a finite-rank free group and let $\Phi:F\rightarrow F$ be an irreducible automorphism, and suppose that $G=F\rtimes_{\Phi}\integers$ is word-hyperbolic.  Then $G$ acts freely and cocompactly on a CAT(0) cube complex.
\end{thmi}

This result is a special case of Corollary~\ref{cor:irreducible_case}, which handles the more general case of a hyperbolic ascending HNN extension of a free group by an irreducible endomorphism.

Theorem~\ref{thmi:irreducible} provides a widely-studied class of hyperbolic groups for which Gromov's question (see~\cite{Gromov87}) of whether hyperbolic groups are CAT(0) has a positive answer, but goes further, since nonpositively-curved cube complexes enjoy numerous useful properties beyond having universal covers that admit a CAT(0) metric.  For example, combining Theorem~\ref{thmi:irreducible} with a result of~\cite{AgolVirtualHaken} shows that groups $G$ of the type described in Theorem~\ref{thmi:irreducible} are virtually special in the sense of~\cite{HaglundWiseSpecial} and therefore virtually embed in a right-angled Artin group.  This implies that $G$ has several nice structural features, including $\integers$-linearity.

A group $G\cong F\rtimes_{\Phi}\integers$ is word-hyperbolic exactly when $\Phi$ is atoroidal~\cite{BestvinaFeighnCombination,Brinkmann}, so that Theorem~\ref{thmi:irreducible} applies to all mapping tori of irreducible, atoroidal automorphisms of free groups.  More generally, ascending HNN extensions are hyperbolic precisely if they have no Baumslag-Solitar subgroups~\cite{Kapovich:2000}.

We actually prove the following more general statement:

\begin{thmi}\label{thmi:general}
Let $\phi:V\rightarrow V$ be a train track map of a finite graph $V$.  Suppose that $\phi$ is $\pi_1$-injective and that each edge of $V$ is expanding.  Moreover, suppose that the transition matrix $\mathfrak M$ of $\phi$ is irreducible and that the mapping torus $X$ of $\phi$ has word-hyperbolic fundamental group $G$.  Then $G$ acts freely and cocompactly on a CAT(0) cube complex.
\end{thmi}

Our CAT(0) cube complex arises by applying Sageev's construction~\cite{Sageev:cubes_95} to a family of walls in the universal cover $\widetilde X$ of $X$.  To ensure that the resulting action of $G$ on the dual cube complex is proper and cocompact, we show that there is a quasiconvex wall separating any two points in $\visual G$, thus verifying the cubulation criterion in~\cite{BergeronWise}.   As train track maps are central to the proof that there are many walls in this sense, our results build upon the work of Bestvina, Feighn, and Handel in~\cite{BestvinaHandel,BestvinaFeighnHandel_lamination}.

It appears likely that in the case where $\phi$ is $\pi_1$-surjective, the hypothesis that $\phi$ is irreducible can be removed, and we are currently working on developing the methodology in this paper to generalize Theorem~\ref{thmi:irreducible} to all hyperbolic mapping tori of free group automorphisms\footnote{We posted a tortuous generalization eight months after submitting this paper; see~\cite{HagenWise:general}.}.

Moreover, for the construction of immersed walls in $X$, hyperbolicity of $G$ plays a minor role.  It is therefore natural to wonder which free-by-cyclic groups admit free actions on CAT(0) cube complexes arising from immersed walls constructed essentially as in Section~\ref{sec:immersed_walls}.  If $\Phi$ is fully irreducible and $G$ is not hyperbolic, then $\Phi$ is represented by a homeomorphism of a surface, by~\cite[Thm.~4.1]{BestvinaHandel}.  Consequently, in this case $G$ acts freely on a locally finite, finite-dimensional CAT(0) cube complex~\cite{PrzytyckiWise:mixed}.  It is reasonable to conjecture that in general, if $G=F\rtimes_{\Phi}\integers$ is hyperbolic relative to virtually abelian subgroups, then $G$ acts freely on a locally finite, finite-dimensional CAT(0) cube complex.  The techniques in this paper are largely portable to that context.  However, one cannot expect to obtain cocompact cubulations for general free-by-cyclic groups.  Indeed, Gersten's group $\langle a,b,c,t\,\mid\,a^t=a,\,b^t=ba,\,c^t=ca^2\rangle$ is free-by-cyclic but does not act metrically properly by semisimple isometries on a CAT(0) space~\cite{Gersten:automorphism}, and hence Gersten's group cannot act freely on a locally finite, finite-dimensional CAT(0) cube complex.  Nevertheless, Gersten's group does act freely on an infinite-dimensional CAT(0) cube complex~\cite{WiseTubular}, so there is still much work to do in this direction.

\subsection*{Summary of the paper}
In Sections~\ref{sec:mapping_tori} and~\ref{sec:leaf_level_ladder}, we describe the mapping torus $X$ and introduce some features -- \emph{levels} and \emph{forward ladders} -- that play a role in the construction of immersed walls in $X$.

In Section~\ref{sec:immersed_walls}, we describe \emph{immersed walls} $W\rightarrow X$ when $\phi:V\rightarrow V$ is an arbitrary $\pi_1$-injective map sending vertices to vertices and edges to combinatorial paths, under the additional assumptions that no power of $\phi$ maps an edge to itself and $\pi_1X$ is hyperbolic.  The immersed wall $W$ is homeomorphic to a graph and has two parts, the \emph{nucleus} and the \emph{tunnels}, and is determined by a positive integer $L$ and a collection of sufficiently small intervals $d_i\subset V$, each contained in the interior of an edge.  The nucleus is obtained by removing from $V$ each \emph{primary bust} $d_i$, along with its $\phi^L$-preimage.  The tunnels are ``horizontal'' immersed trees joining endpoints of $d_i$ to endpoints of its preimage.  Let $\widetilde W\rightarrow\widetilde X$ be a lift of the universal cover of $W$ and let $\overline W\subset\widetilde X$ be its image.  Since $W\rightarrow X$ is not in general $\pi_1$-injective, $\widetilde W\rightarrow\overline W$ is not in general an isomorphism.  However, under suitable conditions described in Section~\ref{sec:quasiconvexity}, $\overline W$ is a wall in $\widetilde X$ whose stabilizer is a quasiconvex free subgroup of $G$.  The immersed walls in $X$ are analogous to the ``cross-cut surfaces'' introduced in~\cite{CooperLongReid}, and Dufour used these to cubulate hyperbolic mapping tori of self-homeomorphisms of surfaces~\cite{DufourThesis}.

Section~\ref{sec:cutting_geodesics} and~\ref{sec:getting_lol} are devoted to the proof of Theorem~\ref{thmi:general}.  We use a continuous surjection $\widetilde X\rightarrow\mathcal Y$ to an $\reals$-tree that arises in the case where $\phi$ is a train track representative of an irreducible automorphism (see~\cite{BestvinaFeighnHandel_lamination}).

\subsection*{Acknowledgements}\label{subsec:acknowledgements}
We thank the referees for their many helpful comments and corrections that improved this text.  This is based upon work supported by the National Science Foundation under Grant Number NSF 1045119 and by NSERC.

\section{Mapping tori}\label{sec:mapping_tori}
Let $V$ be a finite connected graph based at a vertex $v$, and let $\phi:V\rightarrow V$ be a continuous, basepoint-preserving map such that $\phi(w)$ is a vertex for each vertex $w$ of $V$, and such that $\phi(e)$ is a \emph{combinatorial path} in $V$ for each edge $e$ of $V$.  This means that there is a subdivision of $e$ such that vertices of the subdivision map to vertices and whose open edges map homeomorphically to open edges.  We also assume that $\phi$ is parametrized so that these homeomorphisms are linear.  Moreover, we require that the map $\Phi:F\rightarrow F$ induced by $\phi$ is injective, where $F\cong\pi_1V$ is a finite-rank free group.  We note that any injective $\Phi:F\rightarrow F$ is represented by such a map $\phi$.

The reader should have in mind the case where $\Phi$ is an irreducible automorphism of $F$ and $\phi$ is a train track map representing $\Phi$, in the sense of~\cite{BestvinaHandel}:

\begin{defn}[Train track map]\label{defn:train_track_map}
$\phi:V\rightarrow V$ is a \emph{train track map} if for all edges $e$ of $V$ and all $n\geq0$, the path $\phi^n(e)\rightarrow X$ is immersed.
\end{defn}

For an integer $L\geq 1$, let $X_L$ be obtained from $V\times[0,L]$ by identifying $(x,L)$ with $(\phi^L(x),0)$ for each $x\in V$, so that $X_L$ is the mapping torus of $\phi^L$, and let $X=X_1$.  See Figure~\ref{fig:pic_of_x}.  Let $G=\pi_1X$ and let $G_L=\pi_1X_L$ for each $L\geq 1$.  Note that if $\Phi$ is surjective then $G_L\cong F\rtimes_{\Phi^L}\integers$.

\begin{figure}[ht]
\includegraphics[width=0.5\textwidth]{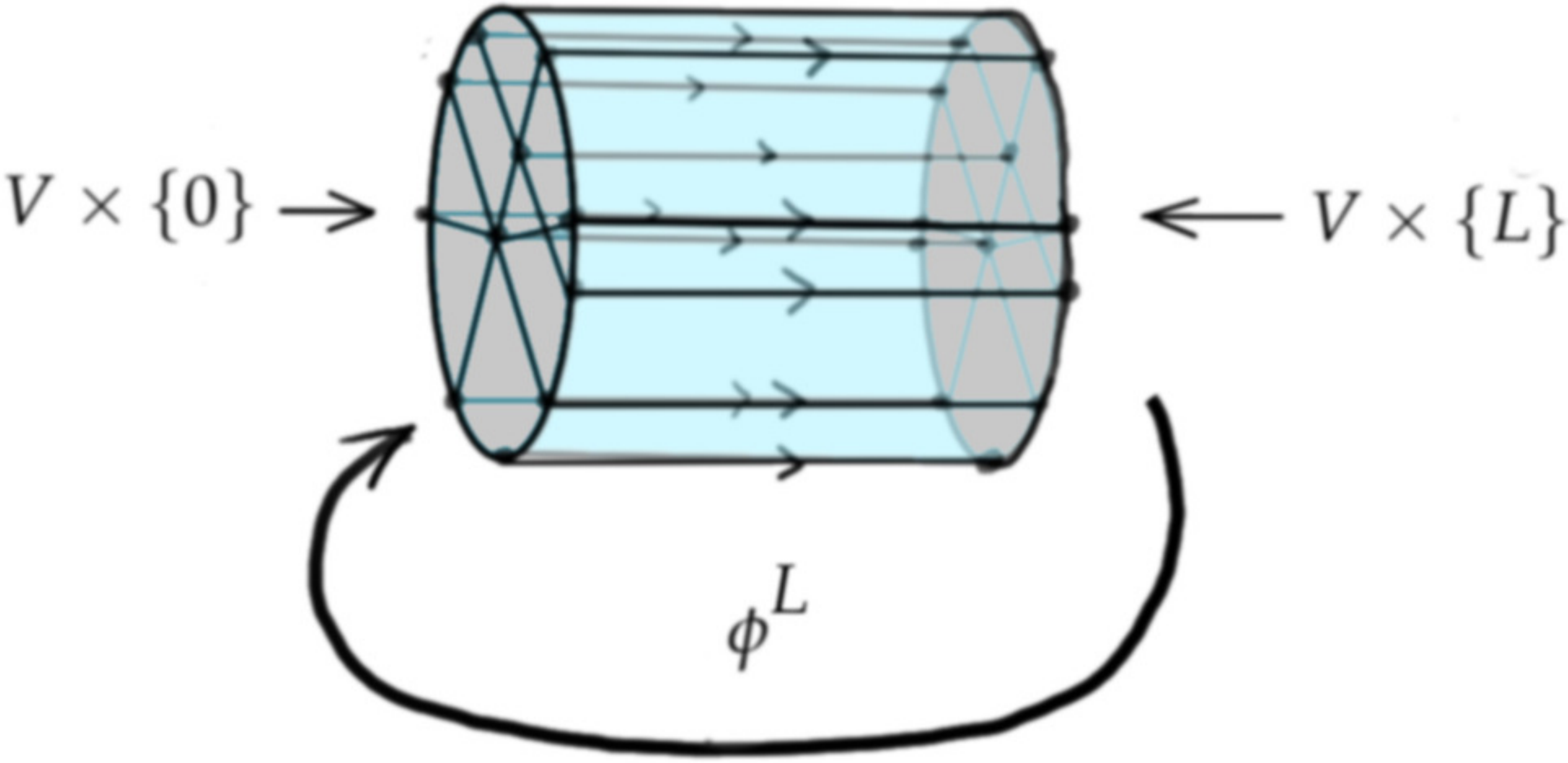}\\
\caption{The mapping torus $X_L$.}\label{fig:pic_of_x}
\end{figure}

We regard $V$ as a subspace of $X_L$, and we denote by $E$ the image in $X_L$ of $V\times\{\frac{1}{2}\}$; the space $E$ plays a role in Section~\ref{sec:immersed_walls}.

We now describe a cell structure on $X_L$.  Let $V\times[0,L]$ have the product cell structure: its vertices are $V^0\times\{0,L\}$, its \emph{vertical edges} are the edges of $V\times\{0,L\}$, and its \emph{horizontal edges} are of the form $\{w\}\times[0,L]$, where $w\in V^0$.  We direct each horizontal edge $\{w\}\times[0,L]$ from $\{w\}\times\{0\}$ to $\{w\}\times\{L\}$, and horizontal edges of $X_L$ are directed accordingly.  The $2$-cells of $X_L$ are images of the $2$-cells of $V\times[0,L]$, which have the form $e\times[0,L]$, where $e$ is an edge of $V$.

For each vertex $w\in V^0\subset X_L^0$, we let $t_w$ denote the unique horizontal edge outgoing from $w$.  When $L=1$, let $z\in G$ be the element represented by the loop $t_v$, where $v$ is the $\phi$-invariant basepoint of $V$.  Note that conjugation by $z$ induces the monomorphism $\Phi:F\rightarrow F$.  For each vertical edge $e$, joining vertices $a,b$, there is a $2$-cell $R_e$ with attaching map $t_b^{-1}e^{-1}t_a\phi^L(e)$, where $t_a,t_b$ are horizontal edges and $\phi^L(e)$ is a combinatorial path in $V$.  

Define a map $\varrho_L:X_L\rightarrow X$ as follows.  First, $\varrho_L$ restricts to the identity on $V$.  Each horizontal edge $t_w$ of $X_L$, joining $w$ to $\phi^L(w)\in V^0$, maps to the concatenation of $L$ horizontal edges of $X$ beginning at $w$.  This determines $\varrho_L:X_L^1\rightarrow X$.  This map extends to the 2-skeleton by mapping each 2-cell $R_e$ of $X_L$ to a disc diagram $D_e\rightarrow X$.  Specifically, for $0\leq i\leq L$, the $i^{th}$ component $P_i$ of the vertical 1-skeleton of $D_e$ is the path $\phi^i(e)$, and there are strips of 2-cells between consecutive vertical components.  The boundary path of $D_e$ consists of $P_0,P_L$, and the horizontal edges of $X$ joining the initial [terminal] point of $P_i$ to the initial [terminal] point of $P_{i+1}$ for $0\leq i\leq L-1$.  Such a diagram $D_e$ is a \emph{long 2-cell} and is depicted in Figure~\ref{fig:long_2_cell}.  Note that $D_e\rightarrow X$ is an immersion when $\phi$ is a train track map.  Otherwise, the paths $P_i\rightarrow X$ are not necessarily immersed and the map $D_e\rightarrow X$ need not be locally injective.  The map $G_L\rightarrow G$ induced by $\varrho_L$ embeds $G_L$ as an index-$L$ subgroup of $G$.

\begin{figure}[ht]
\includegraphics[width=0.4\textwidth]{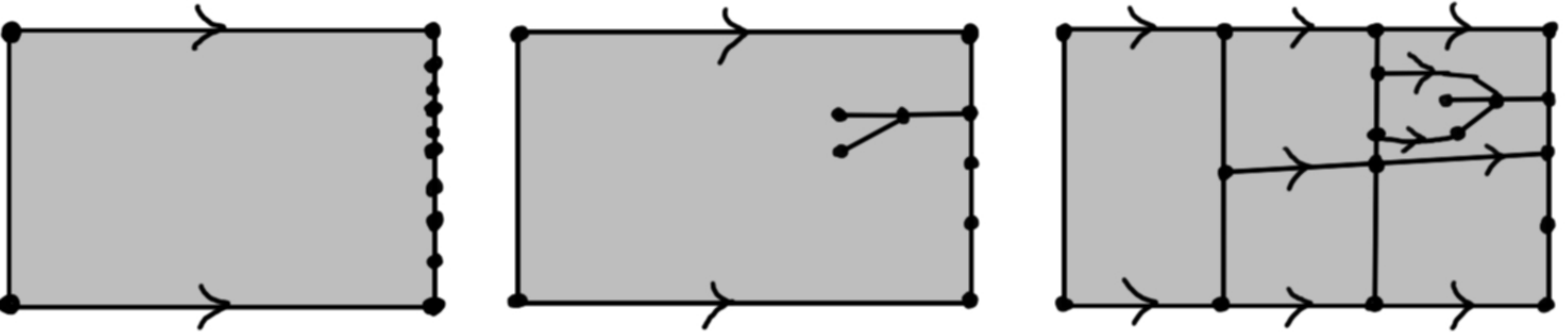}
\caption{A 2-cell of $X_L$ is shown at left, its image in $X_L$ is at the center, and its image in $X$, which is the image of a long 2-cell in $X$, is shown at right.  Here $L=3$.  Horizontal edges have arrows and all non-arrowed edges are vertical.}\label{fig:long_2_cell}
\end{figure}

The universal cover $\widetilde X_L\rightarrow X_L$ inherits a cell structure from $X_L$.  Let $\tilde v\in\widetilde X_L^0$ be a lift of the basepoint $v\in X_L^0$, and let $\widetilde V_0$ denote the smallest $F$-invariant subgraph containing the $\tilde v$ component of the preimage of $V$.  Let $\widetilde V_{nL}=z^{nL}\widetilde V_0$ for $n\in\integers$.  

There is a \emph{forward flow} map $\tilde\phi_L:\widetilde X_L\rightarrow\widetilde X_L$ defined as follows.  For each $p\in V\times\{0\}$, let $S_p$ be the path $\{p\}\times[0,L]\rightarrow X_L$.  The \emph{horizontal ray} $m_p\rightarrow X_L$ at $p$ is the concatenation $S_pS_{\phi^L(p)}S_{\phi^{2L}(p)}\cdots$.  For $\tilde p\in\widetilde V_{nL}$ mapping to $p$, let $\tilde m_{\tilde p}$ be the lift of $m_p$ at $\tilde p$.  For any $\tilde a\in\tilde m_{\tilde p}$, the point $\tilde\phi_L(\tilde a)$ is defined by translating $\tilde a$ a positive distance $L$ along $\tilde m_{\tilde p}$.  When $L=1$, we denote $\tilde\phi_L$ by $\tilde\phi$.

Let $\mathbf R_L$ denote the \emph{combinatorial line} with a vertex for each $nL\in L\integers$ and an edge for each $[nL,nL+L]$ and let $\mathbf S_L$ be a circle with a single vertex and a single edge of length $L$.  We define a map $q_L:\widetilde X_L\rightarrow\mathbf R_L$ as follows.  There is a map $\bar q_L:X_L\rightarrow \mathbf S_L$ induced by the projection $V\times[0,L]\rightarrow[0,L]$.  The map $\bar q_L$ lifts to the desired map $q_L$.  Note that $q_L$ sends vertical edges to vertices and horizontal edges and 2-cells to edges of $\mathbf R_L$.  We let $\mathbf R=\mathbf R_1$, $\mathbf S=\mathbf S_1$, and $q=q_1$.

Let $\widetilde E_{nL}=q_L^{-1}(nL+\frac{1}{2})$.  Each horizontal edge $t_w\cong\{w\}\times[0,L]\subset\widetilde X_L$ intersects $\widetilde E_{nL}$ at the point $\{w\}\times\{\frac{1}{2}\}$ for a unique $n\in\integers$.  

\subsection{Metrics and subdivisions}\label{subsec:metrics_and_subdivisions}
For each edge $e$ of $X$, let $|e|$ be a positive real number, with $|t_w|=1$ for each horizontal edge $t_w$.  The assignment $e\mapsto |e|$ is a \emph{weighting} of $X^1$, and pulls back to a $G$-equivariant weighting of $\widetilde X^1$, with all horizontal edges having unit weight.  Regarding $e$ as a copy of $[0,1]$, the subinterval $d\cong[a,b]\subset e$ has weight $|d|=(b-a)|e|$.  Consider an embedded path $P\rightarrow\widetilde X^1$ (not necessarily combinatorial).  The \emph{length} $|P|$ of $P$ is the sum of the weights of $P\cap e$, where $e$ varies over all edges.  This yields a geodesic metric $\dist$ on $\widetilde X^1$ such that $(\widetilde X^1,\dist)$ is quasi-isometric to $\widetilde X^1$ with the usual combinatorial path-metric in which edges have unit length.

For each $L\geq 1$, let ${\widetilde X^\bullet}_L$ be the subdivision of $\widetilde X_L$ such that the lift $\tilde\varrho_L:\widetilde X_L\rightarrow\widetilde X$ of $\varrho_L$ sends open cells homeomorphically to open cells.  The resulting map ${\widetilde X^\bullet}_L\rightarrow\widetilde X$ is an isomorphism on subspaces $\widetilde V_{nL}$ and sends 2-cells to long 2-cells.  Note that 2-cells of $\widetilde X_L$ do not immerse in $\widetilde X$ unless $\phi$ is a train track map.  Pulling back weights of edges in $\widetilde X$ to ${\widetilde X^\bullet}_L$ yields a metric $\dist_L$ on $({\widetilde X^\bullet}_L)^1$ with respect to which $({\widetilde X^\bullet}_L)^1\rightarrow\widetilde X^1$ is a distance-nonincreasing quasi-isometry.  We shall work mainly in $\widetilde X$, except in Section~\ref{sec:cutting_geodesics}, where it is essential to consider ${\widetilde X^\bullet}_L$.

Beginning in Section~\ref{sec:immersed_walls}, we shall assume that $G$ is word-hyperbolic, so that there exists $\delta\geq 0$ such that $(\widetilde X^1,\dist)$ is $\delta$-hyperbolic.

\section{Forward ladders and levels}\label{sec:leaf_level_ladder}
In this section, we define various subspaces of $\widetilde X$ needed in the construction and analysis of quasiconvex walls in $\widetilde X$ and ${\widetilde X^\bullet}_L$.

\begin{defn}[Midsegment]\label{defn:midsegment}
Let $R_e\rightarrow\widetilde X$ be a 2-cell with boundary path $t_b^{-1}e^{-1}t_a\tilde\phi(e)$, where $e$ is a vertical edge joining vertices $a,b$.  Regarding $R_e$ as a Euclidean trapezoid with parallel sides of length $|e|$ and $|\tilde\phi(e)|$, the \emph{midsegment in $R_e$} determined by $x\in e$ is the line segment joining $x$ to $\tilde\phi(x)$.  The \emph{midsegment} in $\widetilde X$ determined by $x$ is the image of the midsegment in $R_e$ determined by $x$ under the map $R_e\rightarrow\widetilde X$, and is denoted $m_x$.  Midsegments are directed so that $x$ is initial and $\tilde\phi(x)$ is terminal.  The midsegment $m_x$ is \emph{singular} if $\tilde\phi(x)\in\widetilde X^0$ and \emph{regular} otherwise. In general, $R_e\rightarrow\widetilde X$ is not an embedding, and there may be distinct $x,y\in e$ with the property that the terminal points of $m_x$ and $m_y$ coincide.  Note, however, that the intersection of two midsegments contains at most one point.  See Figure~\ref{fig:midsegment}.

\begin{figure}[ht]
\includegraphics[width=0.3\textwidth]{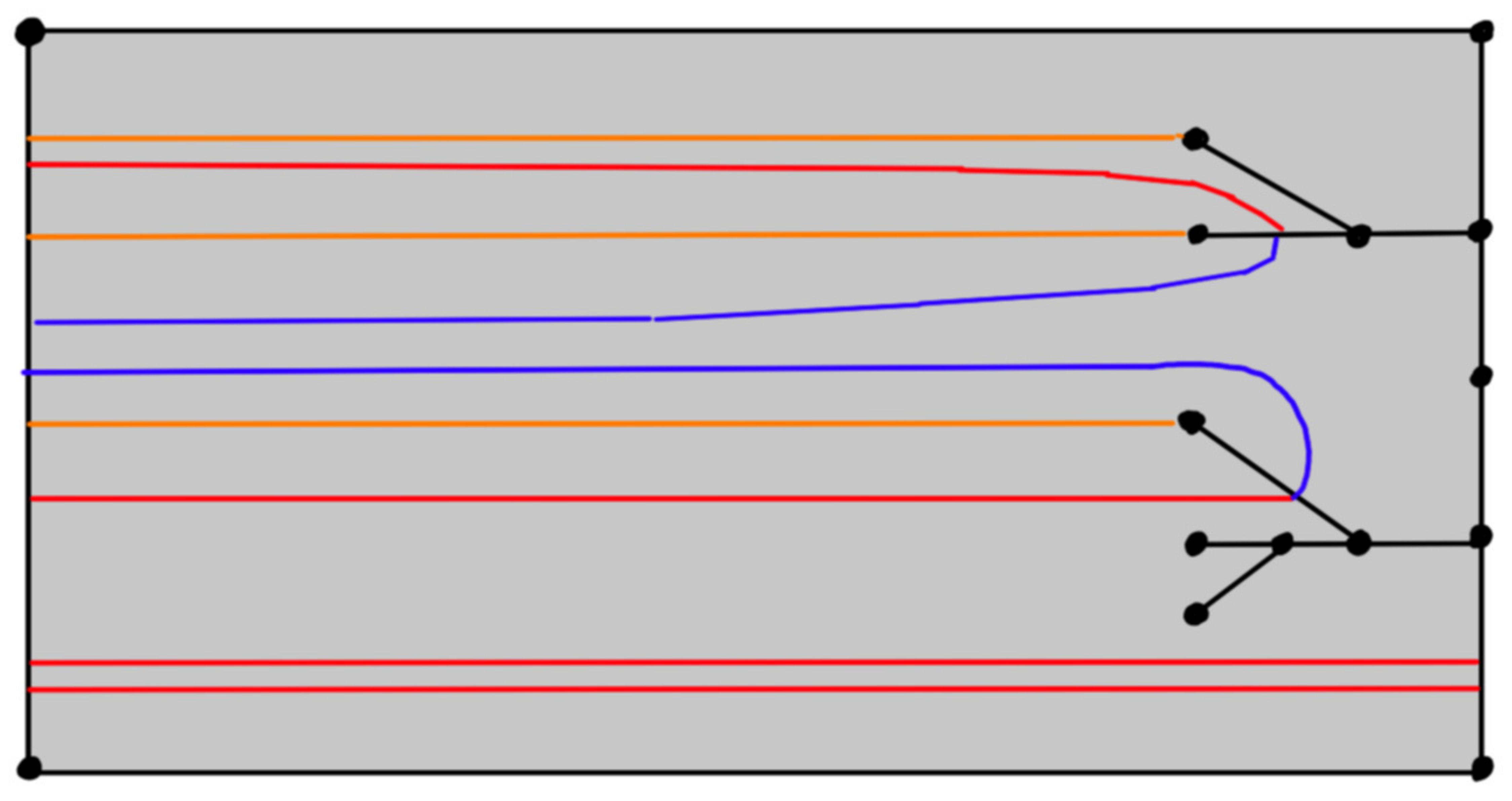}\\
\caption{Some midsegments in the image of a 2-cell in $\widetilde X$.}\label{fig:midsegment}
\end{figure}

\end{defn}

\begin{defn}[Forward path, forward ladder]\label{defn:forward_ladder}
Let $x\in\widetilde V_n$ for some $n\in\integers$ and let $M\in\integers$.  The \emph{forward path} $\sigma_M(x)$ of length $M$ determined by $x$ is the embedded path that is the concatenation of midsegments starting at $x$ and ending at $\tilde\phi^M(x)$.  In other words, $\sigma_M(x)$ is isomorphic to the combinatorial interval $[0,M]$, whose vertices are the points $\tilde\phi^i(x),\,0\leq i\leq M$ and whose edges are the midsegments joining $\tilde\phi^i(x)$ to $\tilde\phi^{i+1}(x)$.  Any path $\sigma$ of this form is a \emph{forward path}.  Note that $\sigma$ is a directed path with respect to the directions of midsegments in the sense that each internal point in which $\sigma$ intersects the vertical 1-skeleton of $\widetilde X$ has exactly one incoming and one outgoing midsegment.  The \emph{forward ladder} $N(\sigma)$ associated to $\sigma$ is the smallest subcomplex of $\widetilde X$ containing $\sigma$.  The $1$-skeleton $N(\sigma)^1$ plays an important role in many arguments.  See Figure~\ref{fig:ladder}.
\end{defn}

\begin{figure}[ht]
\includegraphics[width=0.3\textwidth]{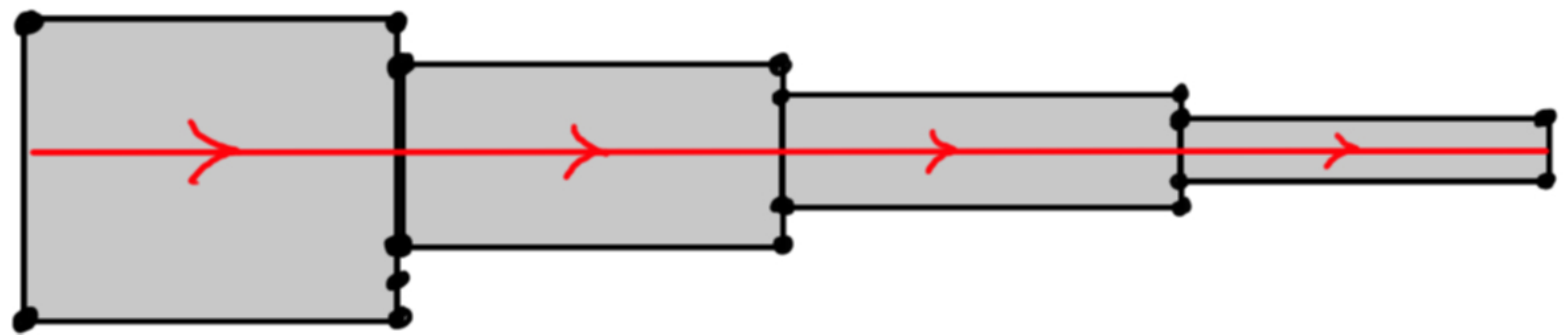}
\caption{A forward ladder.  The forward path is labelled with arrows.}\label{fig:ladder}
\end{figure}

A subgraph $Y$ of $\widetilde X^1$ is \emph{$\lambda$-quasiconvex} if every geodesic of $\widetilde X^1$ starting and ending on $Y$ lies in $\neb_\lambda(Y)$.  We use the notation $\neb_r(Y)$ to denote the closed $r$-neighborhood of $Y$.

\begin{prop}[Quasiconvexity of forward ladders]\label{prop:forward_ladder_quasiconvex}
There exist constants $\lambda_1\geq 1,\lambda_2\geq 0$ such that for each forward path $\sigma\rightarrow\widetilde X$, the inclusion $N(\sigma)^1\hookrightarrow\widetilde X^1$ is a $(\lambda_1,\lambda_2)$-quasi-isometric embedding.  Hence, if $\widetilde X^1$ is $\delta$-hyperbolic, there exists $\lambda\geq 0$ such that each $N(\sigma)^1$ is $\lambda$-quasiconvex.
\end{prop}

\begin{proof}
Let $\sigma$ join $x$ to $\tilde\phi^M(x)$, so that $\sigma=m_xm_{\tilde\phi(x)}\cdots m_{\tilde\phi^{M-1}(x)}$.  Let $R_i$ be the 2-cell containing $m_{\tilde\phi^i(x)}$.  Then a geodesic $P$ of $N(\sigma)^1$ joining $x$ to $\tilde\phi^M(x)$ has the form $P=Q_0t_0Q_1t_1Q_2\cdots t_{M-1}Q_M$, where each $t_i$ is a horizontal edge in $R_i$ and each $|Q_i|\leq\max_{e}\{|\phi(e)|\}$.  Since each $m_i$ is a midsegment, $R_i\neq R_j$ for $i\neq j$, whence $q(P)$ is a combinatorial interval of length $M$, and the preimage in $N(\sigma)^1$ of each point in $q(P)$ is uniformly bounded.  Hence $P$ is a uniform quasigeodesic in $\widetilde X^1$.
\end{proof}

In the case that $\widetilde X^1$ is $\delta$-hyperbolic, we denote by $\lambda$ the resulting quasiconvexity constant of the 1-skeleton of a forward ladder.

\begin{defn}[Level]\label{defn:level}
Let $x\in\widetilde V_n$ and let $L\geq 0$.  Note that the preimage $(\tilde\phi_L)^{-1}(x)$ is a finite set $\{x_i\}$ in $\widetilde V_{n-L}$.  Let $\sigma_L(x_i)$ be the forward path beginning at $x_i$ and ending at $x$.  The \emph{level} $T^o_L(x)$ is the subspace $\cup_i\sigma_L(x_i)$.  The point $x$ is the \emph{root} of $T_L^o(x)$ and $L$ is the \emph{length}.  The \emph{carrier} $N(T^o_L(x))$ is the smallest subcomplex of $\widetilde X$ containing $T^o_L(x)$.  Note that $N(T^o_L(x))=\cup_iN(\sigma_i)$, where $\sigma_i$ varies over the finitely many maximal forward paths in $T^o_L(x)$.  Note that each level has a natural directed graph structure in which edges are midsegments.
\end{defn}

\begin{prop}[Properties of levels]\label{prop:properties_of_levels}
Let $T^o_L(x)$ be a level.  Then:
\begin{enumerate}
 \item $T^o_L(x)$ is a directed tree in which each vertex has at most one outgoing edge.
 \item If $x\not\in\widetilde X^0$, then there exists a topological embedding $T^o_L(x)\times[-1,1]\rightarrow\widetilde X$ such that $T^o_L(x)\times\{0\}$ maps isomorphically to $T^o_L(x)$.
 \item If $L'\geq L$, then $T^o_L(x)\subseteq T^o_{L'}(x)$.
\end{enumerate}
\end{prop}


\begin{proof}
$T^o_L(x)$ is connected since it is the union of a collection of paths, each of which terminates at $x$.  Each vertex of $T^o_L(x)$ has at most one outgoing edge.  Hence any cycle in $T^o_L(x)$ is directed.  The map $q:T^o_L(x)\rightarrow\mathbf R$ thus shows that there are no cycles in $T^o_L(x)$.  This establishes assertion~$(1)$.

Let $x_i\in(\tilde\phi_L)^{-1}(x)$ and let $\sigma_i\subset T^o_L(x)$ be the forward path joining $x_i$ to $x$.  Then $\sigma_i$ is disjoint from $\widetilde X^0$, since $T^o_L(x)$ is regular.  Hence there exists $\epsilon_i>0$ such that $N(\sigma_i)$ contains an embedded copy of $\sigma_i\times[-\epsilon_i,\epsilon_i]$ with $\sigma_i\times\{0\}=\sigma_i$, which we denote by $F_i$.  Let $\epsilon=\min_i\epsilon_i$.  For each $i$, let $F'_i\subset F_i$ be $\sigma_i\times[-\epsilon,\epsilon]\subseteq\sigma_i\times[-\epsilon_i,\epsilon_i]$, and let $F=\cup_iF'_i$.  Since $\sigma_i\cap\sigma_j$ is a forward path for all $i,j$, the subspace $F\cong T^o_L(x)\times[-\epsilon,\epsilon]$.  See Figure~\ref{fig:product_neighborhoods}.

\begin{figure}[ht]
\includegraphics[width=0.6\textwidth]{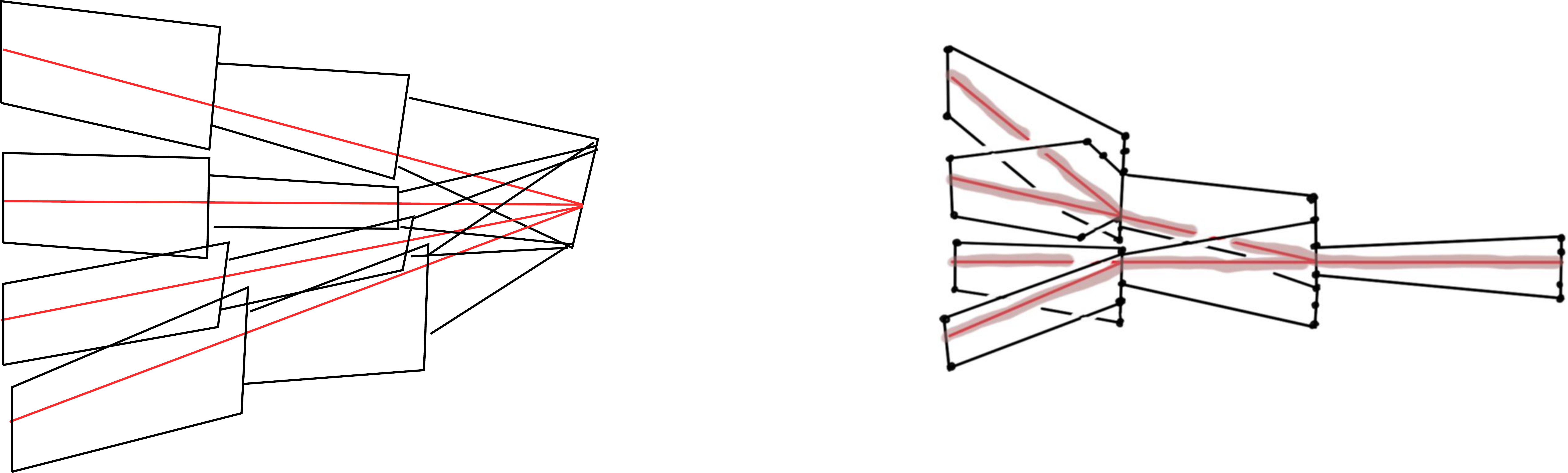}
\caption{The product neighborhood of a regular level in $\widetilde X$ is shown at right; the corresponding level in $\widetilde X_L$ appears at left.  In general, the product neighborhood may contain several subintervals of each vertical edge since $\phi$ is not in general an immersion on edges.}\label{fig:product_neighborhoods}
\end{figure}

Assertion~$(3)$ follows from the fact that $\tilde\phi^{L'}:\widetilde X\rightarrow\widetilde X$ factors as $\widetilde X\stackrel{\tilde\phi^{L}}{\longrightarrow}\widetilde X\stackrel{\tilde\phi^{L'-L}}{\longrightarrow}\widetilde X$.
\end{proof}

For each $L\geq 1$, forward paths and levels are defined in precisely the same way in $\widetilde X_L$.  A level of $\widetilde X_L$ is subdivided when we formed ${\widetilde X^\bullet}_L$ in Section~\ref{subsec:metrics_and_subdivisions}.  Accordingly, each length-$L$ level in ${\widetilde X^\bullet}_L$ is isomorphic to a star whose edges are subdivided into length-$L$ paths.  The map $\tilde\varrho_L:{\widetilde X^\bullet}_L\rightarrow\widetilde X$ sends each length-$n$ level of $\widetilde X_L$, each of whose maximal forward paths contains $nL$ midsegments of ${\widetilde X^\bullet}_L$, to a length-$nL$ level in $\widetilde X$.  Thus $\varrho_L$ maps subdivided stars to rooted trees, as shown in Figure~\ref{fig:product_neighborhoods}.

The image in $X_L$ of a level from ${\widetilde X^\bullet}_L$ is also referred to as a level; it will be clear from the context whether we are working in the base space or the universal cover.

The following observation about forward ladders is required in several places in Section~\ref{sec:quasiconvexity} and Section~\ref{sec:cutting_geodesics}.

\begin{lem}\label{lem:vert_horiz_bound}
Let $\sigma$ be a forward path.  Then for each $R\geq 0$, there exists $\Theta_R\geq 0$, independent of $\sigma$ and $n$, such that $\diam(\neb_R(N(\sigma)^1)\cap\neb_R(\widetilde V_n))\leq\Theta_R$ for all $n\in\integers$.
\end{lem}

\begin{proof}
This follows from the fact that $q(\widetilde V_n)=n$, while the image of $q|_{N(\sigma)}$ is an interval, each of whose points has uniformly bounded preimage in $N(\sigma)^1$.
\end{proof}

\section{Immersed walls, walls, and approximations}\label{sec:immersed_walls}
In this section, we will describe immersed walls $W\rightarrow X$, which are determined by two parameters.  The first parameter is a collection $\{d_i\}$ of subintervals of edges in $V$, called \emph{primary busts}.  The second parameter is an integer $L\geq 1$ called the \emph{tunnel length}.  The graph $W$ consists of $V-\cup_id_i-\cup_i(\phi^L)^{-1}(d_i)$ together with a collection of rooted trees called \emph{tunnels}, and is immersed in $X_L$.  We shall show that when $L$ is sufficiently large, $W\rightarrow X$ corresponds to a quasiconvex codimension-1 subgroup of $G$.

\subsection{Primary busts}\label{subsec:bust_choose}
Let $\{e_1,\ldots,e_k\}$ be edges of $V\subset X$.  For each $i$, let $e'_i$ be the image of $e_i$ under the isomorphism $V\rightarrow E$ given by $(x,0)\mapsto(x,\frac{1}{2})$.  The subspaces $e'_i$, regarded as edges of $E$, are \emph{primary busted edges}.  We will choose closed nontrivial intervals $d'_i\subset\interior{e'_i}$, whose distinct endpoints we denote by $p_i^{\pm}$.  The corresponding subinterval of $e_i$ is denoted $d_i$, and its endpoints $q_i^{\pm}$ correspond to $p_i^{\pm}$.  Let $E^{\flat}=E-\cup_{i=1}^k\interior{d'_i}$ and let $V^{\flat}$ denote its preimage in $V$ under the above isomorphism $V\rightarrow E$.  The subspace $E^{\flat}$ is the \emph{primary busted space}, and each $d_i$ (or $d'_i$) is a \emph{primary bust}; $\interior{d_i}$ (or $\interior{d'_i}$) is an \emph{open primary bust}.

Let $C$ be a component of $V^{\flat}$ and let $\widetilde C$ be a lift of its universal cover to some $\widetilde V_n$.  Since $C\hookrightarrow V\hookrightarrow X$ is $\pi_1$-injective, $\widetilde C$ embeds in $\widetilde V_n$.  Its parallel copy $\widetilde C'\subset\widetilde E_n$ is a \emph{primary nucleus}, and likewise, each component of $E^{\flat}$ is a \emph{primary nucleus} in $X$.

\begin{rem}[Quasiconvexity of $\widetilde C$ under various conditions]\label{rem:nucleus_subgroup_conditions}
In our applications, we will require $\widetilde C$ to be quasiconvex in $\widetilde X^1$.  This is achievable in several ways.  Clearly, if $\{e_i\}$ contains enough edges that $E^{\flat}$ is a forest, then the subspaces $\widetilde C\subset\widetilde X^1$ are finite trees and therefore quasiconvex.

Quasiconvexity of $\widetilde C$ occurs under other circumstances.  For example, suppose that $\Phi:F\rightarrow F$ is an automorphism and $\phi$ is a train track map that is \emph{aperiodic} in the sense of~\cite{MitraQuasiconvex}, i.e. $\phi^n(e)$ traverses $f$ for all edges $e,f$ and all sufficiently large $n$.  Then, provided $\{e_i\}$ contains at least one edge corresponding to a nontrivial splitting of $F$, the following theorem of Mitra (see~\cite[Prop.~3.4]{MitraQuasiconvex}), which is an analog of a result of Scott and Swarup~\cite{Scott_Swarup}, ensures that each $\widetilde C$ is quasiconvex:

\begin{thm}\label{thm:mitra}
Let $\Phi:F\rightarrow F$ be an aperiodic automorphism of the finite-rank free group $F$.  If $H\leq F$ is a finitely generated, infinite-index subgroup, then $H$ is quasiconvex in $F\rtimes_{\Phi}\integers$.
\end{thm}
\end{rem}

\subsection{Constructing immersed walls}\label{subsec:immersed_wall_construction}
We now assume that $\widetilde X^1$ is $\delta$-hyperbolic.  Let $L\geq 1$ be an integer, called the \emph{tunnel length}.  For any set $\{d_i\}$ of nontrivial primary busts, the spaces $E^{\flat}$ and $V^{\flat}$ embed in $X_L$ by maps factoring through $E\hookrightarrow X_L$ and $V\hookrightarrow X_L$ respectively.  For each $i$, let $\{d_{ij}\}_{j}$ denote the finite set of components of $(\phi^L)^{-1}(d_i)$.  For each $i,j$, let $d'_{ij}$ be the parallel copy of $d_{ij}$ in $E$.  Each $d_{ij}$ or $d_{ij}'$ is a \emph{secondary bust}.  In order to choose busts, we will assume that each edge $e$ of $V$ is \emph{expanding} in the sense that $\phi^k(e)\neq e$ for all $k>0$.  This assumption is justified by the following lemma (see also~\cite{BestvinaHandel}).

\begin{defn}
We say $x\in V$ is \emph{periodic} if $\phi^n(x)=x$ for some $n\geq 1$.  A point $x\in V$ has \emph{period~$m$} if $\phi^m(x)=x$ and $\phi^k(x)\neq x$ for $0<k<m$.  We then refer to $x$ as being \emph{$m$-periodic}.

A forward path $\sigma\rightarrow\widetilde X$ is \emph{periodic} if it is a subpath of a bi-infinite forward path whose stabilizer in $G$ is nontrivial.  Note that this holds exactly when each point of $\widetilde X^1\cap\sigma$ projects to a periodic point of $V$.
\end{defn}

Recall that the map $\phi:V\rightarrow V$ is \emph{irreducible} if for all edges $e,f$, there exists $n\geq0$ such that $\phi^n(e)$ traverses $f$.

\begin{lem}\label{lem:edges_expanding}
Let $F\rtimes_{\Phi}\integers$ be hyperbolic.  Then $\Phi:F\rightarrow F$ can be represented by a map $\phi:V\rightarrow V$ with respect to which each edge of $V$ is expanding and no edge is mapped to a point.  Moreover, if $\Phi$ has an irreducible train track representative, then $\phi:V\rightarrow V$ can be chosen to be an irreducible train track map with respect to which each edge is expanding.
\end{lem}

\begin{proof}
We begin with a representative $\phi:V\rightarrow V$, which we will adjust by contracting subtrees of $V$.  Let $U\subset V$ be the union of all vertices and all closed edges $e$ such that $|\phi^k(e)|$ is bounded as $k\rightarrow\infty$.  First, note that $\phi(U)\subseteq U$.  Second, each component of $U$ is contractible, since otherwise either $\phi$ is not $\pi_1$-injective or $X$ would contain an immersed torus, contradicting hyperbolicity.  We now collapse the $\phi$-invariant forest $U$ as in~\cite[Page~7]{BestvinaHandel}, resulting in a graph $\overline V$ and a map $\bar\phi:\overline V\rightarrow\overline V$ representing $\Phi$ (by reparametrizing, we can assume that the restriction of $\bar\phi$ to each edge is a combinatorial path).  Note that either $U$ contained no edges (so all edges were expanding and did not map to points), or $\overline V$ has strictly fewer edges than $V$.  We repeat the above procedure finitely many times to obtain a graph $\overline V$ and a map $\bar\phi:\overline V\rightarrow\overline V$ such that edges map to nontrivial paths and all edges are expanding.

The collapse of $U$ preserves the property of being a train track map.  Indeed, let $\bar e$ be an edge of $\overline V=V/U$ that is the image of an edge $e$ of $V$.  Let $n>0$, and consider the restriction of $\bar\phi^n$ to $\bar e$.  The path $\bar\phi^n(\bar e)$ is obtained from the immersed path $\phi^n(e)$ by collapsing each edge that maps to $U$.  Let $\bar u,\bar v$ be consecutive edges of $\bar\phi^n(\bar e)$ that fold.  Then there is a subpath $u^{-1}fv\subset\phi^n(e)$, where $u\mapsto\bar u,v\mapsto\bar v$ and $f$ is an immersed path in $U$.  Observe that $f$ is a closed path since $u,v$ have the same initial point.  This contradicts the fact that $U$ is a forest.

Finally, the property of irreducibility is preserved by collapsing invariant forests.  Indeed, let $\bar e,\bar f$ be edges of $\overline V$ that are images of edges $e,f$ of $V$.  Then by irreducibility of $\phi$, there exists $m>0$ such that $\phi^m(e)$ passes through $f$, and hence $\bar\phi^m(\bar e)$ passes through $\bar f$.
\end{proof}

A  point $y\in V$ is \emph{singular} if $\phi^k(y)\in V^0$ for some $k$.

\begin{lem}\label{lem:choosing_busts_1}
Let $L\in\naturals$.  Let $\{e_i\}_{i=1}^k$ be a set of expanding edges in $V$, let $x_i\in\interior{e_i}$ for each $i$, and let $\epsilon>0$.  Then there exists a collection $\{d_i\}_{i=1}^k$ of closed subintervals, with each $d_i\subset\interior{e_i}$, such that:
\begin{enumerate}
 \item \label{item:disjoint_from_preimage}$\cup_id_i$ is disjoint from $\cup_{ij}d_{ij}$.
 \item \label{item:neighborhood}$\cup_jd_{ij}$ lies in the $\epsilon$-neighborhood of $(\phi^L)^{-1}(x_i)$ for each $i$.
 \item \label{item:regular_endpoints}The endpoints $p_i^{\pm}$ of $d_i$ are nonsingular.
 \item \label{item:anywhere}If $x_i$ is nonsingular and $\phi^L(x_i)\neq x_i$ then we can choose $d_i$ 
 such that $x_i\in\{p_i^{\pm}\}$ is an endpoint of $d_i$.
 \item \label{item:bustsembed}$\phi^L$ restricts to an embedding on $d_i$, for each $i$.
 \item \label{item:disjoint}Suppose that $\phi^L(x_i)\neq\phi^L(x_j)$ for all $i\neq j$.  Then $\phi^L(d_i)\cap\phi^L(d_j)=\emptyset$.
\end{enumerate}

\end{lem}

\begin{proof}
We first establish the finiteness of the set $\mathcal S$ consisting of points $s\in e_i$ such that $\phi^L(s)=s$.  Each component $b$ of $e_i\cap(\phi^L)^{-1}(e_i)$ is the concatenation of one or more subintervals of $e_i$, each of which maps homeomorphically to $e_i$.  Since $e_i$ is expanding, Brouwer's fixed point theorem implies that each such subinterval contains a unique point $s$ with $\phi^L(s)=s$.  As there are finitely many such $b$, we conclude that $\mathcal S$ is finite.

Let $z_i \in \interior{e_i} -\mathcal S$. There exists a nonempty closed interval $h_i$ containing $z_i$
such that $h_i\cap (\phi^L)^{-1}(h_i) =\emptyset$.
Indeed, if $h_i\cap (\phi^L)^{-1}(h_i) \neq \emptyset$ for each closed interval containing $z_i$ then
there would be a sequence of points converging to $z_i$  whose $\phi^L$-images also converge to $z_i$,
and so $z_i\in S$.
Property~(1) holds whenever $d_i\subset h_i$.

By continuity of $\phi^L$, there exists $\delta>0$ such that $\neb_\delta((\phi^L)^{-1}(x_i)) \subset (\phi^L)^{-1}(\neb_\epsilon(x_i))$.
Property~(2) holds by choosing $z_i\in\neb_\delta((\phi^L)^{-1}(x_i))$ and letting $d_i$ be a nontrivial component of $h_i\cap\closure{\neb_\delta((\phi^L)^{-1}(x_i))}$.
As there are countably many singular points, Property~(3) holds since we can assume that neither endpoint of  $h_i$ is singular.
Property~(4) holds by letting $z_i=x_i$, and then choosing  $h_i$ above so that it has $x_i$ as an endpoint.

To prove~(5), note that $e_i$ has a subdivision where the vertices are points of $(\phi^L)^{-1}(V^0)$.  If $d_i$ is properly contained in a single closed edge in this subdivision, then $\phi^L$ restricts to an embedding on $d_i$.  This can be arranged by choosing $d_i$ sufficiently small (fixing $x_i$).

We prove~(6) by induction on $|\{e_i\}|$.  The base case, where $k=0$, is vacuous.  Suppose that $d_1,\ldots,d_{k-1}$ have been chosen to satisfy~(1)-(6), with each $d_i$ satisfying $x_k\not\in\phi^L(d_i)$ for $1\leq i<k$.  Choose $d_k$ with properties~(1)-(5) small enough to avoid the finitely many $\phi^L(d_i),\,1\leq i<k$.
\end{proof}

Clearly $d_i\cap d_{i'}=\emptyset$ for $i\neq i'$, since $d_i,d_{i'}$ are contained in distinct open edges. Consequently $d_{ij}\cap d_{i'j'}=\emptyset$ unless $i=i'$ and $j=j'$.

The subspace $\mathbf N$ of $E^{\flat}$ obtained by removing the image under $V\stackrel{\sim}{\rightarrow}E$ of each open secondary bust is the \emph{nucleus}.  Observe that $\mathbf N$ need not be connected.  For each $i,j$, let $q_{ij}^{\pm}$ be the endpoints of $d_{ij}$, which map to $q_i^{\pm}\in V$, and let $p^{\pm}_{ij}$ be the corresponding points of $d'_{ij}\subset E$.  See Figure~\ref{fig:nucleus_attaching_points}.

\begin{figure}[ht]
\includegraphics[width=0.6\textwidth]{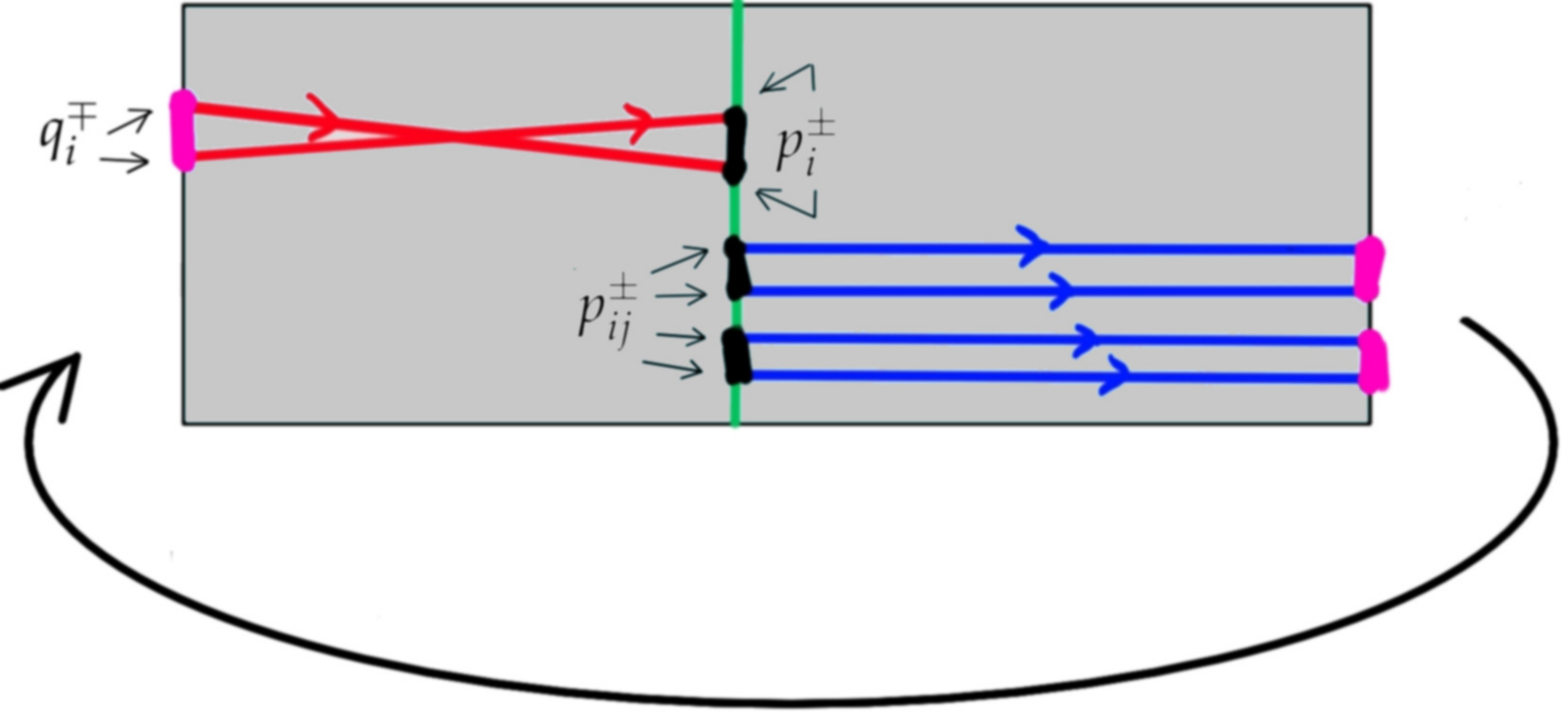}
\caption{Constructing a wall in $X_L$.}\label{fig:nucleus_attaching_points}
\end{figure}

For each $i$, let $\check T_{i}^{o\pm}$ be the image in $X_L$ of the level $T_L^o(\tilde q_i^{\pm})\subset\widetilde X_L$, where $\tilde q_i^{\pm}$ is an arbitrary lift of $q_i^{\pm}\in d_i$. Recall that $\check T_i^{o\pm}$ is an embedded star of length $L$ rooted at $q_i^{\pm}$ with leaves at the various $q_{ij}^{\pm}$.  Let $S_i^{\pm}$ be a segment in the 2-cell $R_{e_i}$ of $X_L$ that joins $q_i^{\pm}$ to $p_i^{\mp}\in E$.  (Note that $S_i^+$ joins $p_i^+$ to $q_i^-$ and $S_i^-$ joins $p_i^-$ to $q_i^+$.)  The arcs $S_i^{\pm}$ are \emph{slopes}.  The \emph{level-part} $T_i^{o\pm}$ is the rooted subtree of $\check T_i^{o\pm}$ with leaves at $p_{ij}^{\pm}$.  The subspace $T_i^{\pm}=T_i^{o\pm}\cup S_i^{\pm}$ obtained by joining the level-part $T_i^{\pm}$ and the slope along the common point $q_i^{\pm}$ is a \emph{tunnel}.  The space $\widehat W^{\bullet}$ determined by the primary busts $\{d_i\}$ and the tunnel length $L$ is the graph obtained by joining each tunnel $T_i^{\pm}$ to $\mathbf N$ along $\{p_{ij}^{\pm}\}\cup\{p_i^{\mp}\}$.  The inclusion $\mathbf N\hookrightarrow X_L$ and the inclusions $T_i^{\pm}\hookrightarrow X_L$ induce a (non-combinatorial) immersion $\widehat W^\bullet\rightarrow X_L$.  Note that $T^{\pm}_i\cap T^{\pm}_j=\emptyset$ when $i\neq j$ since $d_i\cap d_j=\emptyset$. Note that for each $i$, the tunnels $T_i^+$ and $T_i^-$ intersect in the single point $S_i^+\cap S_i^-$.  Composing with the map $X_L\rightarrow X$ gives an immersion $\widehat W^{\bullet}\rightarrow X$. This extends to a local homeomorphism $\widehat W^{\bullet}\times[-1,1]\rightarrow X$ with $\widehat W^{\bullet}$ identified with $\widehat W^{\bullet}\times\{0\}$.  Indeed, we described a map $T_i^{\pm}\times[-1,1]\rightarrow X$ earlier, and $\mathbf N\times[-1,1]\rightarrow X$ is an embedding since $\mathbf N\subset E$, and each $S^{\pm}_i$ lies in a 2-cell.  Appropriately chosen neighborhoods $T_i^{\pm}\times[-1,1]$, and $S_i^{\pm}$, and $\mathbf N$ can be glued to form $\widehat W^{\bullet}\times[-1,1]\rightarrow X_L$.  These gluings can be chosen to preserve a ``normal vector'' at each point of the tunnel, and hence the result is a trivial $[-1,1]$ bundle.  The map $\widehat W^{\bullet}\rightarrow X$ factors through an immersion $W^{\bullet}\rightarrow X$, where $W^{\bullet}$ is obtained from $\widehat W^{\bullet}$ by folding the levels according to the map $\varrho_L:X_L\rightarrow X$ illustrated in Figure~\ref{fig:product_neighborhoods}.   The spaces $\widehat W^\bullet$ and $W^\bullet$ are shown in Figure~\ref{fig:wall_cartoon}.

\begin{figure}[ht]
\includegraphics[width=0.35\textwidth]{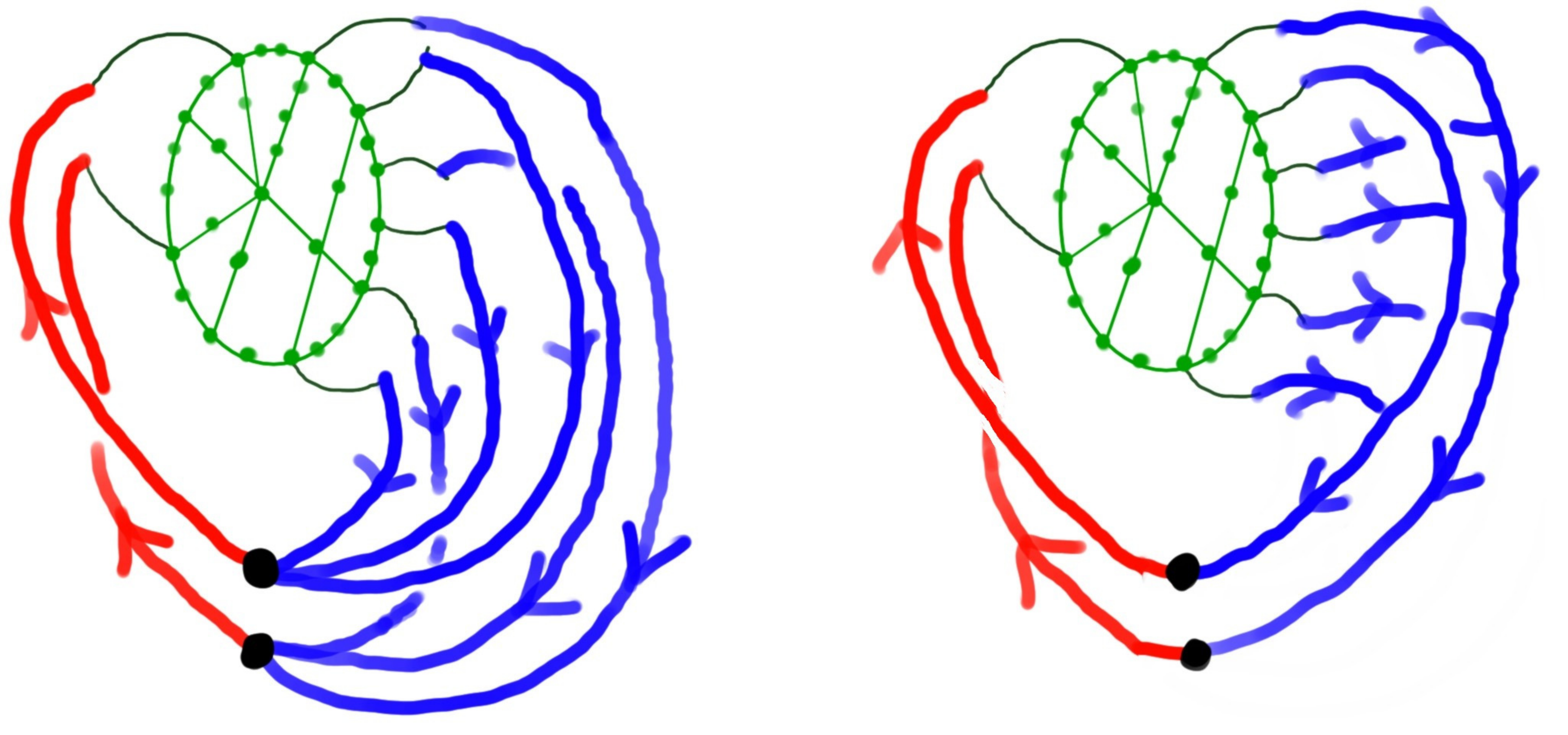}
\caption{At left is $\widehat W^{\bullet}$; at right is $W^{\bullet}$.  Light lines indicate the nucleus, while heavy lines are tunnels.  The emphasized points in each picture are the interior points of a primary bust edge of $V$ where the slopes and levels meet.}\label{fig:wall_cartoon}
\end{figure}

\begin{defn}\label{defn:wall_connected}
A component $W$ of $W^{\bullet}$ is an \emph{immersed wall}.
\end{defn}

\subsection{Description of $\overline W$}\label{subsec:wall_description}
The map $W\rightarrow X$ lifts to a map $\widetilde W\rightarrow\widetilde X$ of universal covers.  For each component $C$ of $\mathbf N$, the universal cover $\widetilde C$ of $C$ lifts to $\widetilde W$, and the restriction of $\widetilde W\rightarrow\widetilde X$ to each such $\widetilde C$ is an embedding.  Moreover, each tunnel lifts to $\widetilde W$, and the map $\widetilde W\rightarrow\widetilde X$ restricts to an embedding on each tunnel $T_i\subset\widetilde W$.  Let $\overline W=\image(\widetilde W\rightarrow\widetilde X)$ and let $H_W=\stabilizer_G(\overline W)$.  We conclude that:

\begin{rem}\label{rem:invariant_halfspace}
When $\overline W$ is locally isomorphic to $W$, the trivial $[-1,1]$-bundle discussed above ensures that there are exactly two components of $\widetilde X-\overline W$, each of which is $H_W$-invariant.
\end{rem}

\begin{rem}[Future shape of $\overline W$]\label{rem:wall_structure}
We now describe the structure of $\overline W$ in the situation in which distinct tunnels are disjoint.  Note that tunnels $T_i^{\pm}$ and $T_j^{\pm}$ in $\overline W$ are disjoint when $i\neq j$, since they map to disjoint tunnels in $\image(W\rightarrow X)$.  Moreover, we shall show below that, under certain conditions, tunnels $T,T'\subset\overline W$, mapping to $T_i^+,T_i^-$ respectively, are disjoint when $L$ is large.  In this situation, $\overline W$ will be shown to have the structure of a tree of spaces, whose underlying vertices are equipped with a 2-coloring.  Red vertices correspond to slopes, while green vertices correspond to subspaces that are maximal connected unions of universal covers of nuclei and lifts of level-parts.  Note that $\overline W$ may still fail to be simply connected -- i.e. $\widetilde W$ may still fail to embed -- since subspaces corresponding to green vertices may not be simply-connected.  If $\overline W$ contains a nucleus in $\widetilde E_n$, then all nuclei lie in $\widetilde E_{n+kL},\,k\in\integers$, and any two nuclei contained in a common vertex space lie in the same space $\widetilde E_{n+kL}$.  A heuristic picture of $\overline W$ is shown in Figure~\ref{fig:heuristic_wall}, and Figure~\ref{fig:small_wall} shows a part of $\overline W$ inside $\widetilde X$.
\end{rem}

\begin{figure}[ht]
\begin{overpic}[width=0.35\textwidth]{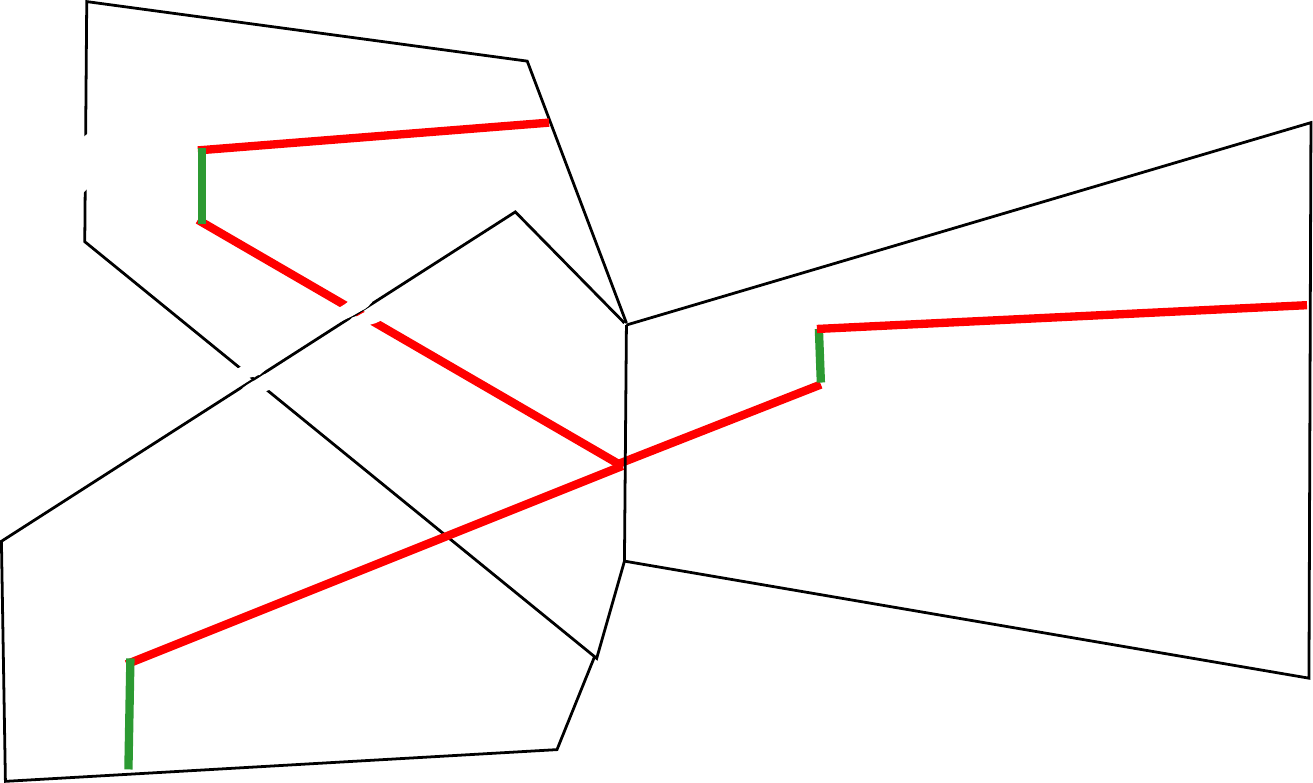}
\put(25,10){\scriptsize{level}}
\put(75,30){\scriptsize{level}}
\put(55,23){\scriptsize{slope}}
\put(-4,45){\scriptsize{nucleus}}
\end{overpic}
\caption{A heuristic picture of part of a wall in $\widetilde X$.  The two nuclei at left, and the levels at left, belong to the same knockout.  This knockout does not contain the slope or the nucleus and level at right.}
\label{fig:heuristic_wall}
\end{figure}

\begin{com}add the replacement figure 8\end{com}

\begin{figure}[ht]
\includegraphics[width=0.4\textwidth]{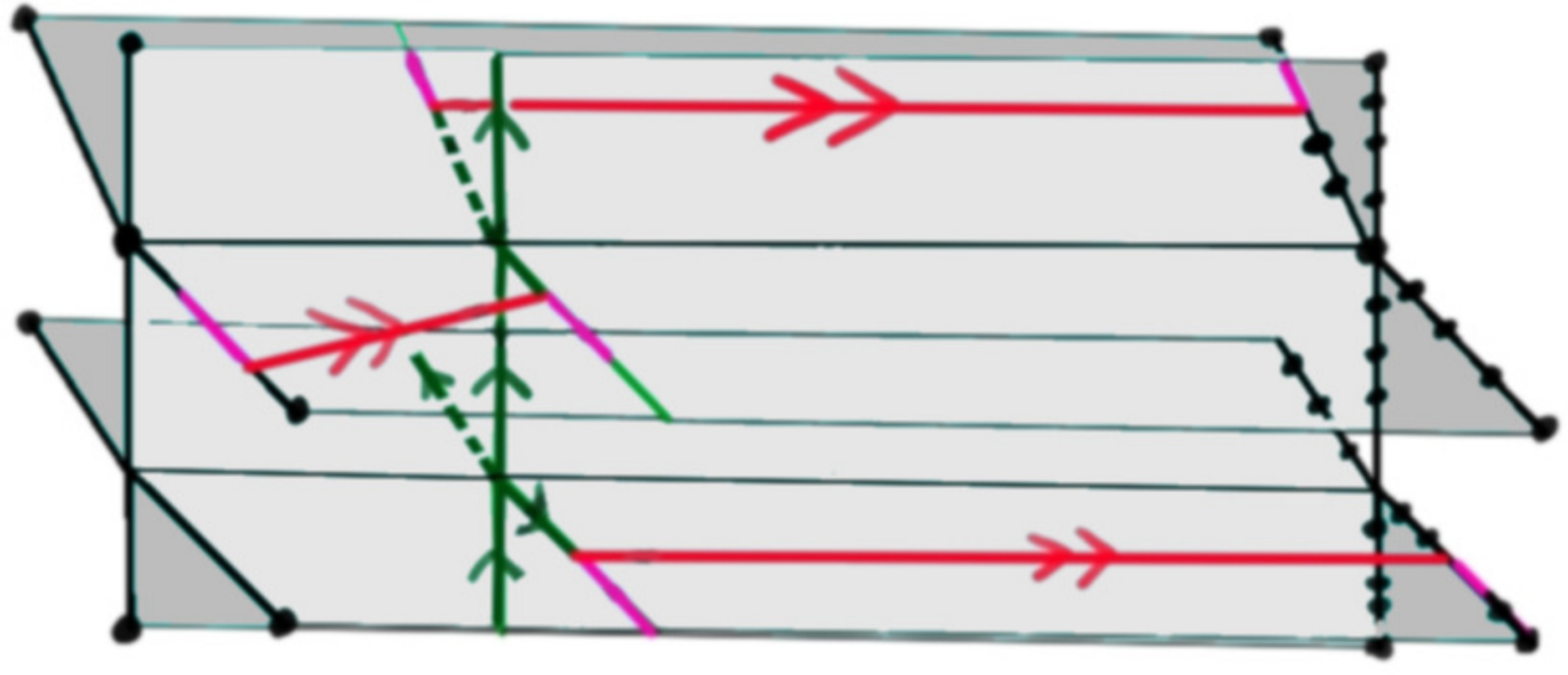}
\caption{Part of a wall $\widetilde W\rightarrow\widetilde X$.  The single-arrowed segments belong to a nucleus, while the double-arrowed segments are tunnels.}\label{fig:small_wall}
\end{figure}


\subsection{The approximation}\label{subsec:approximation}
Let $N(\overline W)$ denote the union of all closed 2-cells of $\widetilde X$ that intersect $\overline W$. We will show that $N(\overline W)^1$ is quasiconvex in $\widetilde X^1$ under certain conditions, notably sufficiently large tunnel length.  However, the quasiconvexity constant will depend on the tunnel length.  This is partly because levels are not uniformly quasiconvex and partly because distinct levels emanating from very close secondary busts may contain long forward paths that closely fellow-travel.  To achieve uniform quasiconvexity we define the \emph{approximation} of $\overline W$, which also has the key feature that it lifts to a geometric wall in ${\widetilde X^\bullet}_L$.

\newcommand{\comappn}[1]{N(\mathbf A(#1))^1}

\begin{defn}[Approximation]\label{defn:approximation}
Let $W\rightarrow X$ be an immersed wall with tunnel length $L$ and primary busted edges $\{e_i\}$.  Let $\overline W$ be the image of a lift $\widetilde W\rightarrow\widetilde X$ of the universal cover of $W$ to $\widetilde X$.  We define a map $\comapp:\overline W\rightarrow\widetilde X$ as follows.  First, suppose that $\widetilde C\subset\overline W$ is the universal cover of a component of the nucleus of $W$.  Let $n\in\integers$ be such that $\widetilde C\subset\widetilde E_n$, and let $\widetilde C'\subset\widetilde V_n$ be the parallel copy of $\widetilde C$.  For each $c\in\widetilde C$, let $c'$ denote the corresponding point of $\widetilde C'$.  Then $\comapp:\widetilde C\rightarrow\widetilde X$ is defined by $\comapp(c)=\tilde\phi^L(c')$.  For each level-part $T^o$ of $\overline W$, let $q$ be the root of $T^o$.  Then $\comapp(t)=q$ for each $t\in T_o$.  Finally, let $S\subset\overline W$ be a slope, beginning at $q$ and ending on a point $p$ in a nucleus $\widetilde C$.  Then $p$ is an endpoint of a primary bust $d_i\subset\widetilde E_n$.  The map $\comapp$ sends the slope $S$ homeomorphically to the path $d_i'P$, where $d_i'$ is the parallel copy of $d_i$ in $\widetilde V_n$ that joins $q$ to $p'$ and $P$ is the forward path joining $p'$ to $\tilde\phi^L(p')$.  See Figure~\ref{fig:comapp_def}.

The \emph{approximation} of $\overline W$ is the subspace $\comapp(\overline W)\subset\widetilde X$.  Note that $\comapp(\overline W)$ is the union of length-$L$ forward paths together with subspaces of $\widetilde V_{nL}$ for each $n\in\integers$.  Let $\comappn{\overline W}$ be the 1-skeleton of the smallest subcomplex of $\widetilde X$ containing $\comapp(\overline W)$.
\end{defn}

\begin{figure}[ht]
\includegraphics[width=0.4\textwidth]{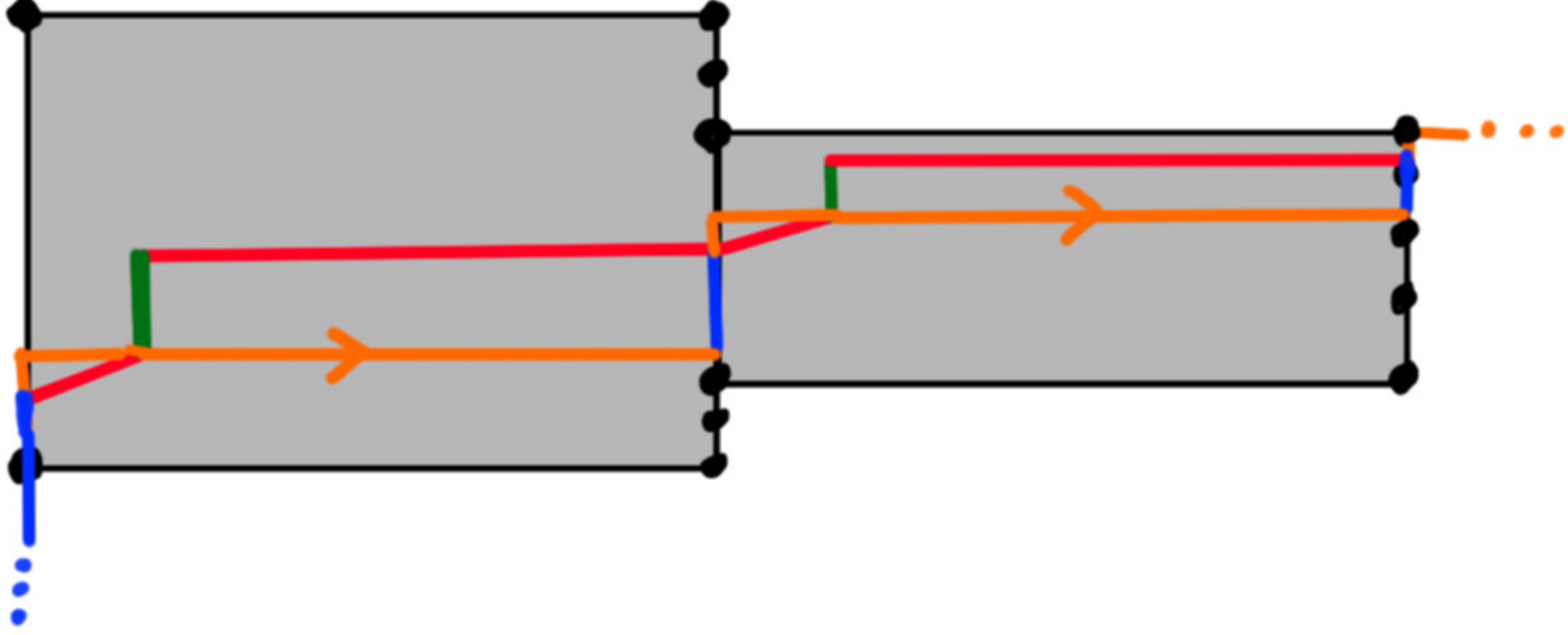}\\
\caption{Part of a wall and its approximation.  The arrowed paths are the approximations of the slopes intersecting them.}\label{fig:comapp_def}
\end{figure}

\begin{rem}
If $\widetilde C_1,\widetilde C_2$ are nuclei of $\overline W$ intersecting a level-part of a tunnel of $\overline W$, then $\comapp(\widetilde C_1)\cap\comapp(\widetilde C_2)\neq\emptyset$.  For each $n\in\integers$, each component of $\comapp(\overline W)\cap\widetilde V_n$ is formed as follows.  A \emph{knockout} $\widetilde K$ is a maximal connected subspace of $\overline W$ that does not contain an interior point of a slope.  The knockout $\widetilde K$ is \emph{at position $n$} if it is the union of nuclei in $\overline W\cap\widetilde E_{n-L}$ together with level-parts traveling from $\widetilde E_{n-L}$ to $\widetilde V_n$.  To each position-$n$ knockout $\widetilde K$, we associate a component of $\comapp(\overline W)\cap\widetilde V_n$, namely the one obtained from the connected subspace $\comapp(\widetilde K)\subset\widetilde V_n$ by adding all (closed) primary bust intervals that intersect $\comapp(\widetilde K)$.  See Figure~\ref{fig:pushforward}.
\end{rem}

\begin{figure}[ht]
\includegraphics[width=0.3\textwidth]{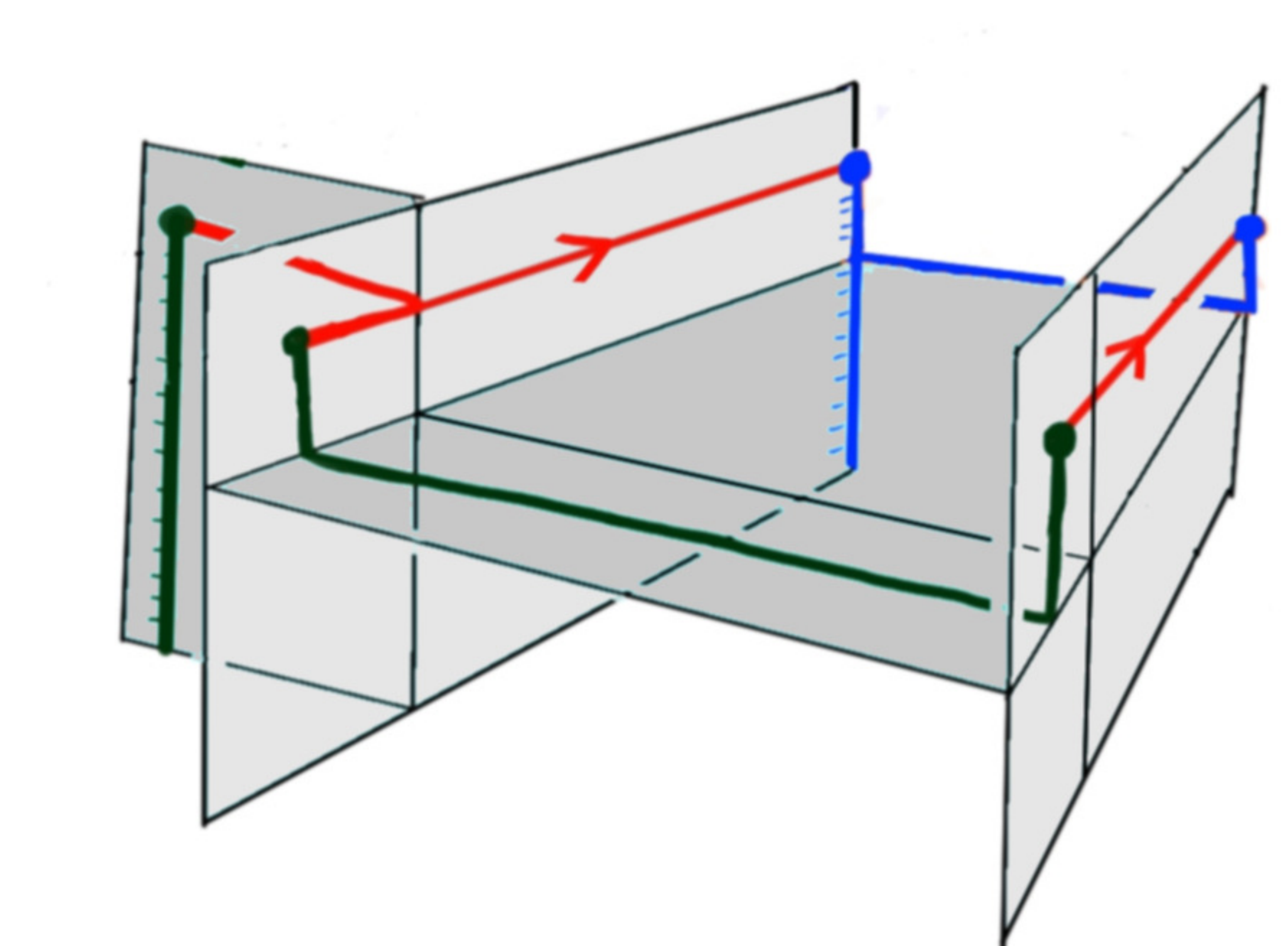}\\
\caption{Two nuclei intersecting a common level-part have intersecting approximations whose union avoids primary busts.}\label{fig:pushforward}
\end{figure}

\begin{rem}\label{rem:distinct_slopes_disjoint_approxes}
If $S_1,S_2$ are distinct slopes, rooted at primary busts $d_1,d_2$, then $\comapp(S_1)\cap\comapp(S_2)=\emptyset$ by Lemma~\ref{lem:choosing_busts_1}.(5)-(6).
\end{rem}

\subsection{Hypotheses on $\comapp(\overline W)$ enabling quasiconvexity}\label{sec:lol}
We recall that we are assuming, for the rest of the paper, that there exists $\delta\geq0$ such that $\widetilde X^1$ is $\delta$-hyperbolic.   The following statements are instrumental in proving that, provided $L$ is sufficiently large, $\comappn{\overline W}\hookrightarrow\widetilde X$ is quasi-isometrically embedded, and the quasi-isometry constants are independent of $L$.

\begin{defn}[Sub-quasiconvex]\label{defn:bust_quasiconvex}
$W$ is \emph{sub-quasiconvex} if there exist constants $\mu_1',\mu_2'$ such that each component in $\widetilde X$ of the preimage of $V-\cup_i\interior{d_i}$ is $(\mu_1',\mu_2')$-quasi-isometrically embedded.  For example, as noted above, $W$ is sub-quasiconvex if $V^{\flat}$ is a forest or if $\Phi$ is an aperiodic isomorphism and there is at least one bust.
\end{defn}

\begin{lem}[Quasiconvexity of approximations of nuclei]\label{lem:avoids_bust}
Approximations have the following properties when $W$ is sub-quasiconvex:
\begin{enumerate}
 \item For each nucleus $\widetilde C$ and each open primary bust $\tilde d$, we have $\comapp(\widetilde C)\cap\tilde d=\emptyset$.
 \item Let $\widetilde K$ be a knockout.  Then $\comapp(\widetilde K)\cap\tilde d=\emptyset$ for each open primary bust $\tilde d$.
 \item \label{nucleus_approx_quasiconvex} Hence there exist $\mu_1\geq1,\mu_2\geq 0$ such that, for each $n\in\integers$ and each component $\mathbf K$ of $\widetilde V_n\cap\comapp(\overline W)$, the inclusion $\mathbf K\rightarrow\widetilde X^1$ is a $(\mu_1,\mu_2)$-quasi-isometric embedding.  Moreover, $\mu_1$ and $\mu_2$ are independent of $\{d_i\}$ and $L$.
\end{enumerate}
\end{lem}

\begin{proof}
\begin{enumerate}
 \item $\widetilde C$ has empty intersection with the set of open secondary busts in $\widetilde V_n$, and hence maps to the complement of the set of open primary busts in $\widetilde V_{n+L}$.
 \item This follows immediately from $(1)$ because level-parts map to points.
 \item Since $W$ is sub-quasiconvex, Statement~(2) implies that $\mathbf K$ is a subtree of a uniform neighborhood in $\widetilde V_n$ of some $\comapp(\widetilde K)$, and the claim follows.  Since there are finitely many possible sets of primary busted edges, the constants $\mu_1,\mu_2$ can be chosen independently not only of $L$ and $\{d_i\}$, but also of $\{e_i\}$.  Indeed, each set $\{e_i\}$ of edges yielding a sub-quasiconvex immersed wall gives rise to a pair of quasi-isometry constants, and we take $\mu_1,\mu_2$ to be the maximal such constants.\qedhere
\end{enumerate}
\end{proof}

Lemma~\ref{lem:avoids_bust} provides uniform quasiconvexity of nucleus approximations, and Proposition~\ref{prop:forward_ladder_quasiconvex} provides uniform quasiconvexity of forward ladders.  Lemma~\ref{lem:vert_horiz_bound} provides a bound on the diameters of coarse intersections of nucleus approximations and carriers of approximations of slopes.  To prove quasiconvexity of $\comappn{\overline W}$ requires the following additional property.

\begin{defn}[Ladder overlap property]\label{defn:ladder_overlap_property}
The family of immersed walls $\{W_i\rightarrow X\}$ has the \emph{ladder overlap property} if there exists $B\geq0$ such that for all $i$ and all distinct tunnels $T_1,T_2\subset\overline W_i$ intersecting a common nucleus,  $$\diam\left(\neb_{3\delta+2\lambda}(N(\comapp(T_1)))\cap\neb_{3\delta+2\lambda}(N(\comapp(T_2)))\right)\leq B,$$
where $\lambda$ is the constant from Proposition~\ref{prop:forward_ladder_quasiconvex}.
\end{defn}

\begin{rem}\label{rem:lol_and_wall}
The purpose of the ladder overlap property is to guarantee that, when $L$ is large and $W$ is sub-quasiconvex, paths of the form $\beta\alpha\beta'$ are uniform quasigeodesics, where $\beta,\beta'$ are geodesics of carriers of slope-approximations and $\alpha$ is a vertical geodesic in $\comapp(\overline W)$.

If the interiors of $\beta,\beta'$ have disjoint images in $\mathbf R$, Lemma~\ref{lem:vert_horiz_bound} ensures that $\beta\alpha\beta'$ is a uniform quasigeodesic.  The interesting situations are those in which $\beta,\beta'$ are both incoming or both outgoing with respect to the vertical part of $\comapp(\overline W)$ containing $\alpha$.  A thin quadrilateral argument shows that in either case, the ladder overlap property ensures that $\beta,\beta'$ have uniformly bounded coarse intersection, from which one concludes that $\beta\alpha\beta'$ is a uniform quasigeodesic (see Lemma~\ref{lem:RBRquasi} below).
\end{rem}

\section{Quasiconvex codimension-1 subgroups from immersed walls}\label{sec:quasiconvexity}
In this section, we determine conditions ensuring that $\comappn{\overline W}$ is quasiconvex and $\overline W$ is a wall in $\widetilde X$.

\subsection{Uniform quasiconvexity}\label{subsec:quasiconvexity_of_approximation}
A collection $\{W\rightarrow X\}$ of immersed walls is \emph{uniformly sub-quasiconvex} if there exist constants $\mu_1,\mu_2$ such that $\comapp(\widetilde K)\hookrightarrow\widetilde X^1$ is a $(\mu_1,\mu_2)$-quasi-isometric embedding for each $W$ and each knockout $\widetilde K$ of $\overline W$.  The first goal of this section is to prove:

\begin{prop}\label{prop:quasiconvexity}
Let $\mathbb W=\{W\rightarrow X\}$ be a uniformly sub-quasiconvex set of immersed walls with the ladder-overlap property.  Then there exists $L_0,\kappa_1,\kappa_2$ such that for all $W\in\mathbb W$ with tunnel length at least $L_0$, the inclusion $\comappn{\overline W}\hookrightarrow\widetilde X^1$ is a $(\kappa_1,\kappa_2)$-quasi-isometric embedding.
\end{prop}

The constants are $\kappa_1=4\lambda_1\mu_1$ and $\kappa_2=\frac{\mu_2}{2}+2L_0(1+\frac{1}{4\lambda_1\mu_1})$.  Here $\mu_1,\mu_2$ are the quasi-isometry constants from uniform sub-quasiconvexity, and $\lambda_1$ is the multiplicative quasi-isometry constant for 1-skeleta of forward ladders.  We emphasize that these are independent of $L$ and of the collection of primary busts.

\begin{proof}[Proof of Proposition~\ref{prop:quasiconvexity}]
For a path $P$ in $\widetilde X^1$, as usual $\|P\|$ denotes the distance in $\widetilde X^1$ between the endpoints of $P$.  If $P$ is a geodesic of $\comappn{\overline W}$, then its edge-length $|P|$ equals the distance in $\comappn{\overline W}$ between the endpoints of $P$.  We will show that when $L\geq L_0$, there are constants $\kappa_1,\kappa_2$ such that $\|P\|\geq\kappa_1^{-1}|P|-\kappa_2$.

\textbf{Alternating geodesics:}  Let $P'$ be a geodesic in the graph $\comappn{\overline W}$.  Suppose $P'$ \emph{alternates}, in the sense that $P'=\alpha_0\beta_1\alpha_1\cdots\beta_k\alpha_k$, where each $\alpha_i$ is a vertical geodesic path, and each $\beta_i$ is a geodesic of the 1-skeleton of a length-$L$ forward ladder (and thus a $(\lambda_1,\lambda_2)$-quasigeodesic).  We allow the possibility that $\alpha_0$ or $\alpha_{k}$ has length~0.

Each $\alpha_i$ is a $(\mu_1,\mu_2)$-quasigeodesic by our hypothesis that knockout-approximations are quasi-isometrically embedded.  Since $W$ has the ladder overlap property, $\diam(\neb_{3\delta+2\lambda}(\beta_i)\cap\neb_{3\delta+2\lambda}(\beta_{i+1}))\leq B$.  Let $B_0=\max(B,\Theta_{3\delta+2\lambda})$, where $\Theta_{3\delta+2\lambda}$ is as in Lemma~\ref{lem:vert_horiz_bound}.  Applying Lemma~\ref{lem:RBRquasi} yields a constant $L_0$ such that, if $L\geq L_0$, then $\|P'\|\geq\frac{1}{4\lambda_1\mu_1}|P'|-\frac{\mu_2}{2}$.

\textbf{$\comapp(\overline W)$ quasi-isometrically embeds:}  Let $P$ be a geodesic of $\comappn{\overline W}$.  By construction $P=\beta_0'P'\beta_{k+1}'$ where $P'$ is alternating and $\beta_0',\beta_{k+1}'$ are (possibly trivial) paths in forward ladders.  If $|\beta_0'|,|\beta_{k+1}'|\geq L_0$, then $\|P\|\geq\frac{1}{4\lambda_1\mu_1}|P|-\frac{\mu_2}{2}$ by Lemma~\ref{lem:RBRquasi}.  If $|\beta_0'|,|\beta_{k+1}'|\leq L_0$, then since $P'$ is alternating,
$$
\|P\|\geq\frac{1}{4\lambda_1\mu_1}|P'|-\frac{\mu_2}{2}-2L_0\geq\frac{1}{4\lambda_1\mu_1}|P|-\frac{\mu_2}{2}-2L_0(1+\frac{1}{4\lambda_1\mu_1}).
$$
In the remaining case, without loss of generality, $P=\beta_0'P''$, where $|\beta_0'|\leq L_0$ and $P''$ satisfies $\|P''\|\geq\frac{1}{4\lambda_1\mu_1}|P''|-\frac{\mu_2}{2}$ by Lemma~\ref{lem:RBRquasi}.  The proof is thus complete with $\kappa_1=4\lambda_1\mu_1$ and $\kappa_2=\frac{\mu_2}{2}+2L_0(1+\frac{1}{4\lambda_1\mu_1})$.\end{proof}

\begin{lem}\label{lem:hsuwiseRBR}
Let $Z$ be $\delta$-hyperbolic, and let $P=\alpha_0\beta_1\alpha_1\cdots\beta_k\alpha_k$ be a path in $Z$ with all $\alpha_i$ and $\beta_i$ geodesic.  Suppose there exists $B\geq 0$ such that for all $i$, each intersection below has diameter $\leq B$:
$$\neb_{3\delta}(\beta_i)\cap\beta_{i+1},\,\,\,\,\,\,\,\neb_{3\delta}(\beta_i)\cap\alpha_i,\,\,\,\,\,\,\,\neb_{3\delta}(\beta_i)\cap\alpha_{i-1}.$$
Then if $|\beta_i|\geq12(B+\delta)$ for each $i$, then $\|P\|\geq\frac{1}{2}\left(\sum_{i=0}^{k}|\alpha_i|+\sum_{i=1}^k|\beta_i|\right)$.
\end{lem}

\begin{proof}
This is a standard argument.  We refer, for instance, to~\cite[Thm~2.3]{HsuWiseCubulatingMalnormal}.
\end{proof}

We now promote Lemma~\ref{lem:hsuwiseRBR} to a statement about piecewise-quasigeodesics.

\begin{lem}\label{lem:RBRquasi}
Let $Z$ be $\delta$-hyperbolic and let $P=\alpha_0\beta_1\alpha_1\cdots\beta_k\alpha_k$ be a path in $Z$ such that each $\beta_i$ is a $(\lambda_1,\lambda_2)$-quasigeodesic and each $\alpha_i$ is a $(\mu_1,\mu_2)$-quasigeodesic.  Suppose that for each $R\geq 0$ there exists $B_R\geq 0$ such that for all $i$, each intersection below has diameter $\leq B_R$:
$$\neb_{3\delta+R}(\beta_i)\cap\beta_{i+1},\,\,\,\,\,\,\,\neb_{3\delta+R}(\beta_i)\cap\alpha_i,\,\,\,\,\,\,\,\neb_{3\delta+R}(\beta_i)\cap\alpha_{i-1}.$$

Then there exists $L_0$ such that, if $|\beta_i|\geq L_0$ for each $i$, then $\|P\|\geq\frac{1}{4\lambda_1\mu_1}|P|-\frac{\mu_2}{2}$.
\end{lem}

\begin{proof}
For each $i$, let $\bar\alpha_i$ [respectively, $\bar\beta_i$] be a geodesic with the same endpoints as $\alpha_i$ [respectively, $\beta_i$], and let $\overline P=\bar\alpha_0\bar\beta_1\bar\alpha_1\cdots\bar\beta_k\bar\alpha_k$ be a piecewise-geodesic with the same endpoints as $P$.  Since $Z$ is $\delta$-hyperbolic, there exists $\mu=\mu(\mu_1,\mu_2,\delta)$ such that $\alpha_i$ and $\bar\alpha_i$ lie at Hausdorff distance at most $\mu$, and there exists $\lambda=\lambda(\lambda_1,\lambda_2,\delta)$ such that $\bar\beta_i$ and $\beta_i$ lie at Hausdorff distance at most $\lambda$.

Note that if $R_1\leq R_2$, then we may assume $B_{R_1}\leq B_{R_2}$.  By hypothesis, $\neb_{3\delta+2\lambda}(\beta_i)\cap\beta_{i+1}$ has diameter $\leq B_{2\lambda}$.  Moreover, if $\bar\beta'\subset\bar\beta_i$ is a subpath that $3\delta$-fellowtravels with a subpath $\bar\alpha'$ of $\bar\alpha_i$ or $\bar\alpha_{i-1}$, then $\beta'$ fellowtravels at distance $3\delta+\mu+\lambda$ with a subpath $\alpha''$ of $\alpha_i$ or $\alpha_{i-1}$, whence $|\beta'|\leq B_{\mu+\lambda}$ by hypothesis.  Letting $L_0\geq12(\delta+B_{\mu+2\lambda})$ and applying Lemma~\ref{lem:hsuwiseRBR} shows that $\overline P$ is a $(\frac{1}{2},0)$-quasigeodesic, and we have:
\begin{equation}\label{eqn:piecewisegeodesic}
\|P\|=\left\|\overline P\right\|\geq\frac{1}{2}\left|\overline P\right|.
\end{equation}

Since $\mu_1,\lambda_1\geq 1$, we can bound $|\overline P|$ as follows:
\begin{eqnarray*}
\left|\overline P\right|=\sum_{i=1}^k|\bar\beta_i|+\sum_{i=0}^k|\bar\alpha_i|&\geq&\sum_{i=1}^k(\lambda_1^{-1}|\beta_i|-\lambda_2)+\sum_{i=0}^k(\mu_1^{-1}|\alpha_i|-\mu_2)\\
&=&\left[\lambda_1^{-1}\sum_{i=1}^k\left(|\beta_i|-\lambda_1(\lambda_2+\mu_2)\right)+\mu_1^{-1}\sum_{i=0}^k|\alpha_i|\right]-\mu_2\\
&\geq&(\lambda_1\mu_1)^{-1}\left[\sum_{i=1}^k(|\beta_i|-\lambda_1(\lambda_2+\mu_2))+\sum_{i=0}^k|\alpha_i|\right]-\mu_2.\\
\end{eqnarray*}
If $L_0\geq 2\lambda_1(\lambda_2+\mu_2)+1$, then, provided that $|\beta_i|\geq L\geq L_0$, we have:
\[\left|\overline P\right|\geq\frac{1}{2\lambda_1\mu_1}\left[\sum_{i=1}^k|\beta_i|+\sum_{i=0}^k|\alpha_i|\right]-\mu_2.\]
Combining this with Equation~\eqref{eqn:piecewisegeodesic} yields $\|P\|\geq\frac{1}{4\lambda_1\mu_1}|P|-\frac{\mu_2}{2}$.
\end{proof}

\subsection{$\overline W$ is a wall when tunnels are long}\label{subsec:W_is_wall}
A subspace $Y\subset\widetilde X$ is a \emph{wall} if $\widetilde X-Y$ has exactly two components, each of which is stabilized by $\stabilizer(Y)$.  Note that this definition is stricter than usual.  For more about wallspaces and the various definitions, background, and references, see~\cite{HruskaWise}.
Our goal is now to show that if $W\rightarrow X$ is an immersed wall with sufficiently long tunnels, then $\overline W$ is a wall.  We need the following useful consequence of quasiconvexity.

\begin{prop}\label{prop:approximation_is_tree}
Let $\mathbb W$ satisfy the hypotheses of Proposition~\ref{prop:quasiconvexity}.  There exists $L_1\geq L_0$ such that  $\comapp(\overline W)$ is a tree for each $W\in\mathbb W$ with tunnel length $L\geq L_1$.
\end{prop}

\begin{proof}
Let $Q$ be an immersed path in $\comapp(\overline W)$, and let $Q'$ be a geodesic of $\comappn{\overline W}
$ with the same endpoints as $Q$.  Proposition~\ref{prop:quasiconvexity} implies that $\|Q'\|\geq\kappa_1^{-1}|Q|-\kappa_2$.  Hence if $|Q'|\geq L_1=\max(L_0,\kappa_1\kappa_2+\kappa_1)$, then $Q$ is not closed.  Any immersed path $Q$ in $\comapp(\overline W)$ either lies in a single vertex space and is thus not closed, or contains a slope approximation and thus $Q'$ has length at least $L$.
\end{proof}

\begin{rem}[Tree of spaces structure on $\overline W$]\label{rem:wall_is_tree_of_spaces}
Proposition~\ref{prop:approximation_is_tree} justifies our claim in Remark~\ref{rem:wall_structure} that $\overline W$ is a tree of spaces when $L$ is sufficiently large, assuming that $W$ is sub-quasiconvex and has the ladder overlap property.  Indeed, any cycle in $\overline W$ that is not contained in a knockout will map to a cycle in $\comapp(\overline W)$, contradicting Proposition~\ref{prop:approximation_is_tree}.
\end{rem}

\begin{prop}\label{prop:tunneliswall}
Let $W\rightarrow X$ be an immersed wall in a collection $\mathbb W$ satisfying the hypotheses of Proposition~\ref{prop:quasiconvexity}.  There exists $L_1\geq L_0$ such that the image $\overline W\subset\widetilde X$ of $\widetilde W\rightarrow\widetilde X$ is a wall provided that $W$ has tunnel length $L\geq L_1$.
\end{prop}

\begin{proof}
Since $\homology^1(\widetilde X)=0$, it suffices to show that $\overline W$ has an
open neighborhood homeomorphic to $\overline W \times [-1,1]$ with $\overline W$  identified with $\overline W\times \{0\}$.
The local homeomorphism $W\times[-1,1]\rightarrow X$ lifts to a map $\widetilde W\times[-1,1]\rightarrow \widetilde X$.  We note as well that $W\rightarrow X$ is locally two-sided by construction. The image of $\widetilde W\times[-1,1]\rightarrow\widetilde X$ would provide the desired neighborhood $\overline W \times [-1,1]$ provided that this map is a covering map onto its image.  By choosing the image of $W\times[-1,1]$ to be sufficiently narrow, the only place where this could fail is where distinct slopes of $\overline W$ intersect.  To exclude this possibility, we will show that distinct tunnels $T_0,T_k$ of $\overline W$ are disjoint.

Suppose that $T_0\neq T_k$ and $T_0\cap T_k\neq\emptyset$.  Let $e$ be the primary busted edge dual to $T_0$ and $T_k$.  Since $T_0,T_k\subset\overline W$, there exists a path $P\rightarrow\overline W$ that starts on $T_0$, ends on $T_k$, and which is disjoint from the interiors of $T_0$ and $T_k$.  Indeed, let $\widetilde P\rightarrow\widetilde W$ be a path joining lifts of $\widetilde T_0,\widetilde T_k$ and let $P$ be the image of $\widetilde P$ in $\overline W$.  Moreover, we assume that $\widetilde P$ is minimal in the sense that it is disjoint from intervening lifts of $T_0,T_k$.  The minimality of $\widetilde P$ ensures that $P$ has the desired property.

There are three cases.  The first is where $P$ starts and ends on the levels of $T_0,T_k$.  The second is where $P$ starts and ends on the slopes of $T_0,T_k$.  The third case is where $P$ starts on the level of (say) $T_0$ and ends on the slope of $T_k$.

In the first case, the approximation $\comapp(P)$ of the image of $P$ is a connected subspace of $\comapp(\overline W)$ that contains the endpoints of $e$ but does not contain the entire edge $e$.  Hence $\comapp(P)\cup e$ contains a cycle. Since $\comapp(P)\subset\comapp(\overline W)$ and $e\subset\comapp(\overline W)$, there is a cycle in $\comapp(\overline W)$, contradicting Proposition~\ref{prop:approximation_is_tree}.

Similarly, in the second case, $\comapp(P)$ is disjoint from $e$, so that $\comapp(P)\cup e\cup\comapp(T_0)\cup\comapp(T_k)$ is a subspace of $\comapp(\overline W)$ that contains a cycle.  In the third case, the contradictory subspace is $\comapp(P)\cup\comapp(T_k)\cup e$.
\end{proof}

We note the following corollary:

\begin{cor}\label{cor:quasiconvexsubgroup}
Let $\mathbb W$ be a set of sub-quasiconvex immersed walls such that $\mathbb W$ has the ladder overlap property.  Then there exists $L_1$ such that for all $W\in\mathbb W$ with tunnel length $L\geq L_1$, the stabilizer $H_W\leq G$ of $\overline W$ is a quasiconvex, codimension-1 free subgroup.
\end{cor}

\section{Cutting geodesics}\label{sec:cutting_geodesics}
In this section, we recall the criterion for cocompact cubulation of hyperbolic groups given in~\cite{BergeronWise} and describe how a sufficiently rich collection of quasiconvex walls in $\widetilde X$ ensures that this criterion is satisfied.

\newcommand{\moustache}[1]{\overline{{#1}}_L}

\subsection{Separating points on $\visual\widetilde X$}\label{subsec:cutting_definition}
Let $\visual\widetilde X$ denote the Gromov boundary of $\widetilde X^1$.  Let $W\rightarrow X$ be an immersed wall with the property that $\comappn{\overline W}$ is quasiconvex in $\widetilde X^1$ and $\overline W$ is a wall.  Let $\leftw$ and $\rightw$ be the components of $\widetilde X-\overline W$, and let $N(\leftw),N(\rightw)$ be the smallest subcomplexes containing $\leftw,\rightw$ respectively.  Then $N(\leftw)^1\cap N(\rightw)^1=N(\overline W)^1$, which is coarsely equal to $\comappn{\overline W}$.  Let $\visual\overline W$ denote $\visual N(\overline W)^1=\visual\comappn{\overline W}$, which is a closed subset of $\visual\widetilde X$ since $\comappn{\overline W}$ is quasiconvex in $\widetilde X^1$.  Let $\visual\leftw=\visual N(\leftw)^1-\visual\overline W$ and let $\visual\rightw=\visual N(\rightw)^1-\visual\overline W$, so that $\visual\leftw$ and $\visual\rightw$ are disjoint open subsets of $\visual\widetilde X$.  Note that $\visual\leftw$ and $\visual\rightw$ are $H_W$-invariant, since $N(\leftw)$ and $N(\rightw)$ are $H_W$-invariant by Remark~\ref{rem:invariant_halfspace}.

Let $p,q\in\visual\widetilde X$ be the endpoints of a bi-infinite geodesic $\gamma:\mathbf R\rightarrow\widetilde X^1$.  Then $\gamma$ is \emph{cut} by $\overline W$ if $p\in\visual\leftw$ and $q\in\visual\rightw$ or vice versa.

The following holds by~\cite[Thm~1.4]{BergeronWise}:

\begin{prop}\label{prop:cutting_criterion}
Suppose that for every geodesic $\gamma:\mathbf R\rightarrow\widetilde X^1$, there exists a wall $\overline W$ of the type described in Section~\ref{sec:immersed_walls}, such that $N(\overline W)$ is quasiconvex and such that $\overline W$ cuts $\gamma$.  Then there exists a $G$-finite collection $\{\overline W\}$ of walls in $\widetilde X$ such that $G$ acts freely and cocompactly on the dual CAT(0) cube complex.
\end{prop}


\subsection{A method for cutting the two types of geodesics}\label{subsec:5.2}

\begin{defn}[Ladderlike, deviating]\label{defn:level_like_deviating}
Let $M\geq 0$ and let $\gamma\subset\widetilde X^1$ be an embedded infinite or bi-infinite path whose image is $\xi$-quasiconvex, for some $\xi\geq 0$.  Then $\gamma$ is \emph{$M$-ladderlike} if there exists a forward ladder $N(\sigma)$, where $\sigma$ is a forward path of length $M$, such that a geodesic of $N(\sigma)$ joining the endpoints of $\sigma$ fellow-travels with a subpath of $\gamma$ at distance $2\delta+\lambda+\xi$.  Here, $\widetilde X^1$ is $\delta$-hyperbolic and $\lambda$ is the constant from Proposition~\ref{prop:forward_ladder_quasiconvex}.  Otherwise, $\gamma$ is \emph{$M$-deviating}.
\end{defn}

Note that if $\gamma$ is $M$-deviating, for each $R\geq 0$ there exists $M_R$ depending only on $\xi,M,R$ such that for all forward paths $\sigma$, we have $\diam(\gamma\cap\neb_R(N(\sigma)^1))\leq M_R$.  Moreover, if $\gamma$ is $2M$-deviating, the same bound holds with $\sigma$ replaced by any level, since any geodesic in a level decomposes as the concatenation of two (possibly trivial) forward paths.

\begin{defn}[Many effective walls]\label{defn:many_effective_walls}
A set $\mathbb W$ of immersed walls in $X$ is \emph{spreading} if:
\begin{itemize}
  \item For arbitrarily large $L$, there exists $W\in\mathbb W$ with tunnel length $L$.
  \item $\mathbb W$ has the ladder overlap property of Definition~\ref{defn:ladder_overlap_property}.
\end{itemize}
$\widetilde X$ has \emph{many effective walls} if Conditions~\eqref{item:bust_point} and~\eqref{item:w_a_periodic} below hold.

\begin{enumerate}
 \item \label{item:bust_point}  For each regular $y\in V$, there exists a spreading set $\mathbb W$ such that for each $\epsilon>0$ and each $m\in\naturals$, there exists $L>m$ and $W\in\mathbb W$ with tunnel length $L$, a primary bust in each edge of $V$, and a primary bust in the $\epsilon$-neighborhood of $y$.


 \item \label{item:w_a_periodic}  For each $a\in\widetilde V_0$, whose image in $V$ is periodic and whose corresponding point in $\widetilde E_0$ is denoted by $a'$, there exists $k=k(a)\geq 0$ such that the following set $\mathbb W_a$ of immersed walls in $X$ is uniformly sub-quasiconvex and spreading: $\mathbb W_a$ is the set of $W$ such that each edge of $V$ contains a primary bust of $W$, and such that for each primary bust $d'$ of $\overline W\cap\widetilde E_0$ (corresponding to an interval $d\subset\widetilde V_0$) that is joined to $a'\in\widetilde E_0$ by a path in a knockout of $\overline W$, we have $\dist_{\widetilde X^1}(\tilde\phi^n(a),\tilde\phi^n(d))\geq3\delta+2\lambda$ for all $n\geq k$.  See Figure~\ref{fig:mew}.
\end{enumerate}
\end{defn}

\begin{figure}
\begin{overpic}[scale=0.10]{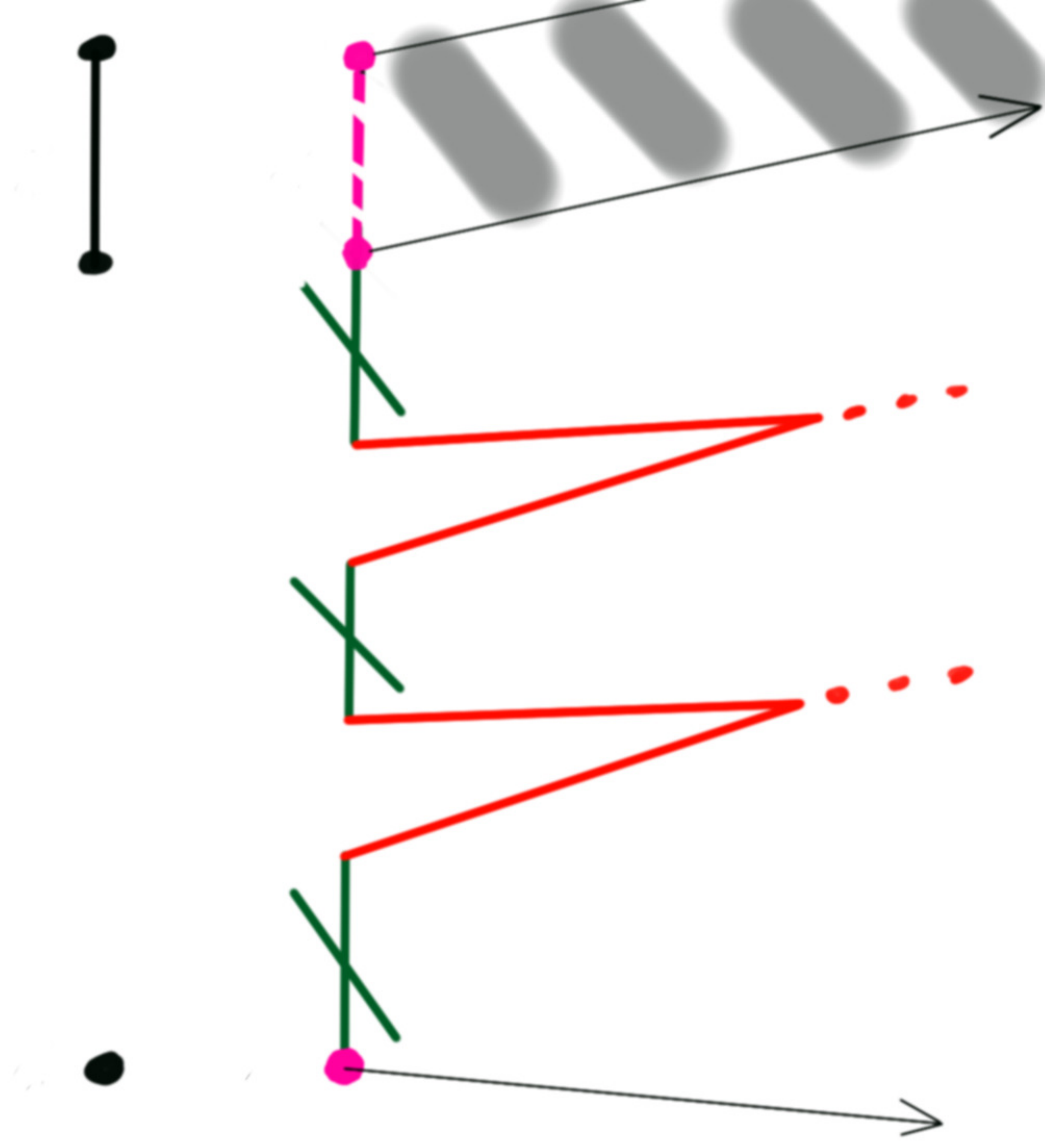}
\put(1,85){$d$}
\put(1,2){$a$}
\put(20,85){$d'$}
\put(19,2){$a'$}
\end{overpic}
\caption{Definition~\ref{defn:many_effective_walls}.\eqref{item:w_a_periodic}.}
\label{fig:mew}
\end{figure}

\begin{rem}\label{rem:uniform_k}
The constant $k$ in Definition~\ref{defn:many_effective_walls} can be chosen independently of the point $a$.  For each $a$, let $k'(a)$ be chosen so that for each bust $d$ and each $n\geq k'(a)$, we have $\dist_{\widetilde X^1}(\tilde\phi^n(a),\tilde\phi^n(d))\geq3\delta+2\lambda+1$.  This is possible since the existence of $k(a)$ implies that the forward rays emanating from $a$ and any point of $d$ diverge.  Fix $\epsilon\in(0,1)$ and let $U_a=(\tilde\phi^{k'(a)})^{-1}(\neb_{\epsilon}(\tilde\phi^{k'(a)}(a))$.  Then for each $b\in U_a$, we have $\dist_{\widetilde X^1}(\tilde\phi^n(b),\tilde\phi^n(d))\geq3\delta+2\lambda+1-\epsilon>3\delta+2\lambda$ for each $n\geq k'(a)$.  Hence we have $k(b)\leq k'(a)$.  The closure of the subset $V^{\circlearrowright}\subseteq V$ consisting of periodic points is compact since $V$ is compact.  Hence the claim follows since the open covering of $V^{\circlearrowright}$ produced above has a finite subcovering.

Similarly, since the collection $\mathbb W_a$ has the ladder overlap property, there is likewise a constant $B_a$ such that for all $W\in\mathbb W_a$, any two tunnels $T,T'$ of $\overline W$ that are joined by a path in $\overline W$ not traversing a slope have the property that the $(3\delta+2\lambda)$-neighborhoods of $\comapp(T),\comapp(T')$ intersect in a set of diameter at most $B_a$.  Let $V_a$ be the set of images in $V$ of points $b\in \widetilde V$ such that for all $W\in\mathbb W_a$, the point $b$ lies in a nucleus of some $\overline W$ and for all primary busts $d$ of $\overline W$, we have $\dist_{\widetilde X^1}(\tilde\phi^n(a),\tilde\phi^n(d))\geq3\delta+2\lambda$ for all $n\geq k$.  The previous argument showed that, with $k$ chosen appropriately, the set $V_a$ is open.  It follows that if $\widetilde X$ has many effective walls, the ladder overlap constant $B_a$ can be chosen independently of $a$.  This is used in the proof of Proposition~\ref{prop:cutting_ladderlike}.
\end{rem}

\begin{defn}[Separating level]\label{defn:separating_level}
$\widetilde X$ is \emph{$(M,K)$-separated} if for each $M$-deviating geodesic $\gamma$ there exists $y\in\widetilde X$ such that the following holds for all sufficiently large $n$: the set $\gamma\cap T^o_n(\tilde\phi^n(y))$ has odd cardinality, and the distance in $T^o_n(\tilde\phi^n(y))$ from $\gamma\cap T^o_n(\tilde\phi^n(y))$ to the root or to any leaf of $T^0_n(\tilde\phi^n(y))$ exceeds $M+K$.  We say $\widetilde X$ is \emph{level-separated} if it is $(M,K)$-separated for all $M>0,K\geq0$.
\end{defn}

\begin{rem}\label{rem:nearly_fixed}
If the level $T_n^o(\tilde\phi^n(y))$ separates $\gamma$ in the above sense, then we can choose $y$ so that the image $\bar y\in V$ of $y$ is not periodic.  Indeed, if $y',y$ are sufficiently close, then $T_n^o(\tilde\phi^n(y))$ and $T_n^o(\tilde\phi^n(y'))$ both separate $\gamma$.  There are points $y'$ arbitrarily close to $y$ whose images in $V$ are not periodic since there are only countably many periodic points.
\end{rem}

\begin{defn}[Bounded level-intersection]\label{defn:bounded_level_intersection}
$\widetilde X$ has \emph{bounded level-intersection} if for each $z\in\widetilde X^1$ and each vertical edge $e\subset\widetilde X^1$, there exists $K=K(z,e)$ such that for every level $T$ with a leaf at $z$, we have $|T\cap e|\leq K$.
\end{defn}

\begin{rem}\label{rem:bounded_level_intersection}
In the case of greatest interest, where $X$ is the mapping torus of a train track map, each level intersects each vertical edge in at most a single point, and hence $\widetilde X$ has bounded level-intersection.  This holds in particular for the complexes $\widetilde X$ considered in Theorem~\ref{thm:irreducible}.  More generally, this holds whenever there is a continuous map from $\widetilde X$ to an $\reals$-tree that is constant on levels and sends edges to concatenations of finitely many arcs.
\end{rem}

\begin{defn}[Exponentially expanding]\label{defn:exponentially_expanding}
The train track map $\phi:V\rightarrow V$ is \emph{exponentially expanding} if there exists an \emph{expansion constant} $\varpi>1$ such that for all edges $e$ of $V$ and all arcs $\alpha\subset e$, and all $L\geq 0$, we have $|\phi^L(\alpha)|\geq\varpi^L|\alpha|$.  Note that if $\phi$ is an irreducible train track map and edges are expanding, then $\phi$ is exponentially expanding, as can be seen by taking $\varpi$ to be the Perron-Frobenius eigenvalue of the transition matrix of $\phi$.  See Section~\ref{sec:forward_space_in_train_track_case} for more on the eigenvalues of the transition matrix.
\end{defn}

The main result of this section is:

\begin{prop}\label{prop:everything_cut}
Suppose that $\phi:V\rightarrow V$ is a $\pi_1$-injective train track map.  Let $X$ be the mapping torus of $\phi$.  Suppose that $\pi_1X$ is word-hyperbolic and that $\widetilde X$ satisfies:
\begin{enumerate}
 \item $\widetilde X$ is level-separated.
 \item $\widetilde X$ has many effective walls.
 \item Every finite forward path fellow-travels at uniformly bounded distance with a periodic forward path.
 \item $\phi$ is exponentially expanding.
\end{enumerate}
Then $G$ acts freely and cocompactly on a CAT(0) cube complex.
\end{prop}

\begin{proof}
Proposition~\ref{prop:cutting_ladderlike} shows that there exists $M$ such that every $M$-ladderlike geodesic is cut by a wall.  Proposition~\ref{prop:cutting_deviating} shows that each $M$-deviating geodesic is cut by a wall; Proposition~\ref{prop:cutting_deviating} requires $\widetilde X$ to have bounded level-intersection, which is the case since $\phi$ is a train track map.  The claim then follows from Proposition~\ref{prop:cutting_criterion} since each geodesic that is not $M$-ladderlike is by definition $M$-deviating.
\end{proof}

\begin{conv}
In the remainder of this section, $\phi:V\rightarrow V$ is assumed to satisfy the initial hypotheses of Proposition~\ref{prop:everything_cut}, except that the enumerated hypotheses will be invoked as needed.
\end{conv}

\subsection{Walls in $\widetilde X_L$}\label{subsec:walls_in_long_space}
Let $W\rightarrow X$ be an immersed wall with tunnel length $L\geq 1$, and suppose that $\overline W$ is a wall and $\comappn{\overline W}$ is $(\kappa_1,\kappa_2)$-quasi-isometrically embedded and $\kappa$-quasiconvex.  Each primary bust has regular endpoints, by Lemma~\ref{lem:choosing_busts_1}.\eqref{item:regular_endpoints}, so that each level-part of $\overline W$ is disjoint from $\widetilde X^0$.   Similarly,  $\widetilde X^0$ is disjoint from  $\comapp(S)$ for each slope $S$ of $\overline W$.

Recall that $\widetilde X^{\bullet}_L$ denotes the subdivision of $\widetilde X_L$ obtained by pulling back the 1-skeleton of $\widetilde X$.  For each $n\in\integers$, the inclusion $\widetilde V_{nL}\hookrightarrow\widetilde X$ lifts to an embedding $\widetilde V_{nL}\hookrightarrow({\widetilde X^\bullet}_L)^1$, and we continue to use the notation $\widetilde V_{nL}$ for this subspace.  We make the same observation and convention about $\widetilde E_{nL}$.  By translating, we can assume that $\overline W$ has a primary bust in $\widetilde V_0$, and hence all primary busts in $\overline W$ lie in the various $\widetilde V_{nL}$ and the map $\widetilde W\rightarrow\widetilde X$ lifts to $\widetilde W\rightarrow\widetilde X^{\bullet}_L$.  Let $\overline W_L$ be the image of $\widetilde W\rightarrow\widetilde X^{\bullet}_L$, so that we have the commutative diagram:
\begin{center}
$
\begin{diagram}
\node{\moustache{W}}\arrow{e}\arrow{s}\node{{\widetilde X^\bullet}_L}\arrow{s}\\
\node{\overline W}\arrow{e}\node{\widetilde X}
\end{diagram}
$
\end{center}

\renewcommand{\leftw}{\overleftarrow W_L}
\renewcommand{\rightw}{\overrightarrow W_L}

Note that $\moustache{W}$ and $\overline W$ are very similar: each tunnel $\moustache T$ of $\moustache{W}$ consists of a slope and a level-part that is a (subdivided) star, and $\overline W$ is obtained from $\moustache{W}$ by folding each such subdivided star into a tree.  The halfspaces $\leftw,\rightw$ in ${\widetilde X^\bullet}_L$ associated to $\overline W_L$ respectively map to the halfspaces $\overleftarrow W,\overrightarrow W$ in $\widetilde X$.

The approximation map $\comapp$ is defined in $\widetilde X_L$ just as it is in $\widetilde X=\widetilde X_1$. Consider $\comapp:\moustache{W}\rightarrow\widetilde X_L$, which is a lift of $\comapp:\overline W\rightarrow\widetilde X$.  There is a corresponding commutative diagram:
\begin{center}
$
\begin{diagram}
\node{\comapp(\moustache{W})}\arrow{e}\arrow{s}{}\node{{\widetilde X^\bullet}_L}\arrow{s}\\
\node{\comapp(\overline W)}\arrow{e}\node{\widetilde X}
\end{diagram}
$
\end{center}
in which the map $\comapp(\overline W_L)\rightarrow\comapp(\overline W)$ is an isomorphism.  
Thus $\comapp(\overline W)\rightarrow\widetilde X$ lifts to an embedding $\comapp(\overline W)\rightarrow\widetilde X^{\bullet}_L$ whose image is $\comapp(\overline W_L)$.  Figure~\ref{fig:juxtaposition} depicts $\moustache{W}$ and $\comapp(\overline W_L)$.

There is also a lift $\comappn{\overline W}\hookrightarrow{\widetilde X^\bullet}_L$.  Since $\comappn{\overline W}\hookrightarrow\widetilde X^1$ factors as $\comappn{\overline W}\hookrightarrow({\widetilde X^\bullet}_L)^1\rightarrow\widetilde X^1$ and since $({\widetilde X^\bullet}_L)^1\rightarrow\widetilde X^1$ is distance nonincreasing, $\comappn{\overline W}\rightarrow({\widetilde X^\bullet}_L)^1$ is a $(\kappa_1,\kappa_2)$-quasi-isometric embedding.  Thus $\visual\comappn{\overline W}$ embeds in $\visual{\widetilde X^\bullet}_L$ as a closed subset.

The following proposition explains that the tree $\comapp(\overline W_L)$ determines a wall in ${\widetilde X^\bullet}_L$, and therefore determines a coarse wall in $\widetilde X$ that coarsely agrees with $\overline W$.
\begin{figure}
\includegraphics[width=0.35\textwidth]{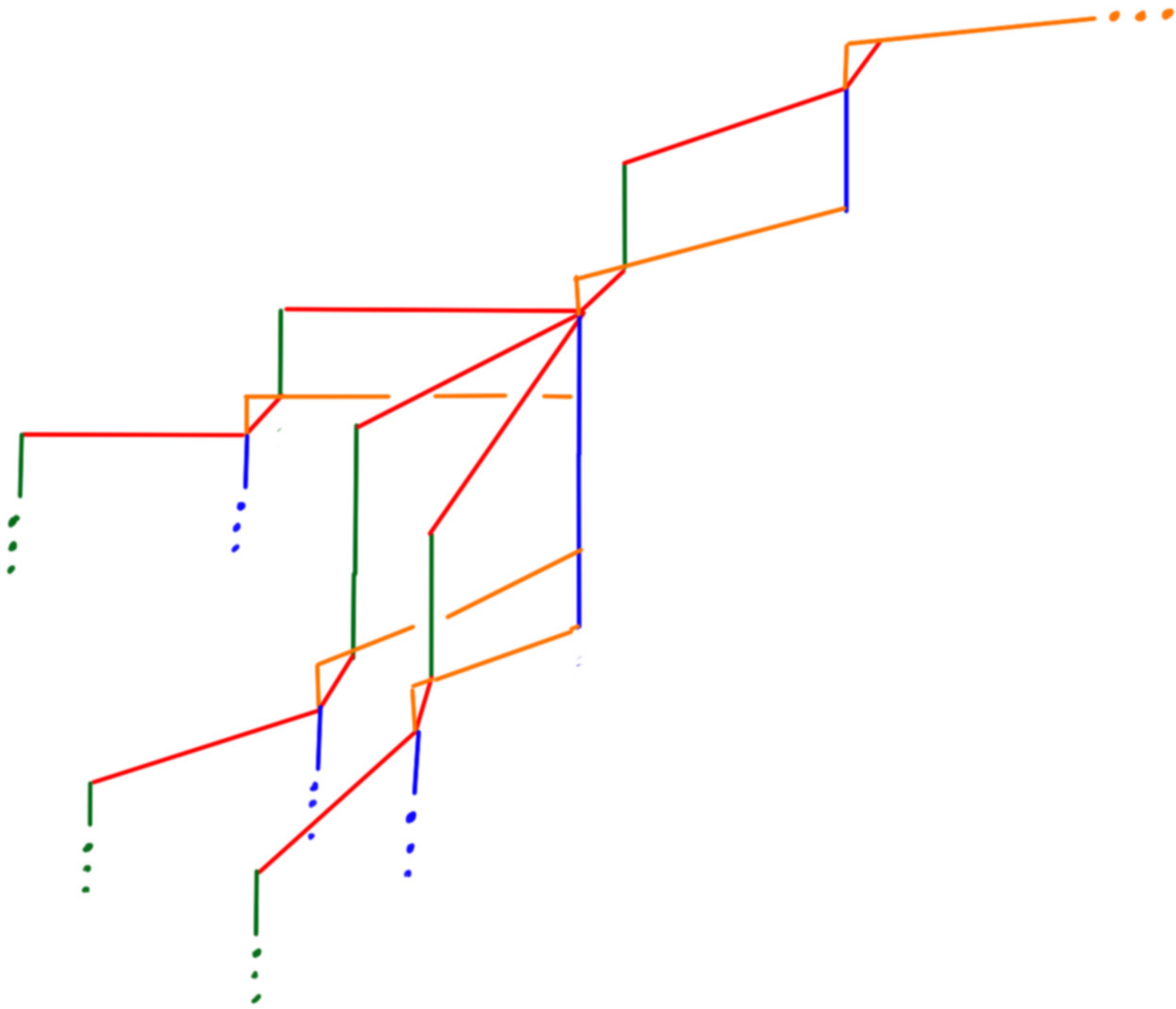}\\
\caption{$\moustache{W}$ and $\comapp(\overline W)$ inside ${\widetilde X^\bullet}_L$.}\label{fig:juxtaposition}
\end{figure}

\begin{prop}\label{prop:approximation_separates}
${\widetilde X^\bullet}_L$ contains subspaces $\lefta,\righta$ such that $\lefta\cup\righta={\widetilde X^\bullet}_L$ and $\lefta\cap\righta=\comapp(\overline W_L)$.  Both $\lefta-\comapp(\overline W_L)$ and $\righta-\comapp(\overline W_L)$ are connected.  Moreover, the images of $\lefta$ and $\righta$ under the map ${\widetilde X^\bullet}_L\rightarrow\widetilde X$ are coarsely equal to $\leftw$ and $\rightw$.
\end{prop}

\begin{proof}
It suffices to produce the subspaces $\lefta,\righta$ so that each is coarsely equal to a component of ${\widetilde X^\bullet}_L-\moustache{W}$.
Let $\leftw,\rightw$ be the closures of the components of ${\widetilde X^\bullet}_L-\moustache{W}$.  The halfspaces $\lefta$ and $\righta$ will be obtained from $\leftw$ and $\rightw$ by adding and subtracting ``discrepancy zones'', which are subspaces between $\moustache{W}$ and $\comapp(\moustache{W})$ suggested by Figure~\ref{fig:juxtaposition}.

\textbf{Discrepancy zones:}  Let $e\subset\widetilde V_{nL}$ be a primary busted edge with outgoing long 2-cell $R_e\subset{\widetilde X^\bullet}_L$.  Let $d\subset e$ be the closed primary bust with endpoints $p,q$.  Let $p',q'$ be the points at distance $\frac{1}{2}$ to the right of $p,q$ within $R_e$.  The slope $S$ travels from $p$ to $q'$, as shown in Figure~\ref{fig:upward_intermediate_zone}.  Let $Z^{\uparrow}$ be the 2-simplex in $R_e$ bounded by $S$ and the part of $\comapp(S)$ between $p$ and $q'$.  The disc $Z^{\uparrow}$ is an \emph{upward discrepancy zone}.

Let $\widetilde C\subset\widetilde E_{nL}$ be a nucleus in $\overline W_L$ and let $\comapp(\widetilde C)\subset\widetilde V_{nL+L}$ be its approximation.  Consider the map $\widetilde C\times[\frac{1}{2},L]\rightarrow{\widetilde X^\bullet}_L$ that restricts to the inclusion $\widetilde C\times\{t\}\hookrightarrow\widetilde V_{nL}\times\{t\}\subset{\widetilde X^\bullet}_L$ for $t<L$ and acts as the map $\tilde\phi^L:\widetilde C\rightarrow\widetilde V_{nL+L}$ on $\widetilde C\times\{L\}$.  The image of this map is a \emph{downward discrepancy zone} $Z^{\downarrow}$.  In other words, $Z^{\downarrow}$ is the closure of $\widetilde C\times[\frac{1}{2},L)$ in $\widetilde X_L^\bullet$.  See Figure~\ref{fig:downward_intermediate_zone}.

\begin{figure}[ht]
\begin{overpic}[scale=0.15]{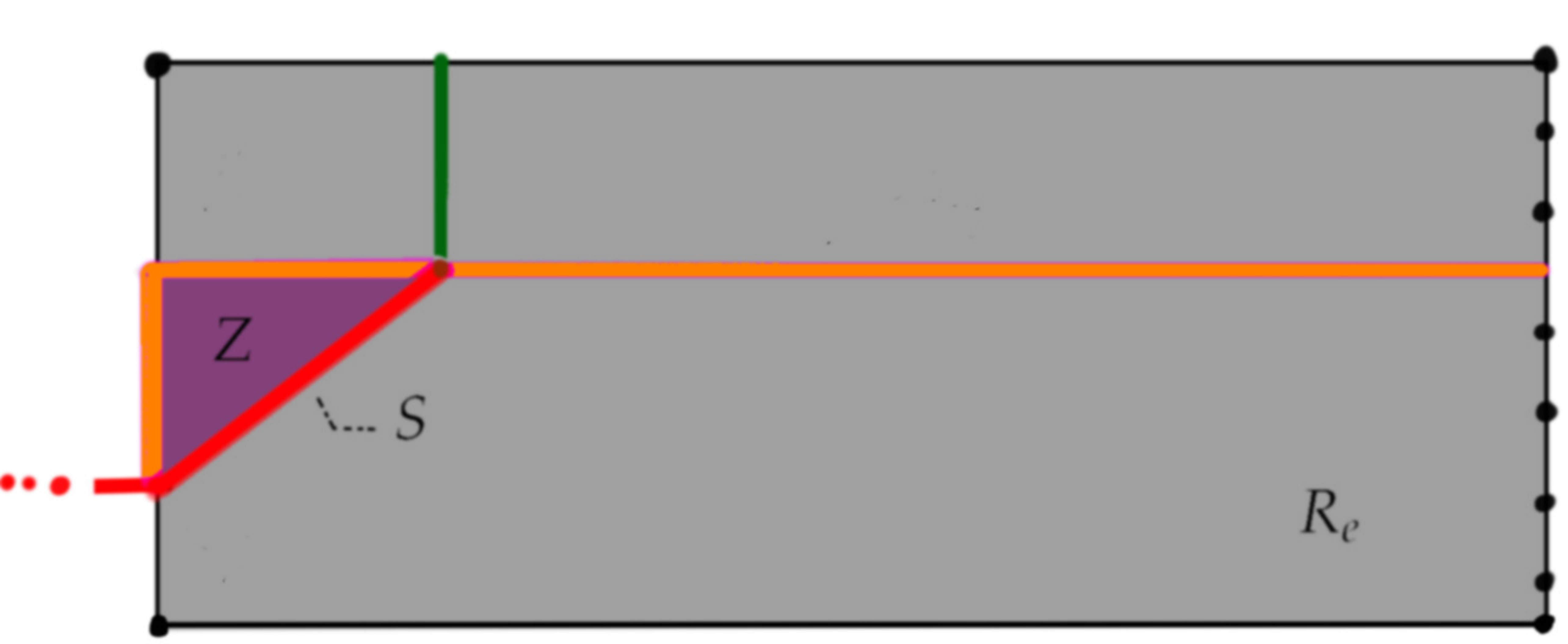}
\put(6,6){$p$}
\put(6,26){$q$}
\put(31,26){$q'$}
\end{overpic}
\caption{An upward discrepancy zone.}\label{fig:upward_intermediate_zone}
\end{figure}

\begin{figure}[ht]
\includegraphics[width=0.5\textwidth]{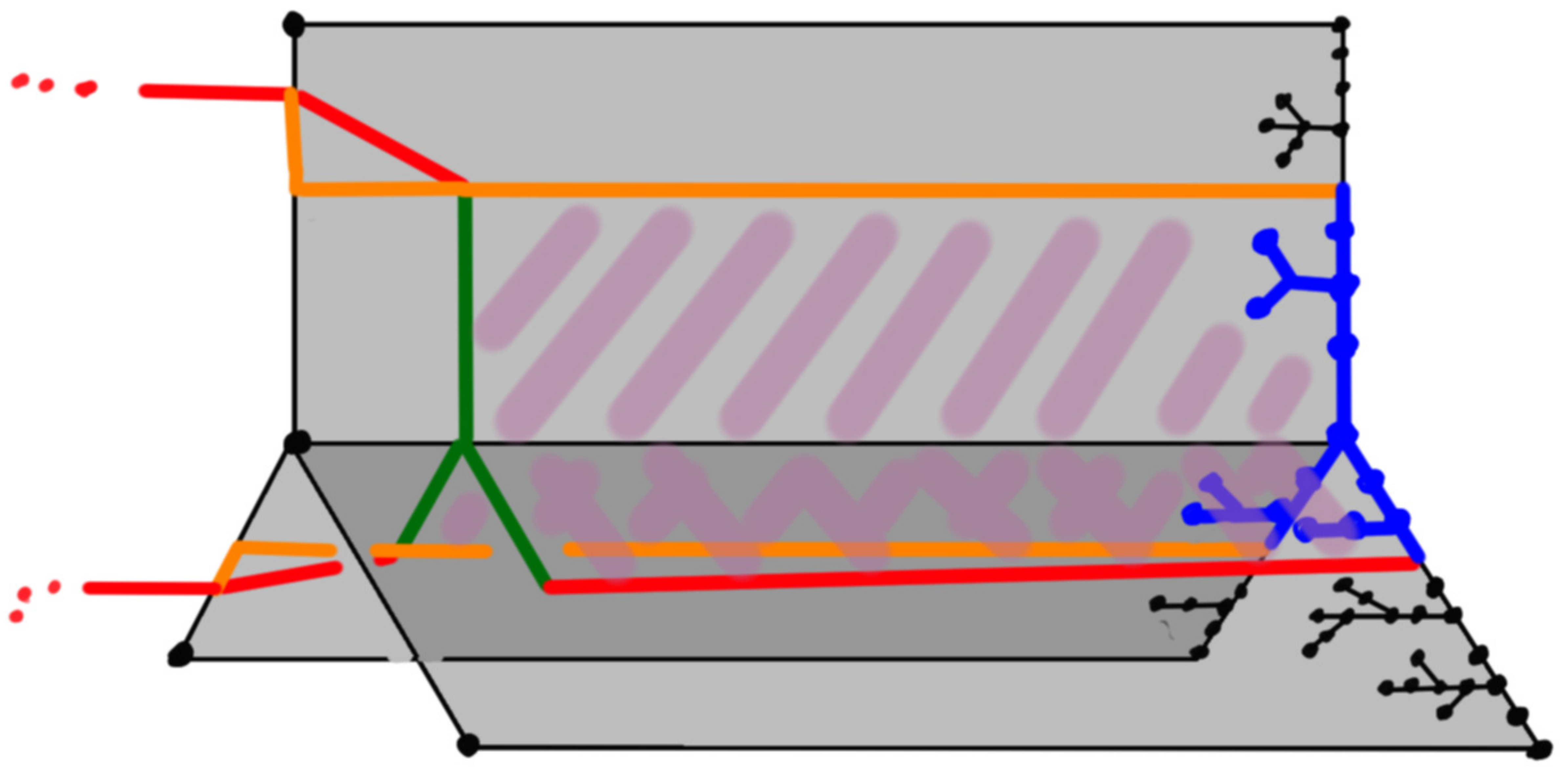}
\caption{A downward discrepancy zone is shaded.}\label{fig:downward_intermediate_zone}
\end{figure}

\textbf{The halfspaces $\lefta$ and $\righta$:}  Let $\mathfrak Z^{\uparrow}$ be the union of all upward discrepancy zones associated to $\comapp(\overline W_L)$, and likewise let $\mathfrak Z^{\downarrow}$ be the union of all downward discrepancy zones.  Let $$\lefta=\closure{\left(\leftw-\mathfrak Z^{\uparrow}\right)\cup\mathfrak Z^{\downarrow}}\text{\,\,\,\,\,\,and\,\,\,\,\,\,}\righta=\closure{\left(\rightw-\mathfrak Z^{\downarrow}\right)\cup\mathfrak Z^{\uparrow}}.$$
Since each discrepancy zone lies at distance less than $L$ from $\overline W_L$, we see that $\lefta$ coarsely equals $\leftw$.  By construction, $\lefta\cup\righta={\widetilde X^\bullet}_L$.  Finally, suppose that $x\in\lefta\cap\righta$.  Then $x$ must lie on the boundary of an discrepancy zone.  If $x\in\overline W_L$, and $x\in\mathfrak Z^{\uparrow}$, then $x\not\in\lefta$ unless $x\in\comapp(\overline W_L)\cap\overline W_L$.  Similarly, if $x\in\overline W_L$ and $x\in\mathfrak Z^{\downarrow}$, then $x\not\in\righta$ unless $x\in\comapp(\overline W_L)\cap\overline W_L$.  Hence $\lefta\cap\righta\subseteq\comapp(\overline W_L)$.  On the other hand, every point in $\comapp(\overline W_L)$ lies in the boundary of an discrepancy zone, and thus $\comapp(\overline W_L)\subseteq\lefta\cap\righta$.

Observe that $\lefta-\comapp(\overline W_L)$ is homeomorphic to $\leftw-\overline W_L$, which is connected.  Likewise $\left(\righta-\comapp(\overline W_L)\right)\cong\left(\rightw-\overline W_L\right)$.  Hence $\lefta-\comapp(\overline W_L)$ and $\righta-\comapp(\overline W_L)$ are connected.
\end{proof}

\renewcommand{\leftw}{\overleftarrow W}
\renewcommand{\rightw}{\overrightarrow W}

\subsection{Lifting and cutting geodesics in ${\widetilde X^\bullet}_L$}\label{subsec:lift_and_cut}
We now describe a criterion ensuring that a given geodesic in $\widetilde X$ is cut by a wall, in terms of quasigeodesics and walls $(\lefta, \righta)$ in ${\widetilde X^\bullet}_L$.

\subsubsection{Lifted augmentations of geodesics}\label{subsubsec:lifted_augmentations}
The following construction adjusts a bi-infinite quasigeodesic $\gamma\rightarrow\widetilde X^1$ so that it can be lifted to a bi-infinite quasigeodesic $\widehat{\gamma_{_{\succ}}}\rightarrow{\widetilde X^\bullet}_L$ such that $\gamma$ and $\widehat{\gamma_{_{\succ}}}$ determine the same pair of points in $\visual\widetilde X\cong\visual{\widetilde X^\bullet}_L$.

\begin{cons}[Lifted augmentations of quasigeodesics]\label{cons:LA}
Let $\gamma:\mathbf R\rightarrow\widetilde X^1$ be an embedded quasigeodesic.  The \emph{augmentation} $\gamma_{_{\succ}}$ of $\gamma$ is defined as follows.  For each (possibly trivial) bounded maximal horizontal subpath $P\subset\gamma$, with endpoints $p,p'\in\widetilde V_n,\widetilde V_{n'}$, let $n''$ be the smallest multiple of $L$ greater than or equal to $\max\{n,n'\}$ and let $p''=\tilde\phi^{n''-n}(p)=\tilde\phi^{n''-n'}(p')$.  Let $Q'$ be the horizontal path $pp''p'$, and replace $P$ by $Q'$.  Performing this replacement for each such $P$ yields $\gamma_{_{\succ}}$.  Note that $\gamma_{_{\succ}}$ is a quasigeodesic that $L$-fellowtravels with $\gamma$, so that $\visual\gamma_{_{\succ}}=\visual\gamma$.  We use the following notation.  First, $P=P_1P_2$, where $P_1$ and $P_2^{-1}$ are forward horizontal paths, one of which is trivial.  Then $Q'=P_1QP_2$, where $Q=Q_1Q_1^{-1}$, with $Q_1$ a forward path.  The terminal point $p''$ of $Q_1$ is the \emph{apex} of $Q$, and $Q=Q_1Q_1^{-1}$ is an \emph{augmentation} of $\gamma$.

The path $\gamma_{_{\succ}}$ lifts to a quasigeodesic $\widehat{\gamma_{_{\succ}}}\rightarrow {\widetilde X^\bullet}_L$.  More specifically, each lift of the union of the vertical edges of $\gamma_{_{\succ}}$ determines a unique lift of $\gamma_{_{\succ}}$ to a quasigeodesic.  Indeed, we can write $\gamma_{_{\succ}}$ in one of the following four forms:
\begin{enumerate}
 \item $\cdots A_{-1}e_{-1}B_{-1}A_0e_0B_0A_1e_1B_1A_2e_2B_2\cdots$, where $e_{\pm i}$ are present for all $i\in\naturals$;
 \item $A_0e_0B_0A_1\cdots$, where $A_0$ is unbounded;
 \item $\cdots A_0e_0B_0$, where $B_0$ is unbounded;
 \item $A_{s}e_sB_s\cdots A_{t}e_tB_t$ with $A_s,B_t$ unbounded.  (This includes the case $B_0A_1$ in which $\gamma=\gamma_{_{\succ}}$ is horizontal.)
\end{enumerate}
Each $A_i$ starts at an apex and each $B_i$ ends at an apex.  Observe that each lift of $e_i$ determines a lifts of $A_i$ and $B_i$ to $\widetilde X^{\bullet}_L$.  Since the apexes lift uniquely, any lift of $B_i$ is concatenable with any lift of $A_{i+1}$, and we conclude that a lift of $\{e_i\}$ induces a lift of $\gamma_{_{\succ}}$.  In the case where $\gamma$ is horizontal, $\gamma=\gamma_{_{\succ}}$ lifts uniquely since any horizontal path starting and ending in $\cup_{k}\widetilde V_{kL}$ lifts uniquely.  Under the quasi-isometry $(\widetilde X^\bullet_L)^1\rightarrow\widetilde X^1$, the quasigeodesic $\widehat{\gamma_{_{\succ}}}$ is sent to $\gamma_{_{\succ}}$, and thus $\visual\widehat{\gamma_{_{\succ}}}=\visual\gamma$.  Finally, if some augmentation of $\gamma$ has a subpath that lifts to a backtrack in $\widehat\gamma_{_{\succ}}$, then we truncate $\gamma_{_{\succ}}$ accordingly and define $\widehat\gamma_{_{\succ}}$ to be the lift of the truncated augmentation.  An augmentation where this truncation is nontrivial is a \emph{truncated augmentation}.   We call $\widehat{\gamma_{_{\succ}}}$ a \emph{lifted augmentation} of $\gamma$.
\end{cons}

\subsubsection{Cutting in ${\widetilde X^\bullet}_L$}\label{sec:general_cutting}
We now establish a criterion, in terms of lifted augmentations, ensuring that a wall $\overline W$ cuts a given quasigeodesic in $\widetilde X^1$.  We first require a classification of discrepancy zones.

\begin{defn}[Exceptional zone, narrow exceptional zones]\label{defn:narrow_intermediate_zones}
Let $W\rightarrow X$ be an immersed wall with tunnel-length $L$.  An \emph{exceptional zone} is a downward discrepancy zone in $\widetilde X_L^{\bullet}$ whose boundary path intersects the interior of a slope approximation.  The downward discrepancy zone shown in Figure~\ref{fig:downward_intermediate_zone} is exceptional. 

We say that $W$ has \emph{narrow exceptional zones} if for each exceptional zone $Z\subset\widetilde X^{\bullet}_L$, associated to a nucleus $\widetilde C$ of $\overline W_L$, the image in $Z\subset\widetilde X^\bullet_L$ of $\widetilde C\times[\frac{1}{2},\frac{3L}{4}]$ does not contain a vertex.
\end{defn}

\begin{lem}\label{lem:no_orange_orange}
Suppose that $\phi:V\rightarrow V$ is a train track map with expanding edges.  Suppose that $W\rightarrow X$ is an immersed wall such that every edge of $V$ contains a primary bust of $W$.  Then if the tunnel length $L$ of $W$ is sufficiently large, each exceptional discrepancy zone $Z$ lies in the interior of a single long 2-cell of $\widetilde X^{\bullet}_L$, and hence $Z$ intersects a single slope-approximation.
\end{lem}

\begin{proof}
Let $\comapp(S)$ be a slope-approximation in a long 2-cell $R$ based at the vertical edge $e\subset\widetilde V_n$.  We will show that for $L$ sufficiently large, the nucleus $\widetilde C$ incident to $S$ is the copy in $\widetilde E_n$ of a subinterval of the interior of $e$.  Let $\alpha$ be a component of $e-\interior{d}$, where $d$ is the primary bust associated to $S$.  Then for all sufficiently large $L$, the path $\phi^L(\alpha)$ traverses an entire edge, and therefore contains a primary bust.
\end{proof}





\begin{lem}\label{lem:exceptional_no_vertex}
Suppose that $\phi:V\rightarrow V$ is a train track map with exponentially expanding edges.  Let $y_1,\ldots,y_s\in V$ be regular points such that each edge of $V$ contains exactly one $y_i$, and let $\epsilon>0$.  Then for all sufficiently large $L$, there exists an immersed wall $W\rightarrow X$ with tunnel-length $L$, such that each primary bust is in the $\epsilon$-neighborhood of some $y_i$, and $W$ has narrow exceptional zones.
\end{lem}

\begin{proof}
Let $\varpi>1$ be the expansion constant of $\phi$.  For each $i$, let $y'_i\in V$ be a periodic regular point in the edge $e_i$ containing $y_i$ with $\dist_{e_i}(y'_i,y_i)<\frac{\epsilon}{2}$.  Let $\chi_i=\min\{\dist_V(\phi^k(y'_i),V^0)\co k\geq 0\},$ which is positive since $y'_i$ is periodic and regular.  Let $\chi=\max_i\chi_i$.  Let $$L>4(\log_{\varpi}\max_i|e_i|-\log_{\varpi}(\chi)).$$ For each $i$, let $d_i\subset\interior{e_i}$ be a primary bust chosen in the $\frac{\epsilon}{2\varpi^L}$-neighborhood of $y'_i$, and therefore in the $\epsilon$-neighborhood of $y_i$.  Lemma~\ref{lem:choosing_busts_1} ensures that this can be done in such a way as to yield an immersed wall $W\rightarrow X$.

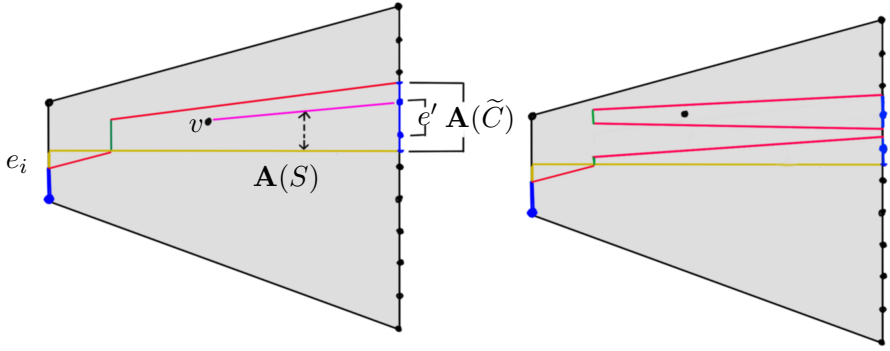
\begin{figure}[ht]
\begin{overpic}[scale=0.15]{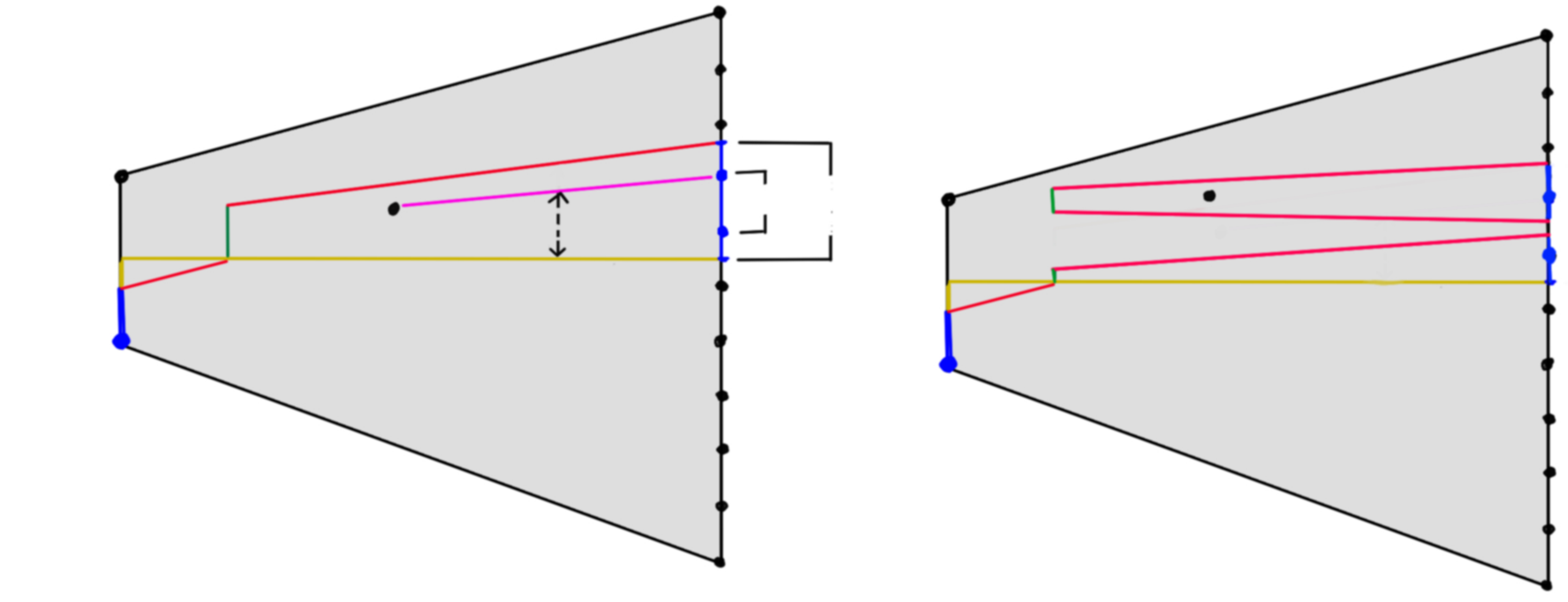}
\put(30,18){$\comapp(S)$}
\put(3,20){$e_i$}
\put(23,24){$v$}
\put(48,24.7){$e'$}
\put(51, 24.5){$\comapp(\widetilde C)$}
\end{overpic}
\caption{The exceptional zone corresponding to $\comapp(S)$ cannot contain the vertex $v$ when $L$ is sufficiently large.   Such a vertex $v$ could only be contained in a non-exceptional downward discrepancy zone, as shown at right.}\label{fig:exceptional_no_vertex}
\end{figure}

Let $Z$ be the image in $\widetilde X$ of an exceptional zone between $\overline W$ and $\comapp(\overline W)$.  By Lemma~\ref{lem:no_orange_orange}, there is a unique slope $S$ such that the forward part of $\comapp(S)$ forms part of the boundary path of $Z$.  See Figure~\ref{fig:exceptional_no_vertex}.  If $v\in Z$ is a vertex at horizontal distince more than $\frac{L}{4}$ from the nucleus-approximation $\comapp(\widetilde C)$ on the right of $Z$, then our choice of $L$ would ensure that the right boundary path of $Z$ contains a complete edge $e'$, and thus a primary bust, which is impossible.
\end{proof}

\begin{prop}\label{prop:general_cutting}
Let $\gamma:\mathbf R\rightarrow\widetilde X^1$ be an embedded quasigeodesic, and let $\widehat{\gamma_{_{\succ}}}\rightarrow{\widetilde X^\bullet}_L$ be a lifted augmentation.  Let $C_o$ be a bounded subset of $\widehat{\gamma_{_{\succ}}}\cap\comapp(\overline W_L)$, let $C$ be the smallest subgraph containing $C_o$.  Let $\widehat{\gamma_{_{\succ}}}\vee_C\comappn{\overline W_L}\rightarrow(\widetilde X^\bullet_L)^1$ be the graph obtained by wedging $\widehat{\gamma_{_{\succ}}}\rightarrow\widetilde X$ and $\comappn{\overline W_L}\rightarrow\widetilde X$  along the common subgraph $C$.  Suppose that:
\begin{enumerate}
 \item \label{item:wedgeqie}$\widehat{\gamma_{_{\succ}}}\vee_C\comappn{\overline W_L}\rightarrow(\widetilde X^\bullet_L)^1$ is a quasi-isometric embedding.
 \item \label{item:localsep}There are nontrivial intervals $f,f'\subset\widehat{\gamma_{_{\succ}}}$, immediately preceding and succeeding $C_o$ within $\widehat{\gamma_{_{\succ}}}$, that lie in $\lefta$ and $\righta$ respectively.
 \item \label{item:nosurprise}For every component $D$ of $\widehat{\gamma_{_{\succ}}}\cap\comapp(\overline W_L)$ disjoint from $C_o$, the 1-neighborhood in $\widehat{\gamma_{_{\succ}}}$ of $D$ lies entirely in $\lefta$ or $\righta$.
\end{enumerate}
Then $\overline W$ cuts $\gamma$.
\end{prop}

\begin{proof}
Hypotheses~(2)~and~(3) together imply that $\widehat{\gamma_{_{\succ}}}$ decomposes as a concatenation $\overleftarrow\gamma\bar\gamma\overrightarrow\gamma$, where $\bar\gamma$ is a bounded path containing $C_o$ and $\overleftarrow\gamma,\overrightarrow\gamma$ are rays contained in $\lefta,\righta$ respectively.  The image of $\widehat{\gamma_{_{\succ}}}\vee_C\comappn{\overline W_L}\rightarrow{\widetilde X^\bullet}_L$ is $\widehat{\gamma_{_{\succ}}}\cup\comappn{\overline W_L}$, which is quasi-isometrically embedded in $({\widetilde X^\bullet}_L)^1$ by hypothesis~(1).  The inclusion $\widehat{\gamma_{_{\succ}}}\cup\comappn{\overline W_L}\hookrightarrow({\widetilde X^\bullet}_L)^1$ thus induces an embedding $\visual\widehat{\gamma_{_{\succ}}}\sqcup\visual\comappn{\overline W_L}\rightarrow\visual{\widetilde X^\bullet}_L$.  The two points of $\visual\widehat{\gamma_{_{\succ}}}$ are $\visual\overleftarrow\gamma\in\visual\lefta$ and $\visual\overrightarrow\gamma\in\visual\righta$, and neither of these points lies in $\visual\comappn{\overline W_L}$ since hypothesis~(1) implies that no sub-ray of $\widehat{\gamma_{_{\succ}}}$ lies in a bounded neighborhood of $\comappn{\overline W_L}$.  Applying the quasi-isometry ${\widetilde X^\bullet}_L\rightarrow\widetilde X$ shows that the points of $\visual\gamma\subset\visual\widetilde X$ lie in $\visual N(\leftw)-\visual\overline W$ and $\visual N(\rightw)-\visual\overline W$, whence $\overline W$ cuts $\gamma$.
\end{proof}

\subsection{Cutting deviating geodesics}\label{sec:deviating}

\begin{prop}\label{prop:cutting_deviating}
Let $\widetilde X$ satisfy the hypotheses of Proposition~\ref{prop:everything_cut}, let $M\geq0$, and let $\gamma:\mathbf R\rightarrow\widetilde X^1$ be an $M$-deviating geodesic.  Then there exists a wall $\overline W$ such that $\comappn{\overline W}$ is quasiconvex in $\widetilde X^1$ and $\overline W$ cuts $\gamma$.
\end{prop}

\begin{proof}
We will find a wall $\overline W\rightarrow\widetilde X$ satisfying the hypotheses of Proposition~\ref{prop:general_cutting}.

\textbf{An oddly-intersecting forward path:}  Since $\widetilde X$ is level-separated, there exists $z\in\widetilde X$ such that for all sufficiently large $n$, there is a regular level $\mathcal T_n=T_n^o(\tilde\phi^n(z))$ with a leaf at $z$, such that $\mathcal T_n$ has odd intersection with $\gamma$ and the distance in $\mathcal T_n$ from $\gamma\cap\mathcal T_n$ to the root or to any leaf of $\mathcal T_n$ exceeds $12(M+\delta)$.

The fact that $\widetilde X$ has bounded level intersection and $\gamma$ is $M$-deviating implies that there exists $N$ and a finite, odd-cardinality set $C_o'\subset\gamma$ such that $\mathcal T_n\cap\gamma=C_o'$ for all $n\geq N$.  Each $\mathcal T_n$ is the union of finitely many maximal forward paths emanating from leaves.  For each $n\geq N$, we wish to choose a leaf $y$ of $\mathcal T_n$ such that the maximal forward path $\sigma_n\subset\mathcal T_n$ emanating from $y$ has the property that $\sigma_n\cap\gamma$ is a fixed odd-cardinality subset $C_o\subseteq C'_o$.  However, to achieve this, we shall slightly modify $\gamma$ as follows by replacing it with an embedded deviating uniform quasigeodesic that coincides with the original $\gamma$ outside a diameter-$2M$ subset.

We now describe the modification of $\gamma$.  Let $e_1,\ldots,e_{|C_o'|}$ be the edges of $\gamma$ intersecting $\mathcal T_n$ for $n\geq N$.  Index these so that $e_i$ precedes $e_j$ in the geodesic $\gamma$ if and only if $i<j$.  The set $\{e_1,\ldots,e_{|C_o'|}\}$ is partially ordered as follows: $e_i\preceq e_j$ if for some $\ell\geq0$, the path $\tilde\phi^\ell(e_i)$ traverses $e_j$.  The edges $e_i,e_j$ are \emph{confluent} if there exists $k$ such that $e_i,e_j\preceq e_k$.  Confluence is an equivalence relation, and there is exactly one confluence class for each $\preceq$-maximal edge.  Since $|C_o'|$ is odd, there exists an odd-cardinality confluence class, and we let $e_k$ be its $\preceq$-maximal element.  Let $e_i,e_j$ be the elements of the confluence class of $e_k$ such that the indices $i,j$ are respectively minimal and maximal.  Let $\alpha_i,\alpha_j$ be forward paths in $\mathcal T_n$ joining $e_i,e_j$ to $e_k$.  Let $A_i$ be an embedded combinatorial path in the forward ladder $N(\alpha_i)$ that joins the terminal vertex $v_i$ of $e_i$ to a vertex $v_k$ of $e_k$ and does not intersect $\alpha_i$.  The edge $e_j$ contains a vertex $v_j$ on the same side of $\mathcal T_n$ as the terminal vertex of $e_i$.  Let $A_j$ be an embedded combinatorial path in $N(\alpha_j)$ joining $v_j$ to $v_k$ and not intersecting $\alpha_j$.  Since $\gamma$ is deviating, $\dist(e_i,e_k)$ and $\dist(e_j,e_k)$ are uniformly bounded.  Hence, since $N(\alpha_i)^1,N(\alpha_j)^1$ are uniformly quasiconvex, the paths $A_i,A_j$ have uniformly bounded length.  Let $A$ be the path obtained from $A_iA_j^{-1}$ by removing backtracks.  We replace the subpath of $\gamma$ between $v_i$ and $v_j$ by $A$, and finally replace $\gamma$ by a bi-infinite embedded in the image of this path.  By construction, for all $n\geq N$, there is a forward path $\sigma_n$ of $\mathcal T_n$, that intersects the modified path exactly once, namely in a point of $e_i$.  The argument proceeds using the new $\gamma$, which is an embedded quasigeodesic that is $M$-deviating, with $M$ a new, larger constant.  However, since the quasi-isometry constants of $\gamma$ play no essential role in the argument, we will assume for simplicity that $\gamma$ remains a geodesic.

There exists $\epsilon>0$ such that for all $x\in\neb_{\epsilon}(y)$, any forward path $\sigma_x$ of length $n\geq N$ emanating from $x$ intersects $\gamma$ in a set $C^x_o$ of interior points of edges that has the same cardinality as $C_o$ and has the property that the smallest subcomplex $C$ containing $C_o$ is exactly the smallest subcomplex containing $C^x_o$.

The wall we will choose will contain a slope $S$ such that $\comapp(S)$ contains such a $\sigma_x$ as its forward part.

\textbf{Quasi-isometric embedding of $\gamma\vee_C\comappn{\overline W}\rightarrow\widetilde X^1$:}  Let $W\rightarrow X$ be an immersed wall such that every edge of $V$ contains a primary bust, and suppose $\overline W\subset\widetilde X$ is the image of a lift $\widetilde W\rightarrow\widetilde X$ such that $\overline W$ contains a slope $S$ with the forward part of $\comapp(S)$ equal to a path $\sigma_x$, emanating from some $x\in\neb_\epsilon(y)$, as above.  Suppose moreover that $W$ was drawn from a set of immersed walls with uniformly bounded ladder-overlap.

Since every edge contains a primary bust, Proposition~\ref{prop:quasiconvexity} provides constants $L_0,\kappa_1,\kappa_2$, depending only on $\widetilde X$, such that if the tunnel length of $W$ is at least $L_0$, then $\comappn{\overline W}$ is $(\kappa_1,\kappa_2)$ quasi-isometrically embedded.  Recall also that $\overline W$ is a genuine wall if the tunnel-length exceeds a uniform constant $L_1$, by Proposition~\ref{prop:tunneliswall}.

There exist constants $L_2\geq L_1,\kappa'_1,\kappa'_2$, depending on $\widetilde X$ and $M$ such that if $W$ has tunnel-length at least $L_2$, then $\gamma\vee_C\comappn{\overline W}\rightarrow\widetilde X^1$ is a $(\kappa_1',\kappa_2')$-quasi-isometric embedding.  Indeed, this follows from an application of Lemma~\ref{lem:RBRquasi}, since $\gamma$ is $M$-deviating and hence has uniformly bounded $(3\delta+2\lambda)$-overlap with $\comapp(S)$.

\textbf{Verification that $\gamma\vee_{C_o}\comapp(\overline W)$ embeds:}  By construction, $\gamma$ does not intersect any point of $\comapp(S)$ outside of $C_o$.  Hence suppose that $\tau\beta_1\alpha_1\cdots\beta_k\alpha_k\beta_{k+1}$ is a path in $\comappn{\overline W}\cup\gamma$ that begins and ends in $C_o$, such that: $\tau$ is a subpath of $\gamma$, and each $\beta_i$ lies in the carrier of a slope-approximation, and each $\alpha_i$ lies in a nucleus approximation, and $|\beta_i|\geq L$ except when $i=k+1$.  If $L$ is sufficiently large and $|\beta_2|=L$, then the existence of such a closed path contradicts the above conclusion that $\comappn{\overline W}\vee_C\gamma$ uniformly quasi-isometrically embeds.  The remaining possibility is that a path of the form $\tau\beta_1\alpha_1\beta_2$ or $\tau\beta_1\alpha_1$ is closed in $\widetilde X$.  In either case, when $L$ is sufficiently large, a thin quadrilateral argument shows that $\gamma$ is forced to $(2\delta+\lambda)$-fellow-travel with $\beta_1$ or $\beta_2$ for distance exceeding $M$, since the fellow-traveling between $\alpha_1$ and $\beta_i$ is controlled by Lemma~\ref{lem:vert_horiz_bound}.  Hence $\gamma\vee_{C_o}\comapp(\overline W)$ embeds in $\widetilde X$.

\textbf{Preventing short backtracks from crossing $\comapp(\overline W_L)$ at an apex:}  We now compute the tunnel-length $L_3\geq L_2$ necessary to ensure that each augmentation $QQ^{-1}$ in $\gamma_{_{\succ}}$ either fails to intersect $\comapp(\overline W)$ or has length at least $\frac{L}{4}$, where $L\geq L_3$ is the tunnel-length of $W$.  Note that if $QQ^{-1}$ is a truncated augmentation in the sense of Construction~\ref{cons:LA}, then the apex lies in some $\widetilde V_{n}$ with $n\not\in L\integers$, and hence $QQ^{-1}\cap\comapp(\overline W)=\emptyset$, so we only need consider non-truncated augmentations.

Let $W$ have tunnel-length $L\geq L_2$ and let $QQ^{-1}$ be a augmentation whose apex $p$ lies in $\comapp(\overline W)$, and hence in a nucleus-approximation.  Suppose that $|Q|\leq\frac{L}{4}$.  Let $\gamma'$ be the subpath of $\gamma$ between $C$ and the initial point of $Q$, let $\beta$ be a geodesic of $\comappn{\overline W}$ joining $p$ to the terminal point of $\comapp(S)$, and let $\tau$ be a geodesic of $\comappn{S}$ joining the initial point of $\gamma'$ to the terminal point of $\beta$.  Since $\gamma$ is deviating, the path $\gamma'Q$ is a quasigeodesic with constants depending only on $M$ and $\lambda$.  Meanwhile, since $\overline W$ has uniform ladder-overlap and $L\geq L_2\geq L_0$, the path $\beta\tau^{-1}$ is a $(\kappa_1,\kappa_2)$-quasigeodesic.  Hence $\gamma'Q$ fellow-travels with $\tau\beta^{-1}$ at distance depending only on $M$ and $\widetilde X$.  This is impossible for sufficiently large $L$, since $\gamma',\tau$ have $(2\delta+2\lambda)$-overlap of length at most $M$.

\textbf{Choosing $\overline W$:}  Since $\widetilde X$ has many effective walls, there exists an immersed wall $W\rightarrow X$ with tunnel length $L\geq L_3$, involving a primary bust in every edge of $V$, such that the image $\overline W$ of a lift $\widetilde W\rightarrow\widetilde X$ satisfies the following:
\begin{enumerate}
 \item $\overline W$ is a wall (since $L_3\geq L_2$).
 \item $\comapp(\overline W)$ is $(\kappa_1,\kappa_2)$-quasiconvex.
 \item $\gamma\cap\comapp(\overline W)=C$, which is contained in the carrier of a slope-approximation $\comapp(S)$.
 \item $\comappn{\overline W}\vee_C\gamma\rightarrow\widetilde X^1$ is a quasi-isometric embedding.
 \item Any augmentation $QQ^{-1}$ of $\gamma$ that intersects $\comapp(\overline W)$ has the property that $|Q|>\frac{L}{4}$ (since $L\geq L_3$).
\end{enumerate}

$W$ is chosen from the spreading set $\mathbb W$ given in Definition~\ref{defn:many_effective_walls}.(1).

\textbf{An arbitrary lifted augmentation:}  Let $\liftapp{\gamma}\rightarrow\widetilde X^{\bullet}_L$ be a lifted augmentation of $\gamma$.  Since the map $\widetilde X^{\bullet}_L\rightarrow\widetilde X$ is a quasi-isometry and restricts to the identity on $\comapp(\overline W_L)$ and sends $\liftapp{\gamma}$ to $\gamma_{_{\succ}}$, the intersection $\widehat C=\liftapp{\gamma}\cap\comappn{\overline W_L}$ is bounded and $\liftapp{\gamma}\vee_{\widehat C}\comappn{\overline W_L}\rightarrow\widetilde X^{\bullet}_L$ is a quasi-isometric embedding.  (We could have chosen a specific lifted augmentation to make $\widehat C\neq\emptyset$, but it follows from the discussion below that this holds for any lifted augmentation.)  Thus any $\liftapp{\gamma}$, together with $\comapp(\overline W_L)$, satisfies Hypothesis~\eqref{item:wedgeqie} of Proposition~\ref{prop:general_cutting}.

We now verify that $\liftapp{\gamma}$ satisfies the remaining two hypotheses of Proposition~\ref{prop:general_cutting}.  To this end, let $\hat\eta$ be an embedded quasigeodesic in $\widetilde X^{\bullet}_L$ obtained from $\tilde\phi_L\circ\liftapp{\gamma}$ by removing backtracks, and let $\eta$ be the image in $\widetilde X$ of $\hat\eta$.  Note that $\hat\eta$ is independent of the choice of lifted augmentation of $\gamma$.

\textbf{Intersection of $\hat\eta$ with $\comapp(\overline W)$:}  We first work in $\widetilde X$.  Recall that $\gamma\cap\comapp(\overline W)$ is the odd-cardinality set $C_o$ of points in $\comapp(S)$ for some slope $S$ of $\overline W$.  Consider the nucleus $\widetilde M\subset\overline W$ that intersects $S$ and $\comapp(S)$.  Then, since we can assume that $L$ is sufficiently large to ensure that each primary bust is separated from each vertex by a secondary bust, $\widetilde M$ corresponds to a subinterval containing no vertex, and hence $\comapp(\widetilde M)$ maps to a subspace of the star of a vertex in $V$.  By Lemma~\ref{lem:exceptional_no_vertex} and our above choice of $L$, the exceptional zone determined by $\widetilde M$ and $\comapp(\widetilde M)$ contains no vertex of a vertical edge containing a point of $C_o$.  It follows that the tunnel $T'$ attached to the unique secondary bust of $\widetilde M$ intersects $\gamma$ in a set of points corresponding bijectively to $C_o$.  Let $S'$ be the slope of $T'$.  Then, since the primary busts can be chosen arbitrarily small, we can assume that $\comapp(S')\cap\eta$ is an odd-cardinality set $E_o'$.  Hence $\comapp(S')\cap\hat\eta\subset\widetilde X^\bullet_L$ is an odd-cardinality set $E_o$ mapping bijectively to $E_o'$.

Since the endpoints of primary busts are regular, the path $\hat\eta$ contains a nontrivial interval $I$ ending at a point of $E_o$ and lying in the image of $S'$ under the forward flow; hence $I\subset\rightw_L$ and, since $I$ does not lie in a discrepancy zone, $I\subset\righta$.  There is likewise a nontrivial interval $I'$ in $\hat\eta$ beginning on $E_o$ and lying in an exceptional zone determined by the nucleus intersecting $\comapp(S')$.  Moreover, $I'$ and $I$ can be chosen to be separated in $\hat\eta$ by $E_o$.  See Figure~\ref{fig:repair}.

\begin{figure}[h]
\begin{overpic}[width=0.6\textwidth]{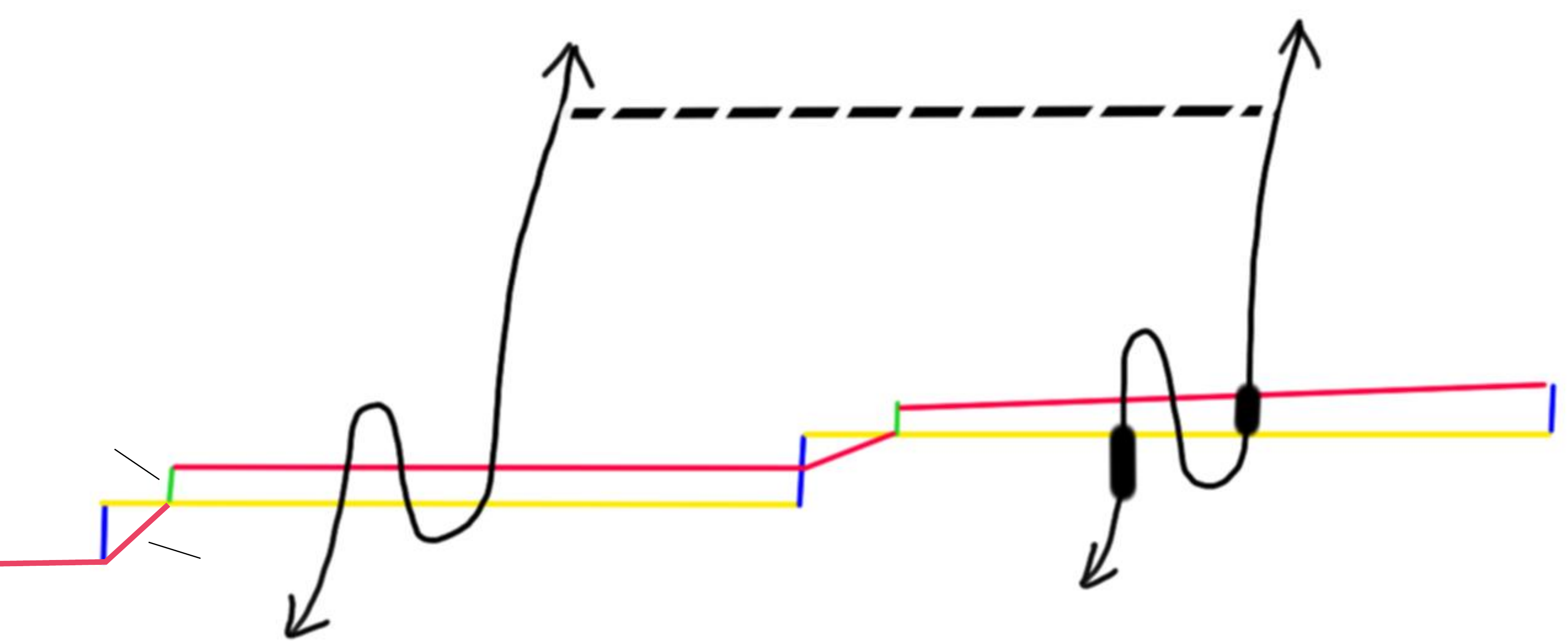}
\put(12,4){$S$}
\put(-5,3){$\overline W$}
\put(3,11){$\widetilde M$}
\put(34,4){$\comapp(S)$}
\put(84,9){$\comapp(S')$}
\put(30,22){$\gamma$}
\put(77,22){$\eta$}
\put(82,33){$p$}
\put(50,35){$\upsilon$}
\put(32,32){$p'$}
\end{overpic}
\caption{The relationship between $\gamma,\eta,\overline W,\comapp(\overline W)$ in $\widetilde X$.  The intervals $I,I'$ in $\widetilde X^\bullet_L$ map to the bold intervals.  The path $\theta$ contains the terminal part of $\comapp(S)$ and the initial part of $\comapp(S')$.}\label{fig:repair}
\end{figure}
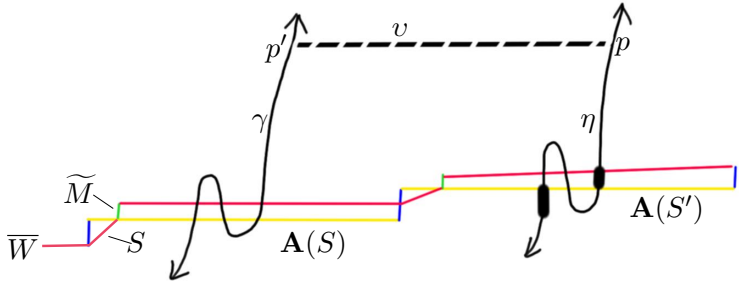

We claim that $\comapp(\overline W)\cap\hat\eta=E_o$.  Otherwise, applying the map $\widetilde X^\bullet_L\rightarrow\widetilde X$ would show that $\comapp(\overline W)\cap\eta$ contains some point $p\not\in E'_o$, since $\widetilde X^\bullet_L\rightarrow\widetilde X$ restricts to a bijection on $\comapp(\overline W)$.  Then there is a forward path $\upsilon$ of length $L$ emanating from a point $p'\in\gamma$ and terminating at $p$.  Let $\theta$ be a geodesic of $\comapp(\overline W)$ joining $p$ to a closest point $a$ of $C_o$, and let $\gamma_1$ be the subpath of $\gamma$ joining $a$ to $p'$.  Then $\gamma_1\upsilon$ is a quasigeodesic with quasi-isometry constants depending only on the deviation constant of $\gamma$, while $\theta$ is a $(\kappa_1,\kappa_2)$-quasigeodesic.  Hence $\gamma_1\upsilon$ fellowtravels with $\theta$ at distance depending only on $\delta,M,\kappa_1,\kappa_2$ and not on $L$.  It follows that there is a uniform upper bound on $|\gamma_1|$ that is independent of $L$.  Hence, if $L$ is sufficiently large, then since $C_o$ lies at distance at least $\frac{L}{4}$ from all nuclei, $\min_n|q(p)-nL|\geq\frac{L}{4}$, so that $p$ lies at horizontal distance at least $\frac{L}{4}$ from any nucleus approximation.  Suppose that $\comapp(S')$ lies in $\theta$.  Then $\theta$ contains a point at distance at least $\frac{L}{4}$ from $\gamma_1\upsilon$, and hence $\gamma_1\upsilon$ and $\theta$ cannot uniformly fellow-travel when $L$ is sufficiently large.  Similarly, if $\theta$ enters some other slope-approximation attached to $\comapp(\widetilde M)$, we find that $\theta$ and $\gamma_1\upsilon$ cannot fellow-travel.  The remaining possibility is that there is a path in $\comapp(\widetilde M)$ joining the endpoint of $\comapp(S)$ to a point of $\upsilon$.  This is impossible since, by Remark~\ref{rem:nearly_fixed}, the level part of $T'$ has odd-cardinality intersection with $\gamma$ and $p\not\in E'_o$.

\textbf{Conclusion:}  It follows from the above discussion that $\hat\eta$ contains two quasigeodesic rays, one in each of the halfspaces of $\widetilde X^\bullet_L$ associated to $\comapp(\overline W)$.  Since $\hat\eta$ fellow-travels with $\liftapp{\gamma}$, we see that $\liftapp{\gamma}$ satisfies all hypotheses of Proposition~\ref{prop:general_cutting}, whence $\overline W$ cuts $\gamma$.
\end{proof}
\subsection{Cutting ladderlike geodesics}\label{subsec:ladderlike_case}

\begin{prop}\label{prop:cutting_ladderlike}
Suppose that $\widetilde X$ has many effective walls and for each bounded forward path $\alpha$ there exists a periodic regular forward path $\alpha'$ such that $N(\alpha)=N(\alpha')$.

Then for each geodesic $\gamma:\mathbf R\rightarrow\widetilde X^1$ that is not $M$-deviating for any $M$, there exists an immersed wall $W\rightarrow X$ such that $\overline W$ is a wall, $\comappn{\overline W}$ is quasiconvex, and $\overline W$ cuts $\gamma$.
\end{prop}

\begin{proof}
Suppose that $\gamma$ contains a path $\gamma'$ such that for some regular $x\in\widetilde X^1$ and some $M$ to be determined, the path $\gamma'$ fellowtravels at distance $(2\delta+2\lambda)$ with the sequence $x,\tilde\phi(x),\ldots,\tilde\phi^M(x)$, where $x$ is a periodic point.  Such $\gamma'$ exists for arbitrarily large $M$ by combining the fact that $\gamma$ is $M$-ladderlike for arbitrarily large $M$ with the first hypothesis.  We shall show that if $M$ is sufficiently large, then there exists a wall $\overline W$ that has the desired properties and cuts $\gamma$ and separates $x$ and $\tilde\phi^M(x)$.

\textbf{Choosing $W$ using many effective walls:}  Without loss of generality, $M$ is an even integer, and we let $a=\tilde\phi^{M/2}(x)$.  Note that $a$ is periodic.  Let $\{e_i\}$ be the collection of edges of $V$, and let $W\rightarrow X$ be an immersed wall busting each $e_i$, with tunnel-length $L$ to be determined.  Let $e_1$ be the edge whose interior contains $a$.  By Remark~\ref{rem:uniform_k} and the fact that $\widetilde X$ has many effective walls, there exist $\kappa_1,\kappa_2, L_1$ depending only on $\widetilde X$ such that we can choose $W$ with tunnel length $L\geq L_1$ so that $\overline W$ is a wall and $\comappn{\overline W}$ is $(\kappa_1,\kappa_2)$-quasi-isometrically embedded.  Moreover, we choose $W$ from the the collection $\mathbb W_a$ of Definition~\ref{defn:many_effective_walls}.\eqref{item:w_a_periodic}, which guarantees that $\overline W$ can be chosen with the following properties:
\begin{enumerate}
 \item There exists $k\geq0$ such that for each primary bust $d$ with an endpoint in $\overline W$ in the same knockout as $a$, we have $\dist(\tilde\phi^n(a),\tilde\phi^n(d))\geq3\delta+2\lambda$ for all $n\geq k$.\label{eq:k}
 \item $W$ has tunnel length $L>\max\{12(\delta+k),L_1\}$, independent of $M$.
 \item The image of $a$ in $V$ lies in the interior of a nucleus of $W$ and so $\comapp(\overline W)$ contains $\tilde\phi^L(a)$.
\end{enumerate}
We assume that $M>JL$, where $J\geq 4$ will be chosen below.  Let $\sigma$ be the uniform quasigeodesic in $\widetilde X^1$ obtained from $\gamma$ by removing $\gamma'$ and replacing it by the sequence $x,\ldots,\tilde\phi^M(x)$.

\textbf{Verifying that $\sigma\vee_{\phi^L(a)}\comappn{\overline W}$ quasi-isometrically embeds:}  Consider paths of the form $\alpha_0\beta_0\cdots\beta_{s-1}\alpha_s\tau$, where $\beta_i$ is a geodesic of the carrier of a slope-approximation, $\alpha_i$ is a vertical geodesic of $\comappn{\overline W}$, and $\alpha_s$ terminates at $\tilde\phi^L(a)$, and $\tau$ is a subpath of $\sigma$ beginning at $\tilde\phi^L(a)$ (actually, the initial part of $\tau$ is a subsequence of $x,\tilde\phi(x),\ldots\tilde\phi^{\frac{M}{2}+L}(x)$ or of $\tilde\phi^{\frac{M}{2}+L}(x),\ldots\tilde\phi^M(x)$).  The $(3\delta+2\lambda)$-overlap between $\alpha_s$ and $\tau$ and between $\alpha_i$ and $\beta_{i}$ and between $\alpha_i$ and $\beta_{i-1}$ is controlled by Lemma~\ref{lem:vert_horiz_bound}, and Condition~\eqref{eq:k} on $\overline W$ ensures that the $(3\delta+2\lambda)$ overlap between $\tau$ and $\beta_{s-1}$ has length at most $k$.  The choice of $L$ now allows us to invoke Lemma~\ref{lem:RBRquasi} to conclude that $\sigma\vee_{\tilde\phi^L(a)}\comappn{\overline W}$ is quasi-isometrically embedded in $\widetilde X^1$, with constants $(\kappa_1',\kappa_2')$ depending only on $\kappa_1,\kappa_2,\lambda$.

\textbf{Verifying that $\sigma\vee_{\tilde\phi^L(a)}\comapp(\overline W)$ embeds:}  We will show that there is no path $\tau\subset\sigma$ beginning at $\tilde\phi^L(a)$ and joining the endpoints of a path $\alpha_0\beta_0\cdots\alpha_{m}$ or $\alpha_0\beta_0\cdots\alpha_m\beta_m$ in $\comappn{\overline W}$ with each $\alpha_i$ vertical, and each $\beta_i$ a path in the carrier of a slope approximation.  Each $\beta_i$ has length $L$ except for the $\beta_m$ in the path of the second form.  Since $\sigma\vee_{\tilde\phi^L(a)}\comappn{\overline W}$ is quasi-isometrically embedded, it suffices to examine the case where $m\leq 1$.  The quadrilateral $\alpha_0\beta_0\alpha_1\tau^{-1}$ is approximated by a quasigeodesic quadrilateral $\bar\alpha_0\beta_0\bar\alpha_1\tau^{-1}$, where each $\bar\alpha_i$ is a geodesic of length exceeding $3\delta+2\lambda$.  This quadrilateral is $2\delta+2\lambda$ thin, and $\beta_0,\tau$ fellow travel at distance $2\delta+2\lambda$ for length at most $k$.  Hence, without loss of generality, $\bar\alpha_0$ must fellow-travel with $\beta_0$ at distance $2\delta+\lambda$ for distance at least $\frac{11L}{24}$, whence $\alpha_0$ must $(2\delta+2\lambda+\mu)$-fellow-travel with $\beta_0$ for distance at least $\frac{11L}{24\mu_1}-\mu_2$, which contradicts Lemma~\ref{lem:vert_horiz_bound} when $L$ is sufficiently large.  (Recall that the quasiconvexity constant $\mu$ and the quasi-isometry constants $(\mu_1,\mu_2)$ of the nucleus approximations are independent of $L$.)  Hence $\sigma\vee_{\tilde\phi^L(a)}\comapp(\overline W)$ embeds.

\textbf{Applying Proposition~\ref{prop:general_cutting}:}  Let $\liftapp{\sigma}$ be a lifted augmentation of $\sigma$ induced by a lift of the forward path joining $x$ to $\phi^M(x)$.  Let $\hat x_i$ denote the lift of $\tilde\phi^i(x)$, so that $\hat x_{M/2}$ is a lift of $a$ and $\hat x_{M/2+L}$ is a lift of $\tilde\phi^L(a)$.  The wall $\overline W$ is the image of a wall $\overline W_L$ such that a nucleus of $\overline W_L$ separates $\hat x_{M/2}$ from $\hat x_{M/2+1}$ and thus $\hat x_{M/2+L}$ lies in a nucleus approximation of $\comapp(\overline W_L)$.  (Recall that $\comapp(\overline W_L)$ maps isomorphically to $\comapp(\overline W)$.)  Thus the points $\hat x_{M/2+L\pm1}$ lie in distinct halfspaces associated to $\comapp(\overline W_L)$.  Indeed, $\hat x_{M/2+L-1}$ lies in a downward discrepancy zone and hence in $\lefta$, while $\hat x_{M/2+L+1}\in\righta$.  See Figure~\ref{fig:ladderlike_cut}.  This verifies Hypothesis~\eqref{item:localsep} of Proposition~\ref{prop:general_cutting}.

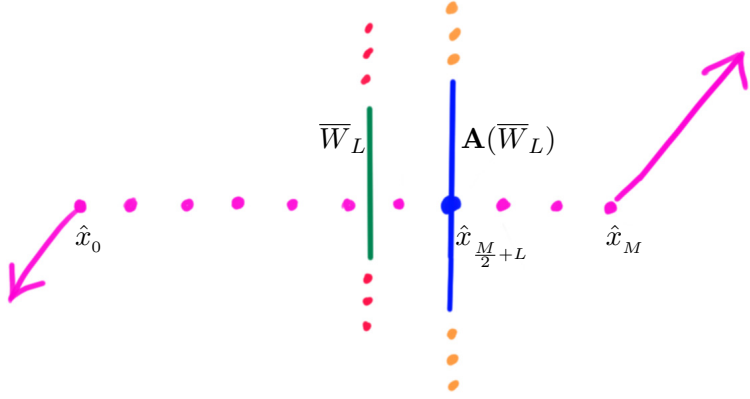
\begin{figure}
\begin{overpic}[scale=0.25]{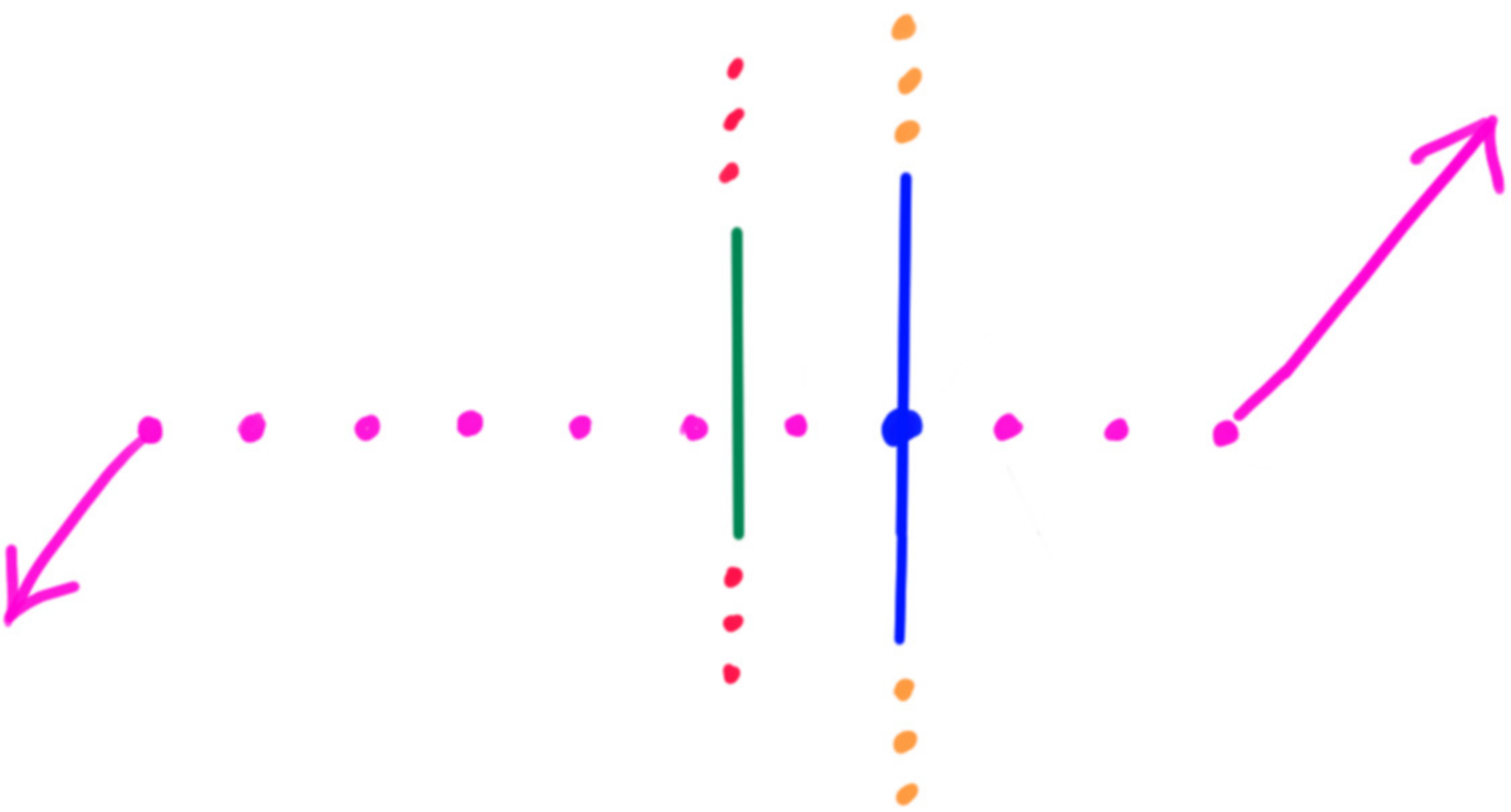}
 \put(9.2,20){$\hat x_{_{0}}$}
 \put(60.3,20){$\hat x_{_{\frac{M}{2}+L}}$}
 \put(80.5,20){$\hat x_{_M}$}
 \put(42,33){$\overline W_L$}
 \put(61,33){$\comapp(\overline W_L)$}
\end{overpic}
\caption{Notation in the proof of Proposition~\ref{prop:cutting_ladderlike}.}\label{fig:ladderlike_cut}
\end{figure}

As in the proof of Proposition~\ref{prop:cutting_deviating}, the fact that $\sigma\vee_{\tilde\phi^L(a)}\comappn{\overline W}$ is quasi-isometrically embedded, together with the fact that $\widetilde X^{\bullet}_L\rightarrow\widetilde X$ is a quasi-isometry, shows that $\liftapp{\sigma}\vee_{\hat x_{M/2+L}}\comappn{\overline W_L}\rightarrow\widetilde X^{\bullet}_L$ is a quasi-isometric embedding.  This verifies Hypothesis~\eqref{item:wedgeqie} of Proposition~\ref{prop:general_cutting}.

Let $y\in\liftapp{\sigma}\cap\comapp(\overline W_L)$.  Then either $y$ maps to a point of $\sigma\cap\comapp(\overline W)$, in which case $y=\tilde\phi^L(a)$ since $\sigma\vee_{\tilde\phi^L(a)}\comapp(\overline W)$ embeds in $\widetilde X$, or $y$ is an apex of $\liftapp{\sigma}$.  The latter is impossible provided $J$ is sufficiently large compared to $\kappa_1',\kappa_2'$.  Indeed, suppose $QQ^{-1}$ is an augmentation beginning on $\sigma$ and having an apex $p\in\comapp(\overline W)$.  Let $\beta\rightarrow\comappn{\overline W}$ join $p$ to $\tilde\phi^L(a)$, let $\tau\rightarrow N(\sigma'')$ join $\tilde\phi^L(a)$ to $x$, and let $P$ be the subpath of $\sigma$ subtended by $x$ and the initial point of $Q$.  Then the concatenation $P\tau^{-1}\beta^{-1}$ is a $(\kappa_1',\kappa_2')$-quasigeodesic containing a subpath of length at least $(J/2-1)L$, namely $\tau$.  Hence if $J>2(\kappa_1'(L+\kappa_2')L^{-1}+1)$, then the offending apex $p$ cannot exist since $|Q|\leq L$.  This verifies Hypothesis~\eqref{item:nosurprise} of Proposition~\ref{prop:general_cutting}, and the proof is complete.
\end{proof}

\section{Leaf-separation and many effective walls in the irreducible case}\label{sec:getting_lol}
In this section, we describe conditions on $\phi$ ensuring that $\widetilde X^1$ satisfies the hypotheses of Proposition~\ref{prop:everything_cut}.

\subsection{Leaves}\label{subsec:leaves_and_train_tracks}

\begin{defn}[Leaf]\label{defn:leaf}
Let $x,y\in\widetilde X$.  Then $x,y$ are \emph{equivalent} if there exist forward paths $\sigma_x,\sigma_y$ such that $x\in\sigma_x,y\in\sigma_y$ and $\sigma_x\cap\sigma_y\neq\emptyset$.  An equivalence class is a \emph{leaf}.
We denote the leaf containing $x$ by $\mathcal L_x$.  The leaf $\mathcal L_x$ is \emph{singular} if it contains a 0-cell, and otherwise $\mathcal L_x$ is \emph{regular}.
\end{defn}

Observe that $\mathcal L_x$ is $\tilde\phi$-invariant.  Moreover, observe that $\mathcal L_x$ has a natural directed graph structure: vertices are points of $\mathcal L_x\cap\widetilde X^1$, and edges are midsegments.  From Proposition~\ref{prop:leaf_properties}.\eqref{leaf_neighborhood} and Proposition~\ref{prop:properties_of_levels}, it follows that this subdivision makes $\mathcal L_x$ a directed tree in which each vertex has exactly one outgoing edge and finitely many incoming edges. 

\begin{prop}[Properties of leaves]\label{prop:leaf_properties}
Leaves have the following properties:
\begin{enumerate}
 \item If $\mathcal L_x$ is a regular leaf and $\phi$ is a train track map, then $\mathcal L_x$ has a neighborhood homeomorphic to $\mathcal L_x\times[-1,1]$ with $\mathcal L_x$ identified with $\mathcal L_x\times\{0\}$.\label{leaf_neighborhood}
 \item Each level is contained in a unique leaf, and $\mathcal L_x$ is an increasing union of levels.\label{leaf_level}
\end{enumerate}
\end{prop}

\begin{proof}
\textbf{Proof of~$\eqref{leaf_neighborhood}$:}  This uses Lemma~\ref{lem:finite_intersection_leaf_edge} below.  For each vertex $v_{\tilde e}=\mathcal L_x\cap\tilde e$ of $\mathcal L_x$, let $U_{\tilde e}$ be an open interval in $\tilde e$ about $v_{\tilde e}$.  For each edge $f_{\tilde c}=\mathcal L_x\cap R_{\tilde c}$ of $\mathcal L_x$, with vertices at $v_{\tilde c}$ and $v_{\tilde d}$, let $U(f_{\tilde c})$ be the open trapezoid in $R_{\tilde c}$ joining $U_{\tilde c}$ to $U_{\tilde d}$.  The desired open neighborhood of $\mathcal L_x$ is $\bigcup_{f_{\tilde c}\in\text{Edges}(\mathcal L_x)}U(f_{\tilde c})$, as shown in Figure~\ref{fig:product_neighborhood}.

\begin{figure}[ht]
\includegraphics[width=0.25\textwidth]{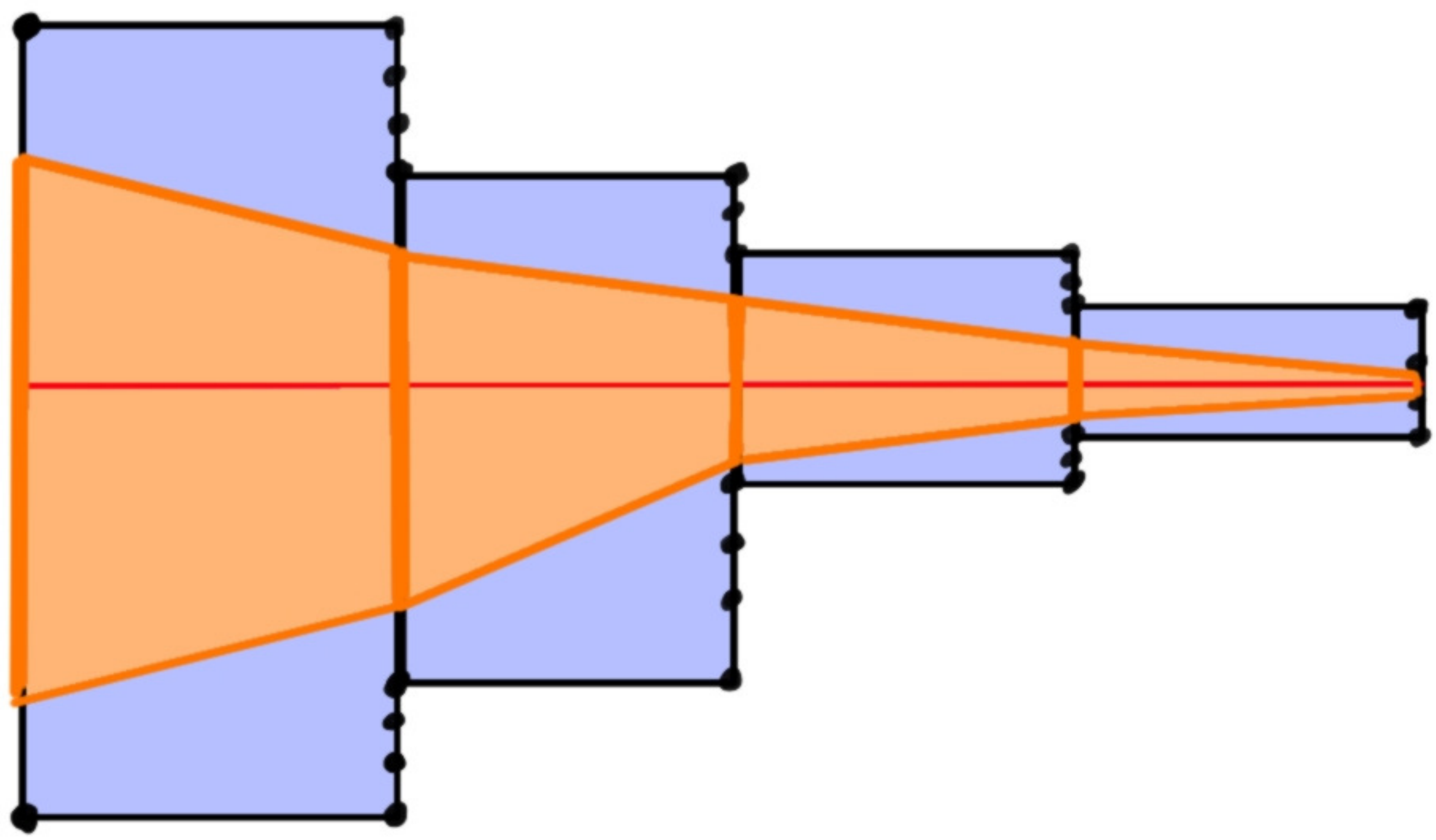}
\caption{A product neighborhood of a regular leaf.}\label{fig:product_neighborhood}
\end{figure}

\textbf{Proof of~$\eqref{leaf_level}$:}  This follows immediately from the definitions of levels and leaves.
\end{proof}

We denote by $\mathcal Y_0$ the set of leaves of $\widetilde X$ and define a surjection $\rho_0:\widetilde X\rightarrow\mathcal Y_0$ by $\rho_0(x)=\mathcal L_x$.

\begin{lem}\label{lem:finite_intersection_leaf_edge}
Let $\tilde e$ be a vertical edge of $\widetilde X$ and let $\mathcal L_x$ be a leaf.  If $\phi$ is a train track map, then $|\tilde e\cap\mathcal L_x|\leq 1$.
\end{lem}

\begin{proof}
When $\phi$ is a train track map, distinct points in each vertical edge $e$ lie on distinct leaves, i.e. the map $\rho_0:e\rightarrow\mathcal Y_0$ is injective.  Note that in this case, each 2-cell of $\widetilde X$ is foliated by a family of distinct fibers of $\rho$, each of which is a midsegment.
\end{proof}

\subsection{Forward space in the train track case}\label{sec:forward_space_in_train_track_case}
Suppose that $\phi$ is a train track map.  We now describe an $\reals$-tree $\mathcal Y$ whose points are equivalence classes of leaves, and a $G$-action on $\mathcal Y$, and use this to establish that $\widetilde X$ is level-separated.  This construction mimics the stable tree discussed in~\cite{BestvinaFeighnHandel_lamination}, although the underlying set is defined differently.  The referee explains that it is a special case of a construction in~\cite{GJLL}.  Let $\mathcal E=\reals[\text{Edges}(V)]$ and denote by $\vec e_i$ the basis element of $\mathcal E$ corresponding to $e_i$.  Let $\mathfrak M:\mathcal E\rightarrow\mathcal E$ be the linear map whose matrix with respect to the basis $\{\vec e_i\}$ has $ij$-entry the number of times the path $\phi(e_i)$ traverses $e_j$, ignoring orientation.  Note that this \emph{transition matrix}, which we also denote by $\mathfrak M$, is a nonnegative matrix.  We further assume that $\mathfrak M$ is irreducible.

Let $\varpi$ be the Perron-Frobenius eigenvalue of $\mathfrak M$.  As shown in~\cite{BestvinaHandel}, $\varpi>1$ since $\phi$ is irreducible and has infinite order.  Let $\mathbf v$ be a $\varpi$-eigenvector, all of whose entries are positive.  For each $i$, let $c_i$ be the magnitude of the Perron projection of $\vec e_i$ onto $\reals[\mathbf v]$.  As is made precise in Definition~\ref{defn:exponentially_expanding}, the map $\phi$ expands edges of $V$ by a factor of $\varpi$.

We now choose an equivariant weighting of vertical edges in $\widetilde X$ by letting $|e_i|=c_i$ for each edge $e_i$ of $V$, letting each horizontal edge of $X$ have unit weight, and pulling back these weights to $\widetilde X$.  This determines the metric $\dist$ on $\widetilde X^1$.  For each $e_i$ and each $n\in\integers$, we define the \emph{scaled length} of a lift $\tilde e_i$ of $e_i$ to $\widetilde V_n$ to be $\varpi^{-n}|\tilde e_i|=\varpi^{-n}c_i$.  Let $\dist_n:\widetilde V_n\times\widetilde V_n\rightarrow[0,\infty]$ be the resulting path-metric.

Given leaves $\mathcal L_x,\mathcal L_y$, with $x,y\in\widetilde V_k$ for some $k$, let $$\dist_{\infty}(\mathcal L_x,\mathcal L_y)=\lim_{n\rightarrow\infty}\dist_{n}(\tilde\phi^n(x),\tilde\phi^n(y)).$$

This limit exists and is finite because $\dist_n(\tilde\phi^n(x),\tilde\phi^n(y))$ is non-increasing and bounded.  Moreover, $\dist_{\infty}(\mathcal L_x,\mathcal L_y)$ is well-defined since for other choices $x'\in\mathcal L_x\cap\widetilde V_{k'}$ and $y'\in\mathcal L_y\cap\widetilde V_{k'}$, for all but finitely many $n$, we have $\phi^{n'}(x')=\phi^n(x)$ and $\phi^{n'}(y')=\phi^n(y)$ for some $n'$.

\begin{lem}\label{lem:metrizing}
Let $\mathcal Y$ be the quotient of $\mathcal Y_0$ obtained by identifying points $\rho_0(x),\rho_0(y)$ for which $\dist_{\infty}(\rho_0(x),\rho_0(y))=0$.  Then the induced pseudometric $\dist_{\infty}:\mathcal Y\rightarrow [0,\infty)$ is a metric.  Let $\rho:\widetilde X\rightarrow\mathcal Y$ be the composition $\widetilde X\stackrel{\rho_0}{\longrightarrow}\mathcal Y_0\rightarrow\mathcal Y$.  Then the restriction of $\rho$ to each vertical edge is an isometric embedding.
\end{lem}

\begin{proof}
$\dist_{\infty}$ is symmetric and satisfies the triangle inequality.  Hence $\dist_{\infty}:\mathcal Y\rightarrow[0,\infty)$ is a metric.  
Let $e_i$ be a vertical edge with endpoints $x,y$.  Then the distance in $\widetilde V$ between the endpoints of $\tilde\phi^n(e)$ is $\varpi^n|e_i|$, whence $\dist_{\infty}(\rho(x),\rho(y))=c_i$.  Our assumption that $\tilde\phi$ has a constant-speed parametrization on each edge implies that the same equality holds for any subinterval of $e_i$.
\end{proof}

\begin{prop}\label{prop:R_tree_base case}
Suppose that every edge of $V$ is expanding with respect to $\phi$.  Then:
\begin{enumerate}
 \item The map $\rho:\widetilde X\rightarrow\mathcal Y$ is continuous.
 \item $(\mathcal Y,\dist_{\infty})$ is a 0-hyperbolic geodesic metric space, i.e. $\mathcal Y$ is an $\reals$-tree.
 \item $\mathcal Y$ admits a $G$-action by homeomorphisms with respect to which $\rho$ is $G$-equivariant.
 \item The restriction of the $G$-action on $\mathcal Y$ to $F$ is an action by isometries. 
 \item The stabilizer in $F$ of $\rho(\tilde x)$ is trivial whenever $\tilde x$ is a lift to $\widetilde X$ of a periodic point in $V$.
\end{enumerate}
\end{prop}

\begin{proof}
\textbf{Continuity of $\rho$:}  The restriction of $\rho$ to each vertical edge $e$ is continuous since it is an isometric embedding, and $\rho$ is continuous on each closed 2-cell since $\rho$ is constant on each midsegment and each 2-cell is therefore foliated by fibers of $\rho$ since $\phi$ is a train track map.  Since $\widetilde X$ is locally finite, the pasting lemma implies that $\rho$ is continuous on $\widetilde X$.

\textbf{$\reals$-tree:}  Let $x,y\in\widetilde V_n$ and let $P\rightarrow\widetilde V_0$ be a path joining $x$ to $y$.  Then since $\rho$ is continuous, $\rho(P)$ is a path joining $\rho(x)$ to $\rho(y)$, whence $\mathcal Y$ is path-connected.  Since $\mathcal Y$ is a path-connected subspace of an asymptotic cone of the simplicial tree $(\widetilde V,\dist_0)$, the space $\mathcal Y$ is an $\reals$-tree~\cite[Prop.~3.6]{KapovichLeeb_asymptotic_cone}.  (The asymptotic cone in question is built using any non-principal ultrafilter on $\naturals$, the observation point $(\tilde v,\tilde\phi(\tilde v),\ldots)$, and the scaled metrics $\dist_n$ on $\widetilde V_0$.)

\textbf{The $G$-action:}  For $g\in G$ and $x\in\widetilde X$, let $g\rho(x)=\rho(gx)$.  This defines an action since $G$ takes leaves in $\widetilde X$ to leaves.  The action is by homeomorphisms since $\rho$ is continuous and $G$ acts by homeomorphisms on $\widetilde X$.

\textbf{The $F$-action is isometric:}  Let $x,y\in\widetilde X$.  Since $F$ acts by isometries on each $\widetilde V_n$, for each $f\in F$, we have
\begin{eqnarray*}
\dist_{\infty}(f\rho(x),f\rho(y))&=&\lim_n\dist_n(\phi^n(fx),\phi^n(fy))\\
&=&\lim_n\dist_{n}(\Phi^n(f)\phi^n(x),\Phi^n(f)\phi^n(y))=\lim_n\dist_n(\phi^n(x),\phi^n(y))=\dist_{\infty}(\rho(x),\rho(y)).
\end{eqnarray*}

\textbf{The $F$-action is free on periodic points:}  Let $x\in V$ be a periodic point and let $\tilde x$ be a lift of $x$ to $\widetilde X$.  By Corollary~\ref{cor:bh_separation}, either $\rho(\tilde x)\neq f\rho(\tilde x)$, and we are done, or the forward rays $\sigma_{\tilde x}$ and $f\sigma_{\tilde x}$ emanating from $\tilde x$ and $f\tilde x$ respectively lie at finite Hausdorff distance.  It follows that the immersed vertical path $\widetilde P$ joining $\tilde x$ to $f\tilde x$ projects to an essential closed path $P\rightarrow V$, based at $x$, such that $\phi^k(P)$ is a periodic Nielsen path for some $k\geq 0$.  This contradicts hyperbolicity of $G$.
\end{proof}


\begin{rem}
When $\phi$ is a $\pi$-isomorphism, and $G$ is hyperbolic, the action of $F$ on $\mathcal Y$ can be shown to be free using Lemma~\ref{lem:splitting_lemma_tt} and the fact that there are no nontrivial periodic Nielsen paths.  We expect that this is true for a general hyperbolic monomorphism, but a free action on the set of periodic points suffices for our purposes.
\end{rem}

\subsection{Level-separation in the train track case}\label{subsec:tt_cube}
The purpose of this subsection is to prove Lemma~\ref{lem:tt_level_separated}.

\begin{defn}[Transverse]\label{defn:transverse}
Let $\mathcal T$ be an $\reals$-tree.  The map $\theta:\mathbf R\rightarrow\mathcal T$ is \emph{transverse} to $y\in\mathcal T$ if for each $p\in\theta^{-1}(y)$, there exists $\epsilon>0$ such that $\theta((p-\epsilon,p))$ and $\theta((p,p+\epsilon))$ lie in distinct components of $\mathcal T-\{y\}$.  Note that if $\theta$ is transverse to $y$, then $\theta^{-1}(y)$ is a discrete set.
\end{defn}

We denote by $\mathbf R^+$ a combinatorial sub-ray of the combinatorial line $\mathbf R$.

\begin{lem}\label{lem:trichotomy_tt}
Let $\mathcal T$ be an $\reals$-tree. Let $\mathcal T_0\subseteq\mathcal T$ have the property that $\mathcal T-\{y\}$ has two components for each $y\in\mathcal T_0$ and each open arc of $\mathcal T$ contains a point of $\mathcal T_0$.  Let $\theta:\mathbf R\rightarrow\mathcal T$ or $\theta:\mathbf R^+\rightarrow\mathcal T$ be a continuous map.  Suppose $\theta$ is transverse to every point in $\mathcal T_0$.  Moreover, suppose that each edge $e$ of the domain of $\theta$ has connected intersection with the preimage of each point in $\mathcal T$.  Then one of the following holds:
\begin{enumerate}
 \item \label{item:odd}There exists a nontrivial arc $\alpha\subset\mathcal T$ such that $|\theta^{-1}(y)|$ is odd for all $y\in\alpha\cap\mathcal T_0$.
 \item \label{item:infinite_preimage}There exists a point $y\in\mathcal T$ with $|\theta^{-1}(y)|$ infinite.
 \item \label{item:bigpreimage_tt}For each $r\geq0$, there exists $y_r\in\mathcal T$ such that $\theta^{-1}(y_r)$ has diameter at least $r$.
\end{enumerate}
\end{lem}
\begin{proof}
For each $p\in\mathbf R$, we denote by $\bar p$ its image in $\mathcal T$ and by $|\theta^{-1}(x)|$ the number of components of the preimage of $x\in\mathcal T$ in $\mathbf R$.

We now show that either $(3)$ holds or $\image(\theta)$ is locally compact.  We first claim that either $(3)$ holds, or for each edge $e$ of $\mathbf R$, there are (uniformly) finitely many edges $f$ such that $\theta(f)\cap\theta(e)\neq\emptyset$.  Indeed, if there are arbitrarily many such $f$, then for each $r\geq 0$, we can choose $f$ such that $\dist_{\mathbf R}(e,f)>r$ but $\theta(e)\cap\theta(f)\neq\emptyset$, yielding~(3).  Second, choose a point $p\in\mathcal T$.  Our first claim shows that either~(3) holds or the set $\{e_i\}$ of edges with $p\in\theta(e_i)$ is finite.  Assume the latter. Then for each $i$ we can choose $\epsilon_i>0$ such that the $\epsilon_i$-neighborhood of $p$ in $\theta(e_i)$ is disjoint from the image of each edge not in $\{e_i\}$.  Let $\epsilon=\min_i\epsilon_i$.  Then the $\epsilon$-neighborhood of $p$ in $\image(\theta)$ lies in $\cup_i\theta(e_i)$ and thus has compact closure.

There exist sequences $\{a_i\}$ and $\{b_i\}$ in $\mathbf R=(-\infty,\infty)$ converging to $\infty$ and $-\infty$ respectively, whose images are sequences $\{\bar a_i\}$ and $\{\bar b_i\}$ that converge to points $\bar a_\infty$ and $\bar b_\infty$ in $\image(\theta)\cup\partial\image(\theta)$.  Indeed, since $\image(\theta)$ is a locally compact $\reals$-tree, $\image(\theta)\cup\partial\image(\theta)$ is compact by~\cite[Exmp.~II.8.11.(5)]{BridsonHaefliger}.

Suppose $\bar a_\infty \neq \bar b_\infty$.  Let $\alpha$ be a nontrivial arc in the geodesic joining $\bar a_{\infty}$ and $\bar b_{\infty}$.  Note that $\theta^{-1}(\bar c)$ has either odd or infinite cardinality for each $\bar c\in\alpha\cap\mathcal T_0$, since it must separate $a_i$ from $b_i$ for all but finitely many $i$.  Hence either conclusion~$(1)$ or $(2)$ holds.

Suppose $\bar a_\infty$ and $\bar b_\infty$ are equal to the same point $\bar p_\infty\not\in\image\theta$.  Let $\bar o$ denote the image of the basepoint $o$ of $\mathbf R$.  The intersections $\bar o\bar a_i \cap \bar o \bar p_\infty$ converge
to the segment $\bar o \bar p_\infty$.  The same holds for $\bar o \bar b_i$.  We use this to choose a new pair of sequences $\{a_i'\}$ and $\{b_i'\}$
that still converge to $\pm\infty$, and with the additional property that
$\bar a_i'=\bar b_i'$. We do this by choosing the image points far out in
$\bar o \bar p_\infty$.  We have thus found  arbitrarily distant points in $\mathbf R$ with the same images, verifying conclusion~(3).

In the remaining case, there exists $p\in\mathbf R$ such that $\bar p=\bar p_\infty$.  Consider the restriction of $\theta$ to a ray $\mathbf R^+$ so that the initial vertex of $\mathbf R^+$ is $p$.  Let $o$ be the vertex adjacent to $p$.  Repeating the previous argument with $a_i=p$ and $b_i$ converging to $p_\infty$ verifies conclusion~(3).  

The case of the ray $\mathbf R^+$ is similar.
\end{proof}

By Lemma~\ref{lem:finite_intersection_leaf_edge} and Proposition~\ref{prop:leaf_properties}, for each regular leaf there is a pair $(\leftl,\rightl)$ of closed halfspaces in $\widetilde X$ such that $\leftl\cup\rightl=\widetilde X$ and $\leftl\cap\rightl=\mathcal L$.  Points of $\rho(\widetilde X^0)$ are \emph{singular} points of $\mathcal Y$, and the other points are \emph{regular}.  If $\mathcal L$ is a regular leaf, then $\mathcal Y-\rho(\mathcal L)$ has two components, namely the interiors of the images of $\leftl$ and $\rightl$.  Since there are countably many singular points in $\mathcal Y$, each open arc in $\mathcal Y$ contains a regular point.

\begin{lem}\label{lem:rho_transverse}
For any geodesic $\gamma:\mathbf R\rightarrow\widetilde X^1$, the map $\theta=\rho\circ\gamma:\mathbf R\rightarrow\mathcal Y$ is transverse to regular points.
\end{lem}

\begin{proof}
Let $y\in\mathcal Y$ be a regular point, so that each $x\in\rho^{-1}(y)$ lies in the interior of a vertical 1-cell, which in turn embeds in $\mathcal Y$ by Proposition~\ref{prop:leaf_properties}.  The image of the vertical 1-cell is separated by $\rho(x)=y$.
\end{proof}

The goal of the rest of this subsection is to prove Corollary~\ref{cor:odd_intersection}, which depends on Corollary~\ref{cor:bh_separation}.  We first give a proof of the latter in the case where $\phi$ is $\pi_1$-surjective incorporating the technology of~\cite{BestvinaFeighnHandelTits}, followed by a self-contained proof in the general case.

\begin{cor}\label{cor:odd_intersection}
Let $\gamma:\mathbf R\rightarrow\widetilde X^1$ be an $M$-deviating geodesic for some $M\geq 0$.  Then there exists a regular leaf $\mathcal L$ such that $|\gamma\cap\mathcal L|$ is finite and odd.
\end{cor}

\begin{proof}
Consider the restriction of $\rho$ to $\gamma$.  By Lemma~\ref{lem:rho_transverse}, $\rho|_{\gamma}$ is transverse to regular points.  By Lemma~\ref{lem:trichotomy_tt}, one of the following holds:
\begin{itemize}
 \item There exists a regular point $y\in\mathcal Y$ such that $\rho^{-1}(y)\cap\gamma$ has finite, odd cardinality.
 \item For all $r\geq 0$, there exists $y_r\in\mathcal Y$ such that $\diam(\rho^{-1}(y_r)\cap\gamma)>r$.  (This includes the case in which some point in $\mathcal Y$ has infinite preimage.)
\end{itemize}
In the first case, note that $\rho^{-1}(y)$ is the union of regular leaves, one of which must therefore have odd intersection with $\gamma$.  We will now show that the second case leads to a contradiction.

In the second case, for each $r\geq 0$, there exists $m\in\integers$ and forward rays $\sigma_1,\sigma_2$, originating at points of $\gamma$ and traveling through $\widetilde V_m$, such that $\rho(\sigma_1)=\rho(\sigma_2)$ and $\dist_{\widetilde V_m}(\sigma_1\cap\widetilde V_m,\sigma_2\cap\widetilde V_m)>r$.  Indeed, let $x_1,x_2\in\gamma$ be chosen so that $\rho(x_1)=\rho(x_2)=y_r$ and $q(x_1)\leq q(x_2)=m$ and $\dist_{\widetilde X}(x_1,x_2)>r+M$.  For some $k\geq 0$, we have $\tilde\phi^k(x_1)=x'_1\in\widetilde V_m$.  We also have $\rho(x'_1)=y_r$.  Since $\gamma$ is $M$-deviating, considering the $\delta$-thin triangle $x_1x_2x'_1$ shows that $\dist_{\widetilde X}(x'_1,x_2)>r$.  Hence $\dist_{\widetilde V_m}(x'_1,x_2)>r$. We now apply Corollary~\ref{cor:bh_separation}.

The rays $\sigma_1,\sigma_2$ cannot fellowtravel when $r$ is sufficiently large, since the conclusion of a thin quadrilateral argument would then contradict the hypothesis that $\gamma$ is $M$-deviating.  Hence, by Corollary~\ref{cor:bh_separation}, we see that $\rho(\sigma_1)\neq\rho(\sigma_2)$, a contradiction.
\end{proof}

  The \emph{tightening} of a path $P$ in a graph is the immersed path that is path-homotopic to $P$.  A \emph{periodic Nielsen path} in $V$ is an essential path $P$ such that the tightening of $\phi^k(P)$ is path-homotopic to $P$ for some $k>0$.  The following is a rephrasing of a special case of~\cite[Lem.~6.5]{Levitt:split}, which splits into ~\cite[Lem. 4.1.4, Lem. 4.2.6, Lem. 5.5.1]{BestvinaFeighnHandelTits}.

\begin{lem}[Splitting lemma]\label{lem:splitting_lemma_tt}
Let $\phi:V\rightarrow V$ be a $\pi_1$-surjective train track map.  Let $P\rightarrow V$ be a path.   Then there exists $n_0$ such that the tightening of $\phi^{n_0}(P)$ is a concatenation $Q_1\cdots Q_k$, where each $Q_s$ is of one of the following types:
\begin{enumerate}
\item a periodic Nielsen path;
\item an edge of $V$;
\item a subinterval of an edge of $V$, if $s\in\{1,k\}$;
\end{enumerate}
Moreover, for all $n\geq n_0$, the tightening of $\phi^n(P)$ is equal to a concatenation of the tightenings of the paths $\phi^{n-n_0}(Q_s)$.
\end{lem}

\begin{cor}\label{cor:bh_separation}
Let $\sigma_1,\sigma_2$ be forward rays beginning on $\widetilde V_m$.  Then either $N(\sigma_1),N(\sigma_2)$ lie at finite Hausdorff distance or $\rho(\sigma_1)\neq\rho(\sigma_2)$.
\end{cor}

Corollary~\ref{cor:bh_separation} means that for each $y\in\mathcal Y$, any two forward rays in $\rho^{-1}(y)$ fellowtravel, in the sense that they lie at finite Hausdorff distance.

\begin{proof}[Proof of Corollary~\ref{cor:bh_separation} when $\phi$ is $\pi_1$-surjective]
Let $P\rightarrow\widetilde V_m$ be a path from $\sigma_1$ to $\sigma_2$.  Lemma~\ref{lem:splitting_lemma_tt} implies that for some $n\geq0$, the tightening of $\tilde\phi^n(P)$ splits as the concatenation of Nielsen paths and edges.  If $\tilde\phi^n(P)$ is the concatenation of Nielsen paths, then $\sigma_1,\sigma_2$ fellowtravel.  Otherwise the splitting contains an edge $e$ and for all $n'\geq n$, we have $\dist_{n'}(\sigma_1\cap\widetilde V_{n'},\sigma_2\cap\widetilde V_{n'})\geq |e|$, whence $\rho(\sigma_1)\neq\rho(\sigma_2)$.
\end{proof}

\begin{proof}[Proof~of~Corollary~\ref{cor:bh_separation} in the general case]
If $\sigma_1,\sigma_2$ do not fellowtravel, then by Lemma~\ref{lem:separated_or_fellow_travel} and Lemma~\ref{lem:separated_implies_band}, the geodesic of $\widetilde V_m$ joining the initial points of $\sigma_1,\sigma_2$ contains an open arc $\alpha\subset e$, for some edge $e$, such that each regular leaf intersecting $\alpha$ separates $\sigma_1,\sigma_2$.  For each $n\geq m$, let $a_n=\sigma_1\cap\widetilde V_n$ and $b_n=\sigma_2\cap\widetilde V_n$.  Then for each $n$, the geodesic of $\widetilde V_n$ joining $a_n,b_n$ contains $\phi^n(\alpha)$.  Regarding $e$ as a copy of $[0,1]$ with weight $|e|$, and $\alpha=(t_1,t_2)\subset[0,1]$, we see that $\dist_n(a_n,b_n)\geq|e|(t_2-t_1).$  Hence $\dist_{\infty}(\rho(\sigma_1),\rho(\sigma_2))>0$.
\end{proof}

\begin{lem}\label{lem:separated_or_fellow_travel}
Let $\sigma_1,\sigma_2$ be forward rays beginning on $\widetilde V_m$.  Then either $N(\sigma_1),N(\sigma_2)$ are at finite Hausdorff distance or there exists a regular leaf separating $\sigma_1$ from $\sigma_2$.
\end{lem}

\begin{proof}
We claim that if $\sigma_1,\sigma_2$ are not separated by a regular leaf, then $\rho(\sigma_1)=\rho(\sigma_2)$.  If $\rho(\sigma_1)\neq\rho(\sigma_2)$, then these points are separated by a point $y\in\mathcal Y$ whose preimage is the union of regular leaves.  Any path joining $\sigma_1,\sigma_2$ must intersect the union of these leaves in an odd-cardinality set, so one of them separates $\sigma_1,\sigma_2$.  Hence suppose $\rho(\sigma_1)=\rho(\sigma_2)$ and let $z=\rho(\sigma_i)\in \mathcal Y$.

Let $p\geq m$ be such that there is a vertical geodesic $I_{p}\rightarrow\widetilde V_{p}$ joining $\sigma_1\cap\widetilde V_p,\sigma_2\cap\widetilde V_p$ and with the property that $\rho^{-1}(z)\cap I_p$ has minimal cardinality.  For simplicity, having chosen $p$, we will translate so that $p=0$.

Having chosen $I_0$, we now inductively define paths $I_n\rightarrow\widetilde V_n$ joining $\sigma_1$ to $\sigma_2$ as follows.  For $n\geq 0$, express $I_n=e_1e_2\cdots e_k$ as a concatenation of partial edges: $e_1,e_k$ are closed subintervals of edges and the other $e_i$ are entire edges.  Let $I_{n+1}\rightarrow\widetilde V_{n+1}$ be the path $\tilde\phi(e_1)\cdots\tilde\phi(e_k)$.  Let $\bar I_n$ be the image of $I_n$ in $\widetilde X$ and note that $\bar I_n$ is a finite subtree of $\widetilde V_n$.  Observe that $T=\rho(\bar I_0)\subset\mathcal Y$ is a finite tree, since it is the union of finitely many closed embedded arcs.  Let $\rho_n:I_n\rightarrow\mathcal Y$ be the composition $I_n\rightarrow\widetilde X\stackrel{\rho}{\rightarrow}\mathcal Y$.  Since each $\bar I_n\rightarrow\bar I_{n+1}$ is surjective, $\rho(\bar I_n)=T$ for all $n\geq 0$.  The maps $e_i\rightarrow\tilde\phi(e_i)$ induce a map $I_n\rightarrow I_{n+1}$ so that the following diagram commutes.

\begin{center}
$
\begin{diagram}
\node{I_n}\arrow{e}{}\arrow{s}{}\node{I_{n+1}}\arrow{s}{}\arrow{se,l}{\rho_{n+1}}\\
\node{\bar I_n}\arrow{e,t}{\tilde\phi}\node{\bar I_{n+1}}\arrow{e,t}{\rho}\node{T\subset\mathcal Y}
\end{diagram}
$
\end{center}

Since $\rho(\sigma_1)=\rho(\sigma_2)$, each $\rho_n:I_n\rightarrow T$ is a closed path in $T$ beginning and ending at $z\in T$.  If $d_{\widetilde V_n}(\sigma_1\cap\bar I_n,\sigma_2\cap\bar I_n)$ is uniformly bounded as $n\rightarrow\infty$, then $\sigma_1,\sigma_2$ lie at uniformly bounded vertical distance, and so $N(\sigma_1)$ and $N(\sigma_2)$ lie at finite Hausdorff distance.

Since $I_n$ is vertical, $\rho_n^{-1}(z)$ is finite, and $I_n=Q_1Q_2\cdots Q_r$, where the interiors of the $Q_i$ are the components of $I_n-\rho_n^{-1}(z)$.  Let $\bar Q_i$ denote the image of $Q_i$ in $\widetilde V_n$.  Note that $r$ is independent of $n$; indeed, this is ensured by the minimality achieved through our choice of $p$.  It follows that no regular leaf intersects $\bar Q_i$ and $\bar Q_j$ for $i\neq j$, for otherwise we could apply $\tilde\phi$ finitely many times and reduce $r$.

Let $a_i$ and $b_i$ be the endpoints of $Q_i$, and let $\bar a_i,\bar b_i$ be their images in $\bar I_n$.  We will show that there exists $M$, independent of $n$, such that $d_{\widetilde V_n}(\bar a_i,\bar b_i)\leq M$. We conclude that $d_{\widetilde V_n}(N(\sigma_1),N(\sigma_2))\leq rM$ for all sufficiently large $n$.

To verify the existence of $M$, we shall show that there exists a leaf $\mathcal L_i$ that intersects the initial and terminal (possibly partial) edges of $\bar Q_i$, intersecting these edges in points $c_i,d_i$ respectively.  This leaf $\mathcal L_i$ must intersect $\bar I_0$ in points $\hat c_i,\hat d_i$ with $\tilde\phi^n(\hat c_i)=c_i$ and $\tilde\phi^n(\hat d_i)=d_i$.  Hence there are forward paths $\hat c_ic_i$ and $\hat d_id_i$ of $N(\mathcal L_i)$ whose intersections with $\widetilde X^1$ are $\lambda$-quasigeodesics lying in forward rays of $N(\mathcal L_i)$.  The quasigeodesic quadrilateral $\hat c_ic_id_i\hat d_i$ shows that $\hat c_ic_i$ and $\hat d_id_i$ fellowtravel at distance $M'=M'(\delta,\lambda,|I_0|)$, and hence $d_{\widetilde V_n}(\bar a_i,\bar b_i)\leq M$, where $M=M'+2$.

It remains to find the leaf $\mathcal L_i$.  Note that if $\bar a_i,\bar b_i$ lie on a common leaf, we are done.  We can assume that no regular leaf separates $\bar a_i$ from $\bar b_i$.  Indeed, any such leaf could not separate $\sigma_1,\sigma_2$ by our assumption, and thus must end on $\bar Q_j$ for some $i\neq j$, which was ruled out above.   Hence each leaf emanating from the image of the initial edge of $\bar Q_i$ intersects the images of an even number of edges of $\bar Q_i$.  Let $\mathcal L$ be such a regular leaf, and let $b'_i\in\mathcal L\cap\bar Q_i$ be a point outside of the image of the initial partial edge of $\bar Q_i$.  We claim that by choosing $\mathcal L$ to intersect the $\bar Q_i$ at a point $a'_i$ sufficiently close to $\bar a_i$, we can ensure that $b'_i$ lies in the image of the terminal partial edge of $\bar Q_i$.  

Indeed, choose a sequence $\{a'_{ik}\}_k$ of regular points in the initial edge of $\overline Q_i$, with $a'_{ik}\rightarrow\bar a_i$.  For each $k$. let $\mathcal L_{ik}$ be the regular leaf containing $a'_{ik}$.  If $\mathcal L_{ik}$ intersects the terminal edge of $\overline Q_i$, we  are done, so we let $b'_{ik}$ be a point of $\mathcal L_{ik}\cap\overline Q_{ik}$ that lies in a non-terminal, non-initial edge.  By possibly passing to a subsequence, compactness allows us to assume that $\{b'_{ik}\}$ converges to some $b_i'\in\overline Q_i$ different from $\bar b_i$.  Since $\rho$ is continuous and $\rho(b'_{ik})=\rho(a'_{ik})\rightarrow z$, we have $\rho(b_i')=z$, contradicting the fact that the interior of $\overline Q_i$ contains no point in $\rho^{-1}(z)$.
\end{proof}

\begin{lem}\label{lem:separated_implies_band}
Let $\sigma_1,\sigma_2$ be forward rays beginning on $\widetilde V_m$ that do not fellowtravel.  Suppose there exists a regular leaf $\mathcal L$ separating $\sigma_1,\sigma_2$.  Then the geodesic of $\widetilde V_m$ joining the initial points of $\sigma_1,\sigma_2$ contains an open arc $\alpha\subset e$, for some edge $e$, such that each regular leaf intersecting $\alpha$ separates $\sigma_1,\sigma_2$.

Hence for all $n\geq m$, the geodesic of $\widetilde V_n$ joining $\sigma_1\cap\widetilde V_n$ and $\sigma_2\cap\widetilde V_n$ contains $\tilde\phi^n(\alpha)$.
\end{lem}

\begin{proof}
Let $P\rightarrow\widetilde V_m$ be a vertical geodesic joining $\sigma_1,\sigma_2$.  For any $n\geq m$, given a path $U\rightarrow\widetilde V_n$ joining $\sigma_1,\sigma_2$, a \emph{syllable} of $U$ is a maximal subpath $Q$ that is \emph{legal} in the sense of~\cite{BestvinaHandel}; since $\phi$ is a train track map, this means that $\tilde\phi^k(Q)$ is embedded for all $k\geq 0$.  Consider the vertical geodesic $T_k$ joining the endpoints of $\tilde\phi^k(P)$ for $k\geq 0$.  Each $T_k$ can be expressed as a concatenation of syllables, since $\phi$ is a train track map, and this decomposition is unique.  Choose $k\geq 0$ such that the number of syllables in the decomposition of $T_k$ is equal to the number of syllables in $T_{k'}$ for all $k'\geq k$.  Let $T_k=Q_1\cdots Q_n$ be a decomposition into syllables.  Observe that nonconsecutive syllables of $T_k$ cannot intersect a common leaf, for otherwise applying some iterate of $\tilde\phi$ would result in a path with fewer syllables.

Since $\mathcal L$ intersects each syllable in at most one point, the minimality of $T_k$ guarantees that $|\mathcal L\cap T_k|<3$ and hence, since this cardinality is odd, $|\mathcal L\cap T_k|=1$.  Hence there exists a unique $Q_i$ such that $\mathcal L\cap T_k$ is contained in $\interior{Q_i}$.

For each $p\in\naturals$, let $B_i^{\pm}(\frac{1}{p})$ be the two half-open $\frac{1}{p}$-neighborhoods in $Q_i$ bounded at $\mathcal L\cap Q_i$.  If the lemma does not hold, then for each $p$ there exists a regular leaf $\mathcal L^{\pm}(\frac{1}{p})$ intersecting $B_i^{\pm}(\frac{1}{p})$ but failing to separate $\sigma_1$ and $\sigma_2$.  Each $\mathcal L^{\pm}(\frac{1}{p})$ has even intersection with $T_k$ and thus also intersects $Q_{i\pm 1}$ in a single point.  The sequence $\{\mathcal L^{\pm}(\frac{1}{p})\cap Q_{i\pm1}\}_p$ has a subsequence that converges to a point $z_{\pm}\in Q_{i\pm1}$ such that $\rho(z_{\pm})=\rho(\mathcal L)$.  Observe that no regular leaf separates $z_{\pm}$ from $\mathcal L\cap Q_i$, since such a separating regular leaf would have to intersect some $\mathcal L^{\pm}(\frac{1}{p})$, which is impossible since leaves are disjoint.  Hence, by Lemma~\ref{lem:separated_or_fellow_travel}, the forward rays $\sigma_{z\pm}$ and $\sigma$, respectively emanating from $z_{\pm}$ and $\mathcal L\cap Q_i$, must fellowtravel.

\begin{figure}[ht]
\includegraphics[width=0.5\textwidth]{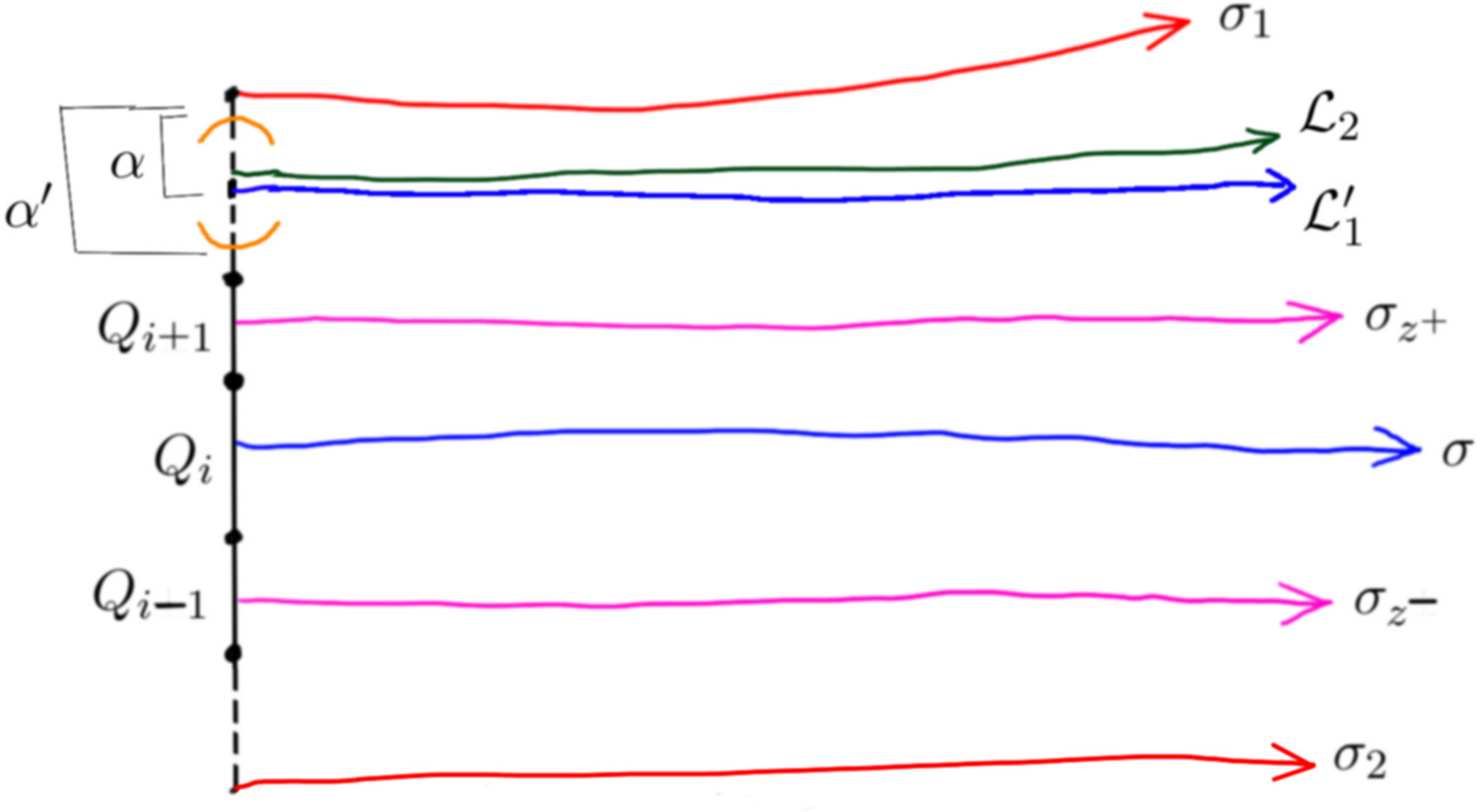}
\caption{The forward rays and leaves in the proof of Lemma~\ref{lem:separated_implies_band}}\label{fig:sigma_ft}.
\end{figure}

If $\sigma$ fellowtravels with $\sigma_1$ [resp. $\sigma_2$] and $\sigma_{z^{\pm}}$ fellowtravels with $\sigma_2$ [resp. $\sigma_1]$, then since $\sigma,\sigma_{z^{\pm}}$ fellowtravel, we conclude that $\sigma_1,\sigma_2$ fellowtravel, contradicting our hypotheses.  See Figure~\ref{fig:sigma_ft}.  If $\sigma,\sigma_1$ (for example) do not fellowtravel, then $\sigma_1,\sigma_{z^+}$ (say) also do not fellowtravel.  Lemma~\ref{lem:separated_or_fellow_travel} implies that a regular leaf $\mathcal L_1$ separates $\sigma_{z^+},\sigma_1$.  The part of $T_k$ subtended by $\sigma_{z^+},\sigma_1$ has strictly fewer syllables than $T_k$, so by induction, there is an open interval $\alpha'\subset T_k$ with the following properties:
\begin{enumerate}
 \item $\alpha'$ is contained in the interior of some edge.
 \item $\alpha'$ intersects a regular leaf $\mathcal L'_1$ that separates $\sigma_{z^+}$ and $\sigma_1$.
 \item $\alpha'$ lies on the part of $T_k$ between $\sigma_{z^+}$ and $\sigma_1$.
 \item All regular leaves intersecting $\alpha'$ separate $\sigma_{z^+}$ from $\sigma_1$.
\end{enumerate}

Let $\mathcal L_2$ be a regular leaf intersecting $\alpha'$ between $\mathcal L_1'$ and $\sigma_1$.  Then $\mathcal L_2$ separates $\sigma_1$ from $\sigma_{z^+}$ by the induction hypothesis, and therefore $\mathcal L_2$ separates $\sigma_1$ from $\sigma_2$.  The subinterval of $\alpha'$ above $\mathcal L'_1$ (and so containing all such $\mathcal L_2$) is the desired interval $\alpha$.  See Figure~\ref{fig:sigma_ft}.

In the base case, $T_k$ has a single syllable, and any open subinterval of an edge suffices.

Finally, let $n\geq m$ and let $P_n\rightarrow\widetilde V_n$ be the geodesic joining $\sigma_1,\sigma_2$.  For any $x\in\tilde\phi^n(\alpha)$, and any $\epsilon>0$, there exists a regular point $y\in\tilde\phi^n(\alpha)$ at distance less than $\epsilon$ from $x$, since edges are expanding.  The regular leaf $\mathcal L_y$ separates $\sigma_1,\sigma_2$, so that $y\in P_n$.  Since this holds for arbitrarily small $\epsilon$ and $P_n$ is closed, $x\in P_n$.
\end{proof}

We have now arrived at the main goal of this subsection:

\begin{lem}\label{lem:tt_level_separated}
Suppose that $\phi$ is a train track map, that every edge of $V$ is expanding, and that $\mathfrak M$ is irreducible.  Then $\widetilde X$ is level-separated.
\end{lem}

\begin{proof}
Let $\gamma:\mathbf R\rightarrow\widetilde X^1$ be an $M$-deviating geodesic and let $K\geq0$.  By Corollary~\ref{cor:odd_intersection}, there is a regular leaf $\mathcal L$ such that $|\mathcal L\cap\gamma|$ is finite and odd.  Let $C_0=\mathcal L\cap\gamma$ and choose $y\in\mathcal L$ such that $q(c)-q(y)>M+K$ for all $c\in C_0$.  Then for all sufficiently large $n$, there is a level $T^n_o(\tilde\phi^n(y))\subset\mathcal L$ that contains $y$ as one of its leaves and satisfies $T^n_o(\tilde\phi^n(y))\cap\gamma=C_0$.  Hence $\widetilde X$ is level-separated.
\end{proof}

\subsection{Proof of Theorem~\ref{thmi:general}}\label{sec:irreducible_case}

\begin{thm}\label{thm:irreducible}
Let $\phi:V\rightarrow V$ be a train track map of a finite graph $V$.  Suppose that $\phi$ is $\pi_1$-injective and that each edge of $V$ is expanding.  Moreover, suppose that the transition matrix $\mathfrak M$ of $\phi$ is irreducible and that the mapping torus $X$ of $\phi$ has word-hyperbolic fundamental group $G$.  Then $G$ acts freely and cocompactly on a CAT(0) cube complex.
\end{thm}

\begin{proof}
Let $\mathcal Y$ be the forward space arising from the map $\tilde\phi:\widetilde X\rightarrow\widetilde X$.  Since $\phi$ is a train track map, $\widetilde X$ has bounded level intersection by Remark~\ref{rem:bounded_level_intersection} and is level-separated by Lemma~\ref{lem:tt_level_separated}.  By Lemma~\ref{lem:periodic_paths_everywhere}, each finite forward path uniformly fellow-travels with a periodic forward path.  Hence by Proposition~\ref{prop:everything_cut}, it suffices to show that $\widetilde X$ has many effective walls by verifying Conditions~\eqref{item:bust_point} and~\eqref{item:w_a_periodic} of Definition~\ref{defn:many_effective_walls}.

\textbf{Condition~\eqref{item:bust_point}:}  Let $y\in V$ be regular and let $\epsilon>0$.  Let $\mathbf S$ be a finite subtree of $\widetilde V_0$ such that each contractible subspace of $V$ has one or more lifts to $\mathbf S$.

Let $x_0\in V$ be a periodic point in the interior of the edge $e_0$ containing $y$, chosen so that $\dist_{e_0}(x_0,y)<\epsilon$.  This choice is possible since periodic points are dense in each edge of $V$.  Indeed, by irreducibility of $\mathfrak M$, for each edge $e$ of $V$, and each subinterval $d\subset e$, there exists $k>0$ such that the path $\phi^k(d)$ properly contains $e$.  Brouwer's fixed point theorem implies that $d$ contains a $\phi^k$-fixed point.

Let $e_1,\ldots,e_r$ be the edges of $V$, except $e_0$.  Suppose that we have chosen periodic points $\{x_i\in\interior{e_i}\}$ for $0\leq i< s$, for some $s\leq r$, with the property that $\rho(\tilde x_i)\neq\rho(\tilde x_j)$ for $i\neq j$ and any lifts $\tilde x_i,\tilde x_j$ of $x_i,x_j$ to $\mathbf S$.  Let $\tilde e_{i1},\ldots,\tilde e_{ip_i}$ be the lifts of $e_i$ to $\mathbf S$.  Likewise, let $\tilde x_{ij}$ be the lift of $x_i$ to $\tilde e_{ij}$.  Choose $x_{s}\in\interior{e_{s}}$ to be a periodic point with the property that no lift of $x_s$ to $\mathbf S$ lies in $\cup_{i<s,j}\rho^{-1}(\{\rho(\tilde x_{ij})\})$.  Let $\tilde x_{s1},\ldots,\tilde x_{sp_s}$ be the lifts of $x_s$ to $\mathbf S$.

Iterating this procedure, we obtain a set $\{x_0,\ldots,x_r\}$ of periodic points in $V$ such that:

\begin{enumerate}
 \item Each edge $e_i$ of $V$ contains exactly one point $x_i$ in its interior.
 \item The point $x_0$ lies in the interior of the edge $e_0$ containing $y$ and $\dist_{e_0}(x_0,y)<\epsilon$.
 \item Let $\tilde x_{ip}\neq\tilde x_{jq}$ be lifts of $x_i,x_j$ in the closure of a lift of a component of $V-\cup_i\{x_i\}$.  Then $\rho(\tilde x_{ip})\neq\rho(\tilde x_{jq})$.  This holds by construction when $i\neq j$, and holds by Proposition~\ref{prop:R_tree_base case}.(5) when $i=j$ and $p\neq q$.
\end{enumerate}

For each $\tilde x_{ip}$, let $\periodicline_{ip}$ be the bi-infinite periodic forward path containing $\tilde x_{ip}$.  Let $N(\periodicline_{ip})$ be the 1-skeleton of the smallest subcomplex of $\widetilde X$ containing $\periodicline_{ip}$, so that $N(\periodicline_{ip})$ is $\lambda$-quasiconvex in $\widetilde X^1$ by Proposition~\ref{prop:forward_ladder_quasiconvex}.

We now show that for each $R\geq 0$ there exists $B_R$ such that $$\diam(\neb_R(N(\periodicline_{ip}))\cap\neb_R(N(\periodicline_{jq})))\leq B_R$$ whenever $\periodicline_{ip}\neq\periodicline_{jq}$.  Since $\periodicline_{ip},\periodicline_{jq}$ are periodic, they either fellow-travel or have bounded coarse intersection; the following argument precludes the former possibility, whence the claimed $B_R$ exists since there are finitely many pairs $\periodicline_{ip},\periodicline_{jq}$.  Let $\dist_{(ip,jq)}=\dist_{\infty}(\rho(\tilde x_{ip}),\rho(\tilde x_{jq}))$.  By definition of $\dist_{\infty}$, when $\periodicline_{ip}\neq\periodicline_{jq}$, there exists $n^o_{(ip,jq)}>0$ such that for all $n\geq n^o_{(ip,jq)}$ we have $$\dist_{n}(\tilde\phi^{n}(\tilde x_{ip}),\tilde\phi^{n}(\tilde x_{jq}))\geq\frac{\varpi^{n}\dist_{(ip,jq)}}{2}.$$  Let $n_{(ip,jq)}\geq n^o_{(ip,jq)}$ have the property that $\varpi^{n_{(ip,jq)}}\dist_{(ip,jq)}\geq 2R$.  Let $m=\max\{n_{(ip,jq)}\}$.  Then for all $\periodicline_{ip}\neq\periodicline_{jq}$, and all $n\geq m$ we have $$\dist_{n}(\tilde\phi^{n}(\tilde x_{ip}),\tilde\phi^{n}(\tilde x_{jq}))\geq R.$$

We now construct the uniformly sub-quasiconvex spreading set $\mathbb W$.    For $L\geq 1$, let $\epsilon'=\frac{\epsilon}{\varpi^L}$.  By Lemma~\ref{lem:choosing_busts_1}, there exist primary busts $d_i\subset e_i$, each disjoint from its $\phi^L$-preimage, with $d_i\subset\neb_{\epsilon'}(x_i)$.  Let $W\rightarrow X$ be the immersed wall with tunnel-length $L$ and primary busts $d_i$.  Choose $J$ such that $\phi^J(x_s)=x_s$ for all $0\leq s\leq r$.  We choose $\mathbb W$ to be the set of all walls constructed in this way, where $L$ is divisible by $J$.

$\mathbb W$ is uniformly sub-quasiconvex since each component of $V-\cup_i{\interior{e_i}}$ is a finite tree.  Let $T_i,T_j$ be distinct tunnels of $\overline W$ and suppose that $\comapp(T_i),\comapp(T_j)$ intersect a common nucleus approximation $\mathbf N$.  The forward parts of $\comapp(T_i),\comapp(T_j)$ begin at endpoints of primary busts $\tilde d_{ip},\tilde d_{jq}$ which are lifts of primary busts $d_{i},d_{j}$ near the periodic points $x_i,x_j$ respectively.  Let $\tilde x_{ip},\tilde x_{jq}$ be the lifts of $x_i,x_j$ at distance $\epsilon'$ from $\tilde d_{ip},\tilde d_{jq}$.  There are three cases according to whether each of $\comapp(T_i),\comapp(T_j)$ is incoming or outgoing at $\mathbf N$.  In the case where one is incoming and the other outgoing, consideration of the map $q$ shows that the diameter of the intersection of the $R$-neighborhoods of $\comappn{T_i}$ and $\comappn{T_j}$ is bounded by a function of $R$.

Suppose that $\comapp(T_i)$ and $\comapp(T_j)$ are both outgoing from $\mathbf N$.  Our choice of $\epsilon'$ ensures that $\comapp(T_i)$ fellow-travels at distance $\epsilon$ with the forward path of length $L$ emanating from $\tilde x_{ip}$ and similarly for $\comapp(T_j)$ and $\tilde x_{jq}$.  (More precisely, each point of $\comapp(T_i)\cap\widetilde X^1$ is at distance at most $\epsilon$ from the corresponding point of the forward path emanating from $\tilde x_{ip}$.)  Hence the coarse intersection of $\comapp(T_i)$ and $\comapp(T_j)$ is controlled by the function $R\mapsto B_R$ and the uniform constant $\epsilon$.

Suppose that $\comapp(T_i)$ and $\comapp(T_j)$ are both incoming to $\mathbf N$.  By translating, we can assume that $\mathbf N\subset\mathbf S$.  Because $J\mid L$, we have that $\tilde\phi^L(\tilde x_{ip})$ and $\tilde\phi^L(\tilde x_{jq})$ are again lifts of $x_i,x_j$ to $\mathbf N\subset\mathbf S$ and thus lie on the bi-infinite periodic forward paths $\periodicline_{ip},\periodicline_{jq}$ that diverge according to the map $R\mapsto B_R$.  As before, $\comapp(T_i)$ and $\comapp(T_j)$ are (uniformly) coarsely contained in the $\epsilon$-neighborhoods of $\periodicline_{ip}$ and $\periodicline_{jq}$.

\textbf{Condition~\eqref{item:w_a_periodic}:}  Let $a\in\widetilde V_0$ and let its image $\bar a\in V$ be periodic with period $J_a$.  As before, let $\mathbf S$ be a finite subtree of $\widetilde V_0$ containing $a$ and having the property that every contractible subspace of $V$ lifts to $\mathbf S$ and let $e_0,\ldots,e_r$ be the edges of $V$, with $\bar a\in e_0$.  Let $x_{-1}=\bar a$.  We temporarily subdivide $e_0$, writing $e_0=e_{-1}'e_0'$ with $x_{-1}\in e'_{-1}$.  We now apply Lemma~\ref{lem:r_tree_quotient} to $V$, and then remove the subdivision vertex, yielding periodic points $x_i\in\interior{e_i}, 0\leq i\leq r$ so that: for all $i,j\geq -1$ and all $n\geq 0$, any lifts $\tilde x_{ip},\tilde x_{jq}$ of $\phi^n(x_i),\phi^n(x_j)$ to $\mathbf S$ satisfy $\rho(\tilde x_{ip})\neq\rho(\tilde x_{jq})$.  As before, let $J$ be the least common multiple of the periods of the $x_i$.

Let $L\geq 0$. Applying Lemma~\ref{lem:choosing_busts_1}, for each $i\geq 0$ let $d_i\subset\interior{e_i}$ be a primary bust such that $d_i\subset\neb_{\frac{C}{\varpi^L}}(x_i)$ and such that there is an immersed wall $W\rightarrow X$ with tunnel length $L$ and primary busts $d_i$.  The collection $\mathbb W_a$ of such walls with $J\mid L$ is uniformly bust-quasiconvex since each component of the complement of the primary busts is contractible.  Arguing as in the verification of Condition~\eqref{item:bust_point}, the characteristic property of $\{x_i\}$ ensures that $\mathbb W_a$ has uniformly bounded ladder overlap (the bound is independent of $L$).  Likewise, there is a uniform bound $k(a)$ on $3\delta+2\lambda$ fellow-traveling between two forward ladders, one emanating from an endpoint of $\tilde d_{ip}$ and one from $a=\tilde x_{0q}$, whenever $\tilde d_{ip}$ is a lift of some $d_i$ that is joined to $a$ by a path in a knockout of $\overline W$.  Indeed, in this situation, $\tilde\phi^L(a)$ is a lift of $\bar a$ to the finite nucleus approximation containing the lift $\tilde\phi^L(\tilde x_{ip})$ of $x_i$, whence the forward paths emanating from $\tilde d_{ip}$ and $a$ have uniformly bounded coarse intersection.  The other case, where $a$ and $\tilde d_{ip}$ lie on the same nucleus approximation, is handled as in the analogous case in the verification of Condition~\eqref{item:bust_point}.
\end{proof}

\begin{lem}\label{lem:r_tree_quotient}
Let $x_{-1}\in V$ be a periodic point in an edge $e_{-1}$ and let $e_0,\ldots, e_r$ be a collection of edges in $V$.  Then for $0\leq i\leq r$, there exist periodic points $x_i\in\interior{e_i}$ such that for all $i,j\geq -1,n\geq 0$ and for all distinct lifts $\tilde x_{ip},\tilde x_{jq}$ of $\phi^n(x_i),\phi^n(x_j)$ to $\mathbf S$, we have $\rho(\tilde x_{ip})\neq\rho(\tilde x_{jq})$.
\end{lem}

\begin{proof}
For all $i$, any two distinct lifts of $\phi^n(x_i)$ to $\mathbf S$ have distinct images in $\mathcal Y$ by Proposition~\ref{prop:R_tree_base case}.(5).  It therefore suffices to verify the claim of the lemma for points $\tilde x_{ip},\tilde x_{jq}$ with $i\neq j$.

We argue by induction on $r$.  In the base case where $r=-1$, there is nothing to prove.  Supposing that $x_{-1},\ldots,x_{r-1}$ satisfy the conclusion of the lemma, we will choose $x_r$.  Since $\rho$ is an embedding on each edge and $\mathbf S$ is the union of finitely many edges, there exists $K\in\naturals$ such that for all $y\in\rho(\mathbf S)$, we have $|\rho^{-1}(y)\cap\mathbf S|\leq K$.  Let $$Q=|\{\rho(\tilde x_{ip}):-1\leq i\leq r-1,\,1\leq p\leq p_i\}|,$$ where $p_i$ is the number of lifts of $x_i$ to $\mathbf S$.

Choose $m\in\naturals$ such that $e_r$ intersects at least $KQ+1$ $\phi$-orbits of $m$-periodic points.  This choice is possible because, for arbitrarily large $m$, the number of $m$-periodic points in $e_r$ is approximately $C\varpi^m$ for some $C>0$, while the claimed $\phi$-orbits exist as long as there are at least $(KQ+1)m$ periodic points in $e_r$ with period $m$.

For each such $m$-periodic $u$, a \emph{lifted orbit} of $u$ is the set of all lifts to $\mathbf S$ of all points $\phi^k(u)$ with $0\leq k<m$.  Note that if $u,u'$ are $m$-periodic points with distinct $\phi$-orbits, then their lifted orbits are disjoint since their projections to $V$ are distinct $\phi$-orbits of the same cardinality and are hence disjoint.  By the pigeonhole principle, there exists an $m$-periodic point $x_r\in e_r$ with the desired property.  Indeed, the points $\rho(\tilde x_{ip})$ with $i<r$ ruled out at most $KQ$ of the $KQ+1$ lifted orbits.
\end{proof}

\begin{lem}\label{lem:periodic_paths_everywhere}
Let $\phi$ be as in Theorem~\ref{thm:irreducible}.  Then for any finite forward path $\sigma\rightarrow\widetilde X$, there exists a periodic forward path $\chi$ with $\sigma\subset N(\chi)$.  If $\sigma$ is regular, then $N(\sigma)=N(\chi)$.
\end{lem}

\begin{rem}
The period of $\chi$ is unbounded as the length of $\sigma$ increases.
\end{rem}

\begin{proof}
This follows from the fact that periodic points are dense in $V$ and the fact that distinct forward rays diverge at a rate governed by $\varpi$.  Let $x$ be the initial point of $\sigma$ and let $n=|\sigma|$.  We choose $y$ to be a point at distance at most $\frac{\varpi^n}{2K}$ from $x$, with the image of $y$ in $V$ periodic regular, where $K$ is the distance from $x$ to a nearest vertex.  Then the length-$n$ forward path $\sigma_y$ fellow-travels with $\sigma$ at distance $\frac{1}{2K}$, and hence the first and last vertical edges in the carriers of $\sigma,\chi$ are equal.
\end{proof}

We conclude with the following:

\begin{cor}\label{cor:irreducible_case}
Let $\Phi:F\rightarrow F$ be a monomorphism of the finitely generated free group $F$.  Suppose that $\Phi$ is irreducible and that the ascending HNN extension $G=F*_{\Phi}$ is word-hyperbolic.  Then $G$ acts freely and cocompactly on a CAT(0) cube complex.
\end{cor}

\begin{proof}
This follows from the fact that such $\Phi$ is represented by a map $\phi:V\rightarrow V$ satisfying the hypotheses of Theorem~\ref{thm:irreducible}.  Indeed, any irreducible endomorphism has an irreducible train track representative~\cite{BestvinaHandel,Reynolds,DicksVentura}.
\end{proof}

\bibliographystyle{alpha}
\bibliography{MrFreeByZ}

\begin{thebibliography}{GJLL98}

\bibitem[Ago12]{AgolVirtualHaken}
I.~Agol.
\newblock The virtual {H}aken conjecture.
\newblock Preprint, \textsc{arXiv:1204.2810}. With an appendix by Ian Agol,
  Daniel Groves, and Jason Manning, 2012.

\bibitem[BF92]{BestvinaFeighnCombination}
M.~Bestvina and M.~Feighn.
\newblock A combination theorem for negatively-curved groups.
\newblock {\em J. Diff. Geom.}, 35:85--101, 1992.

\bibitem[BFH97]{BestvinaFeighnHandel_lamination}
M.~Bestvina, M.~Feighn, and M.~Handel.
\newblock Laminations, trees, and irreducible automorphisms of free groups.
\newblock {\em Geom. Func. Anal.}, 7(2):215--244, 1997.

\bibitem[BFH00]{BestvinaFeighnHandelTits}
Mladen Bestvina, Mark Feighn, and Michael Handel.
\newblock The {T}its alternative for ${O}ut({F}_n)$ {I}: {D}ynamics of
  exponentially-growing automorphisms.
\newblock {\em Annals of Math.}, 151:517--623, 2000.

\bibitem[BH92]{BestvinaHandel}
Mladen Bestvina and Michael Handel.
\newblock Train tracks and automorphisms of free groups.
\newblock {\em Ann. of Math.}, 135:1--51, 1992.

\bibitem[BH99]{BridsonHaefliger}
Martin~R. Bridson and Andr{\'e} Haefliger.
\newblock {\em Metric spaces of non-positive curvature}.
\newblock Springer-Verlag, Berlin, 1999.

\bibitem[Bri00]{Brinkmann}
P.~Brinkmann.
\newblock Hyperbolic automorphisms of free groups.
\newblock {\em Geom. Func. Anal.}, 10(5):1071--1089, 2000.

\bibitem[BW13]{BergeronWise}
Nicolas Bergeron and Daniel~T. Wise.
\newblock A boundary criterion for cubulation.
\newblock {\em Amer. J. Math.}, 2013.

\bibitem[CLR94]{CooperLongReid}
D.~Cooper, D.~D. Long, and A.~W. Reid.
\newblock Bundles and finite foliations.
\newblock {\em Invent. Math}, 118:255--283, 1994.

\bibitem[Duf12]{DufourThesis}
Guillaume Dufour.
\newblock {\em Cubulations de vari{\'{e}}t{\'{e}}s hyperboliques compactes}.
\newblock PhD thesis, Universit{\'{e}} {P}aris-{S}ud, 2012.

\bibitem[DV96]{DicksVentura}
W.~Dicks and E.~Ventura.
\newblock {\em The Group Fixed by a Family of Injective Endomorphisms of a Free
  Group}.
\newblock Contemporary mathematics. American Mathematical Society, 1996.

\bibitem[Ger94]{Gersten:automorphism}
S.~M. Gersten.
\newblock The automorphism group of a free group is not a {C}{A}{T}(0) group.
\newblock {\em Proc. Amer. Math. Soc.}, 121:pp. 999--1002, 1994.

\bibitem[GJLL98]{GJLL}
Damien Gaboriau, Andre Jaeger, Gilbert Levitt, and Martin Lustig.
\newblock An index for counting fixed points of automorphisms of free groups.
\newblock {\em Duke Mathematical Journal}, 93(3):425--452, 1998.

\bibitem[Gro87]{Gromov87}
M.~Gromov.
\newblock Hyperbolic groups.
\newblock In {\em Essays in group theory}, volume~8 of {\em Math. Sci. Res.
  Inst. Publ.}, pages 75--263. Springer, New York, 1987.

\bibitem[HW]{HruskaWise}
G.~Christopher Hruska and Daniel~T. Wise.
\newblock Finiteness properties of cubulated groups.
\newblock {\em Comp. Math.}
\newblock pp. 1--58, to appear.

\bibitem[HW08]{HaglundWiseSpecial}
Fr{\'e}d{\'e}ric Haglund and Daniel~T. Wise.
\newblock Special cube complexes.
\newblock {\em Geom. Funct. Anal.}, 17(5):1 551--1620, 2008.

\bibitem[HW12]{HsuWiseCubulatingMalnormal}
Tim Hsu and Daniel~T. Wise.
\newblock Cubulating malnormal amalgams.
\newblock 2012.

\bibitem[HW14]{HagenWise:general}
Mark~F Hagen and Daniel~T Wise.
\newblock Cubulating hyperbolic free-by-cyclic groups: the general case.
\newblock {\em arXiv preprint arXiv:1406.3292}, 2014.

\bibitem[Kap00]{Kapovich:2000}
Ilya Kapovich.
\newblock Mapping tori of endomorphisms of free groups.
\newblock {\em Communications in Algebra}, 28(6):2895--2917, 2000.

\bibitem[KL95]{KapovichLeeb_asymptotic_cone}
M.~Kapovich and B.~Leeb.
\newblock On asymptotic cones and quasi-isometry classes of fundamental groups
  of 3-manifolds.
\newblock {\em Geom. Func. Anal.}, 5:582--603, 1995.

\bibitem[Lev09]{Levitt:split}
Gilbert Levitt.
\newblock Counting growth types of automorphisms of free groups.
\newblock {\em Geometric and Functional Analysis}, 19(4):1119--1146, 2009.

\bibitem[Mit99]{MitraQuasiconvex}
Mahan Mitra.
\newblock On a theorem of {S}cott and {S}warup.
\newblock {\em Proc. Amer. Math. Soc.}, 127:1625--1631, 1999.

\bibitem[PW]{PrzytyckiWise:mixed}
Piotr Przytycki and Daniel~T. Wise.
\newblock Mixed 3-manifolds are virtually special.
\newblock pages 1--24.
\newblock Available at arXiv:1205.6742.

\bibitem[Rey10]{Reynolds}
Patrick Reynolds.
\newblock Dynamics of irreducible endomorphisms of ${F}_n$.
\newblock 2010.
\newblock cite arxiv:1008.3659.

\bibitem[Sag95]{Sageev:cubes_95}
Michah Sageev.
\newblock Ends of group pairs and non-positively curved cube complexes.
\newblock {\em Proc. London Math. Soc. (3)}, 71(3):585--617, 1995.

\bibitem[SS90]{Scott_Swarup}
G.~P. Scott and G.~A. Swarup.
\newblock Geometric finiteness of certain {K}leinian groups.
\newblock {\em Proc. Amer. Math. Soc.}, 109(3):765--768, 1990.

\bibitem[Wis]{WiseTubular}
Daniel~T. Wise.
\newblock Cubular tubular groups.
\newblock {\em Trans. Amer. Math. Soc.}
\newblock To appear.

\end{thebibliography}
\end{document}